\documentclass[11pt]{article}

\usepackage{fullpage}

\usepackage{amsmath, amsthm, amssymb, url, color, bbm, graphicx, mathrsfs}

\usepackage[%dvips,
            CJKbookmarks=true, 
            bookmarksnumbered=true,
            bookmarksopen=true,
            colorlinks=true,
            citecolor=red,
            linkcolor=blue,
            anchorcolor=red,
            urlcolor=blue,
            linktoc=page
            ]{hyperref}

\usepackage{tikz}
\usetikzlibrary{arrows}
\tikzstyle{int}=[draw, fill=blue!20, minimum size=2em]
\tikzstyle{dot}=[circle, draw, fill=blue!20, minimum size=2em]
\tikzstyle{init} = [pin edge={to-,thin,black}]

\usepackage{enumerate}
\usepackage{xspace,prettyref}
\newrefformat{eq}{(\ref{#1})}
\newrefformat{thm}{Theorem~\ref{#1}}
\newrefformat{sec}{Section~\ref{#1}}
\newrefformat{fig}{Fig.~\ref{#1}}
\newrefformat{tab}{Table~\ref{#1}}
\newrefformat{rmk}{Remark~\ref{#1}}
\newrefformat{clm}{Claim~\ref{#1}}
\newrefformat{def}{Definition~\ref{#1}}
\newrefformat{cor}{Corollary~\ref{#1}}
\newrefformat{lmm}{Lemma~\ref{#1}}
\newrefformat{prop}{Proposition~\ref{#1}}
\newrefformat{app}{Appendix~\ref{#1}}

\newcommand{\XQP}{X^{\mathsf{c}}}

\newcommand{\tXQP}{\tilde{X}^{\mathsf{c}}}

\newcommand{\stepa}[1]{\overset{\rm (a)}{#1}}
\newcommand{\stepb}[1]{\overset{\rm (b)}{#1}}

\newcommand{\allones}{\mathbf{J}}

\newcommand{\ones}{\mathbf{1}}

\newcommand{\pth}[1]{\left( #1 \right)}
\newcommand{\qth}[1]{\left[ #1 \right]}
\newcommand{\sth}[1]{\left\{ #1 \right\}}

\newcommand{\R}{\mathbb{R}}
\newcommand{\eps}{\varepsilon}

\newcommand{\E}{\mathbb{E}}
\renewcommand{\P}{\mathbb{P}}
\newcommand{\C}{\mathbb{C}}
\newcommand{\1}{\mathbbm{1}}
\renewcommand{\Im}{\operatorname{Im}}
\renewcommand{\Re}{\operatorname{Re}}

\newcommand{\Tr}{\operatorname{Tr}}
\newcommand{\polylog}{\operatorname{polylog}}

\newcommand{\tF}{\tilde{F}}
\newcommand{\tG}{\tilde{G}}

\newtheorem{theorem}{Theorem}[section]
\newtheorem{lemma}[theorem]{Lemma}
\newtheorem{proposition}[theorem]{Proposition}
\newtheorem{corollary}[theorem]{Corollary}
\newtheorem*{theorem*}{Theorem}

\theoremstyle{definition}
\newtheorem{definition}{Definition}[section]
\newtheorem{remark}{Remark}[section]

\newcommand{\reals}{{\mathbb{R}}}
\newcommand{\complex}{{\mathbb{C}}}

\newcommand{\expect}[1]{\mathbb{E}\left[ #1 \right]}
\newcommand{\indc}[1]{{\mathbf{1}_{\left\{{#1}\right\}}}}

\newcommand{\defn}{\triangleq}

\renewcommand{\tilde}{\widetilde}

\newcommand{\bone}{\mathbf{1}}
\newcommand{\iu}{\mathbf{i}}
\newcommand{\coord}{\mathbf{e}}
\newcommand{\bI}{\mathbf{I}}
\newcommand{\bJ}{\mathbf{J}}

\newcommand{\sG}{\mathsf{G}}
\newcommand{\cZ}{\mathcal{Z}}

\newcommand{\graphA}{\mathbf{A}}
\newcommand{\graphB}{\mathbf{B}}
\newcommand{\Bern}{\mathsf{Bern}}
\newcommand{\argmin}{\operatorname*{argmin}}
\newcommand{\argmax}{\operatorname*{argmax}}

\newcommand{\prob}[1]{ \mathbb{P}\left\{ #1 \right\} }
\newcommand{\iprod}[2]{\langle #1, #2 \rangle}
\newcommand{\Iprod}[2]{\langle #1, #2 \rangle}

\newcommand{\ER}{Erd\H{o}s-R\'{e}nyi\xspace}

\newcommand{\fnorm}[1]{\|#1\|_{F}}

\newcommand{\calE}{{\mathcal{E}}}

\newcommand{\termI}{\mathrm{(I)}}
\newcommand{\termII}{\mathrm{(II)}}

\renewcommand{\hat}{\widehat}
\newcommand{\ridge}{\eta^2}
\newcommand{\fS}{\mathfrak{S}}

\title{Spectral Graph Matching and Regularized Quadratic Relaxations II: Erd\H{o}s-R\'enyi Graphs and Universality}

\author{Zhou Fan, Cheng Mao, Yihong Wu, and Jiaming Xu\thanks{
Z.\ Fan, C.\ Mao, and Y.\ Wu are with Department of Statistics and Data Science, Yale University, New Haven, USA, 
\texttt{\{zhou.fan,cheng.mao,yihong.wu\}@yale.edu}.
J.\ Xu is with The Fuqua School of Business, Duke University, Durham, USA, \texttt{jx77@duke.edu}.
Y.~Wu is supported in part by the NSF Grants CCF-1527105, CCF-1900507, an NSF CAREER award CCF-1651588, and an Alfred Sloan fellowship.
J.~Xu is supported by the NSF Grants  IIS-1838124, CCF-1850743, and CCF-1856424.
}}

\begin{document}

\maketitle

\begin{abstract}

We analyze a new spectral graph matching algorithm, GRAph Matching by Pairwise
eigen-Alignments (GRAMPA), for recovering the latent vertex correspondence between
two unlabeled, edge-correlated weighted graphs. Extending the exact recovery guarantees established in the companion paper \cite{FMWX19a}
for Gaussian weights, in this work, we prove the universality of these
guarantees for a general correlated Wigner model. In particular, for two \ER graphs
with edge correlation coefficient $1-\sigma^2$ and average degree at
least $\polylog(n)$, we show that GRAMPA exactly recovers the latent vertex correspondence with high
probability when $\sigma \lesssim 1/\polylog(n)$. Moreover, we establish a similar guarantee for a variant of GRAMPA, corresponding to a tighter quadratic programming relaxation of
the quadratic assignment problem. Our analysis exploits a resolvent
representation of the GRAMPA similarity matrix and local laws for the resolvents
of sparse Wigner matrices.

\end{abstract}

\tableofcontents

\section{Introduction}
Given two (weighted) graphs, graph matching aims at finding a bijection between
the vertex sets that maximizes the total edge weight correlation between the two graphs. 
It reduces to the graph isomorphism problem when the two graphs can be matched perfectly.
Let $A$ and $B$ be the (weighted) adjacency matrices of the two graphs on $n$ vertices.
Then the graph matching problem can be formulated as solving the following
\emph{quadratic assignment problem} (QAP)~\cite{Pardalos94thequadratic,burkard1998quadratic}: 
\begin{align}
\max_{\Pi \in \fS_n} \, \langle A, \Pi B \Pi^\top \rangle , 
%\min_{\Pi \in \fS_n} \; \left\|A- \Pi B\Pi^\top \right\|_F^2, 
\label{eq:QAP}
\end{align}
where  $\fS_n$ denotes the set of permutation matrices in $\R^{n \times n}$ and  $\iprod{\cdot}{\cdot}$ denotes the matrix inner product. The QAP is NP-hard to solve or to approximate within a growing factor \cite{makarychev2010maximum}. 

In the companion paper~\cite{FMWX19a}, we proposed a computationally efficient
spectral graph matching method, called GRAph Matching by Pairwise
eigen-Alignments (GRAMPA). Let us write the spectral decompositions of $A$ and $B$ as 
\begin{align}
A=\sum_i \lambda_i v_iv_i^\top
\quad \text{ and } \quad
B=\sum_j \mu_j w_jw_j^\top. \label{eq:eig-dec}
\end{align}
Given a tuning parameter $\eta>0$,
GRAMPA first constructs an $n\times n$ similarity matrix\footnote{In
\cite{FMWX19a}, $X$ is defined without the factor $\eta$ in the numerator. We
include $\eta$ here for convenience in the proof; this does not affect the
algorithm as the rounded solution $\hat{\Pi}$ is invariant to rescaling $X$.}
\begin{align}
X = \sum_{i, j} \frac{ \eta }{ (\lambda_i - \mu_j)^2 + \eta^2 } v_i v_i^\top \allones  w_j w_j^\top,
\label{eq:specrep}
\end{align}
where $\allones$ is the $n\times n$ all-ones matrix.
Then it outputs a permutation matrix  $\hat{\Pi}$ by
    ``rounding'' $X$ to a permutation matrix,
        for example, by solving the following \emph{linear assignment problem} (LAP)
\begin{equation}
\label{eq:linearassignment}
\hat \Pi \in \argmax_{\Pi \in \fS_n} \; \iprod{X}{\Pi} .
\end{equation}

%\begin{align}
%\max_{\Pi \in \fS_n} \, \langle A , \Pi B \Pi^\top \rangle
%\quad \Longleftrightarrow \quad
%\min_{\Pi \in \fS_n} \| A - \Pi B \Pi^\top \|_F^2 
%\quad \Longleftrightarrow \quad
%\min_{\Pi \in \fS_n} \| A \Pi - \Pi B \|_F^2 . 
%\label{eq:opt}
%\end{align}

Let $\Pi_* \in \fS_n$ be the latent true matching,
and denote the entries of $A$ and $\Pi_*B\Pi_*^\top$ as $a_{ij}$ and
$b_{\pi_*(i)\pi_*(j)}$.
A Gaussian Wigner model is studied in \cite{FMWX19a}, where 
$\{ (a_{ij},b_{\pi_*(i)\pi_*(j)}) \}$ are i.i.d.\  pairs of
correlated Gaussian variables such that $b_{\pi_*(i)\pi_*(j)} =
a_{ij} + \sigma z_{ij}$ for a noise level $\sigma \ge 0$, and $a_{ij}$ and $z_{ij}$ are independent standard Gaussian. 
It is shown that GRAMPA exactly recovers the vertex correspondence 
$\Pi_*$ with high probability when $\sigma = O(1/\log n)$.
Simulation results in~\cite[Section 4.1]{FMWX19a} further show that the empirical performance
of GRAMPA under the Gaussian Wigner model is very similar to that under the \ER model where 
$\{ (a_{ij},b_{\pi_*(i)\pi_*(j)}) \}$ are i.i.d.\ pairs of correlated centered Bernoulli random variables, suggesting that the performance
of GRAMPA enjoys universality. 

In this paper, we prove a universal exact-recovery guarantee for GRAMPA, under a
general Wigner matrix model for the weighted adjacency matrix:
Let $A=(a_{ij})$ be a symmetric random matrix in $\R^{n \times n}$, where the
entries $(a_{ij})_{i \le j}$ are independent. Suppose that
\begin{equation}\label{eq:centerscale}
\E\left[a_{ij} \right]=0 \text{ for all } i,j,
\qquad \E\left[a_{ij}^2 \right]=\frac{1}{n} \text{ for all } i \neq j,
\end{equation}
and
\begin{equation}\label{eq:momentcond}
\E\left[ \left|a_{ij}\right|^k\right] \le \frac{C^k}{n d^{(k-2)/2}} \quad \text{ for all } i,j \text{
and each } k \in [2,(\log n)^{10\log\log n}],
\end{equation}
where $d \equiv d(n)$ is an $n$-dependent sparsity parameter and $C$ is an
absolute positive constant.

%\nb{ZF: Added $k \leq (\log n)^{10 \log \log n}$ throughout;
%otherwise this doesn't hold for sub-Gaussian variables.}

%\begin{remark}\label{remark:ER}

%Two special cases of interest are the following:
Of particular interest are the following special cases:
%dense Wigner matrices with sub-Gaussian
%entries, and centered and scaled adjacency matrices
%of Erd\H{o}s-R\'enyi random graphs.
\begin{itemize}
\item Bounded case: 
%the entries are centered with variance $1/n$ and bounded in magnitude by $\frac{C}{\sqrt{n}}$. Then 
The entries are bounded in magnitude by $\frac{C}{\sqrt{n}}$. Then 
(\ref{eq:momentcond}) is fulfilled for $d=n$ and all $k$.

\item Sub-Gaussian case: 
%the entries are centered with variance $1/n$ and sub-Gaussian norm
The sub-Gaussian norm of each entry satisfies
\begin{equation}\label{eq:subgaussiancondition}
\|a_{ij}\|_{\psi_2} \defn
\sup_{k \geq 1} k^{-1/2}\E\left[ \left|a_{ij} \right|^k\right]^{1/k} =O\left(1/\sqrt{n}\right).
\end{equation}
It is
easily checked that (\ref{eq:momentcond}) is satisfied for $d=n/(\log n)^{11 \log \log n}$ and all large $n$.

\item Erd\H{o}s-R\'enyi graphs with edge probability $p \equiv p(n)$. We may
center and scale the adjacency matrix $A$ such that $a_{ij} \sim (\Bern(p)-p)/\sqrt{np(1-p)}$
for $i\neq j$, which satisfies (\ref{eq:centerscale}) and (\ref{eq:momentcond})
for $d=np(1-p)$ (cf. \prettyref{lmm:AB}).
%For dense Erd\H{o}s-R\'enyi graph with edge
%probability $p \asymp 1$, we may take $d \asymp n$.
%(We may either consider a model 
%where this holds also for $i=j$, or set $a_{ii}=0$.)
%This satisfies (\ref{eq:momentcond}) for $d=np(1-p)$.
\end{itemize}
%Belonging to the intersection of these two cases is a
%dense Erd\H{o}s-R\'enyi graph with edge
%probability $p \asymp 1$, in which case we may take $d \asymp n$.

%\end{remark}

%Now let $A,B \in \R^{n \times n}$ be two random matrices, for which
%there is a latent matching between vertices of $A$ and $B$ so that their
%entries are correlated. To be more precise, let  and let $\Pi_* \in \{0, 1\}^{n \times n}$ be the
%corresponding permutation matrix such that $(\Pi_*^\top B
%\Pi_*)_{ij}=b_{\pi_*(i)\pi_*(j)}$.

%\begin{assumption}\label{assump:main}
With the moment conditions \prettyref{eq:centerscale} and \prettyref{eq:momentcond} specified, we are ready to introduce the correlated Wigner model, which encompasses the correlated \ER graph model proposed in \cite{Pedarsani2011} as a special case.
\begin{definition}[Correlated Wigner model] 
Let $n$ be a positive integer, $\sigma \in [0,1]$ an ($n$-dependent) noise
parameter, $\pi_*$ a latent permutation on $[n]$, and 
$\Pi_* \in \{0, 1\}^{n \times n}$ the corresponding permutation matrix
such that $(\Pi_*)_{i \pi_*(i)} =1$.
Suppose that
$\left\{ (a_{ij},b_{\pi_*(i)\pi_*(j)}):i \leq j\right\}$ are independent pairs of
random variables such that both  $A=(a_{ij})$ and
$B=(b_{ij})$ satisfy (\ref{eq:centerscale}) and
(\ref{eq:momentcond}),
\begin{equation}
\E \left[a_{ij}b_{\pi_*(i)\pi_*(j)} \right] \geq \frac{1-\sigma^2}{n}
\quad \text{ for all } i \ne j, 
\label{eq:correlation}
\end{equation}
and for a constant $C>0$, any $D>0$, and all $n \geq n_0(D)$,
\begin{equation}
\prob{ \left\| A - \Pi_* B \Pi_*^\top \right\| \le C\sigma }
\ge 1- n^{-D}
\label{eq:diffnorm}
\end{equation}
where $\|\cdot\|$ denotes the spectral norm.
\end{definition}
%\end{assumption}
The parameter $\sigma$ measures the effective noise level in the model. In the special case of sparse \ER model, $A$ and $B$ are the centered and normalized adjacency matrices of two \ER graphs, which differ by a fraction $2\sigma^2$ of edges approximately.

In this paper, we prove the following exact recovery guarantee for GRAMPA:

\begin{theorem*}[Informal statement]
    For the correlated Wigner model, if $d \ge \polylog(n)$ and $\sigma \le c
\left(\log n \right)^{-2\kappa}$ for any fixed constant $\kappa>2$ and a
    sufficiently small constant $c > 0$, then GRAMPA with $\eta = 1/ \polylog n$ 
    recovers $\pi_*$ exactly with high probability for large $n$. 
    If furthermore $a_{ij}$ and $b_{ij}$
are sub-Gaussian and satisfy (\ref{eq:subgaussiancondition}),
then this holds with $\kappa=1$.
\end{theorem*}

This theorem generalizes the exact recovery guarantee for GRAMPA
proved in~\cite{FMWX19a} for the Gaussian Wigner model, albeit at the expense of
a slightly stronger requirement for $\sigma$ than in the Gaussian case. 
%The requirement $\sigma \leq 1/\polylog(n)$ is the state-of-the-art for
% polynomial time algorithms on sparse \ER graphs~\cite{DMWX18}, although we note that the
% recent work of \cite{barak2018nearly} provided an algorithm with
% super-polynomial runtime $n^{O(\log n)}$ that achieves exact recovery under the
% weaker condition $\sigma \leq c$, for a small enough constant $c>0$.
The requirement that $d \ge \polylog(n)$ and $\sigma \leq 1/\polylog(n)$ is the state-of-the-art for
polynomial time algorithms on sparse \ER graphs~\cite{DMWX18}, although we note that the
recent work of \cite{barak2018nearly} provided an algorithm with
super-polynomial runtime $n^{O(\log n)}$ that achieves exact recovery when $d \ge n^{o(1)}$
under the much weaker condition of $\sigma \leq 1-\left(\log n\right)^{-o(1)}$.
%, for a small enough constant $c>0$.
%\nb{I moved this paragraph here. This will mark the end of the discussion on the spectral aspect. }

The analysis in~\cite{FMWX19a} relies heavily on the rotational invariance of Gaussian
Wigner matrices, and does not extend to non-Gaussian models. Here,
instead, our universality analysis uses a \emph{resolvent
representation} of the GRAMPA similarity matrix \prettyref{eq:specrep} via a
contour integral (cf.\ Proposition~\ref{prop:res-rep}). 
%and the tighter program \prettyref{eq:simpleQP} 
%a contour integral representation 
%via the resolvents of $A$ and $B$. 
Capitalizing on local laws for the resolvent of sparse Wigner
matrices~\cite{EKYYlocal,EKYYsparse}, we show that the similarity matrix
\prettyref{eq:specrep} 
%and \prettyref{eq:simpleQP}
is with high probability \emph{diagonal dominant} in the sense that 
$\min_k X_{k \pi_*(k) } >\max_{\ell \neq \pi_*(k)} X_{k\ell}$. This enables rounding procedures as simple as thresholding to succeed. 

%\nb{The optimization aspect begins.}

From an optimization point of view, GRAMPA can also be interpreted as solving a \emph{regularized quadratic programming (QP) relaxation} of the QAP. 
More precisely, the QAP \prettyref{eq:QAP} can be equivalently written as 
\begin{align}
\min_{\Pi \in \fS_n} \, \| A \Pi - \Pi B \|_F^2 , 
\label{eq:qap-equiv}
\end{align}
and the similarity matrix $X$ in \prettyref{eq:specrep} is 
a positive scalar multiple of the solution $\widetilde{X}$ to 
\begin{align}
\argmin_{X \in \R^{n \times n}} & ~ 
\| A X - X B \|_F^2 + \ridge \|X\|_F^2 \nonumber \\
\text{s.t.} & ~ \bone^\top X \bone=n .
\label{eq:regQP}
\end{align}
(See \cite[Corollary~2.2]{FMWX19a}.)
This is a convex relaxation of the program \prettyref{eq:qap-equiv} with an additional ridge regularization term. 
% \begin{align}
% \min_{X \in \R^{n \times n}}
% \frac{1}{2} \|AX-XB\|_F^2 + \frac{\ridge }{2} \|X\|_F^2-
% \bone^\top X \bone.
% \label{eq:x-est}
% \end{align}
As a result, our analysis immediately yields the same exact recovery guarantees
for algorithms that round the solution $\tilde{X}$ to~\prettyref{eq:regQP} instead of $X$. 
In Section~\ref{sec:constrained}, we study a tighter relaxation of the QAP~\prettyref{eq:qap-equiv} that imposes row-sum constraints, and establish the same exact recovery guarantees (up to universal constants) by employing similar technical tools.

\paragraph{Organization}

The rest of the paper is organized as follows. 
In Section~\ref{sec:universal}, we state the main exact recovery guarantees for GRAMPA under the correlated Wigner model, as well as the results specialized to the (sparse) \ER model. 
We start the analysis by introducing the key resolvent representation of the GRAMPA similarity matrix in Section~\ref{sec:res}. 
As a preparation for the main proof, Section~\ref{sec:rmt} provides the needed tools from random matrix theory. 
The proof of correctness for GRAMPA is then presented in Section~\ref{sec:pf-main}. 
In Section~\ref{sec:constrained}, we extend the theoretical guarantees to a tighter QP relaxation. 
%examine the convex relaxation interpretation of GRAMPA by studying a variant that corresponds to a tighter regularized QP relaxation. 
Finally, Section~\ref{sec:locallaw} is devoted to proving the resolvent bounds which form the main technical ingredient to our proofs. 

%In particular, let $\XQP$ denote the solution to the above constrained QP and 
%$\PiQP$ denote the projection of $\XQP$ to the set of permutation matrices by 
%solving the LAP. {}

%$S$ is a fixed matrix in $\R^{n \times n}$, both to be chosen.
%\nb{ZF: Perhaps still fix $S \equiv \bJ$ here, and defer the generalization to different $S$ to Section 6.}

\paragraph{Notation} 
Let $[n] \defn \{1, \dots, n\}$. Let $\iu=\sqrt{-1}$.
In a Euclidean space $\R^n$ or $\C^n$,
let $\coord_i$ be the $i$-th standard basis vector,
and let $\bone = \bone_n$ be the all-ones vector.
Let $\bJ = \bJ_n$ denote the $n \times n$ all-ones matrix, and let $\bI = \bI_n$
denote the $n \times n$ identity matrix. The subscripts are often omitted when
there is no ambiguity. 

The inner product of $u,v\in\complex^n$ is defined as $\Iprod{u}{v} = u^*v$.
Similarly, for matrices, $\Iprod{A}{B} = \Tr(A^*B)$.
Let $\|v\| \equiv \|v\|_2=\langle v,v \rangle$ and $\|v\|_\infty=\sup_i
|v_i|$ for vectors. Let $\|M\| \equiv \|M\|_\text{op}=\sup_{v:\|v\|=1} \|Mv\|$,
$\|M\|_F=\langle M,M \rangle$, and $\|M\|_\infty=\sup_{i,j} |M_{ij}|$ for
matrices.

Let $x \land y=\min(x,y)$ and $x \lor y=\max(x,y)$.
We use $C,C',c,c',\ldots$ to denote positive constants that may change from line
to line. For sequences of positive real numbers $(a_n)_{n=1}^\infty$ and $(b_n)_{n=1}^\infty$, we write $a_n \lesssim b_n$ (resp. $a_n \gtrsim b_n$) if there is a constant $C>0$ such that $a_n \le C b_n$ (resp. $b_n \le C a_n$) for all $n \ge 1$, 
$a_n \asymp b_n$ if both relations $a_n \lesssim b_n$ and $a_n \gtrsim b_n$ hold, and $a_n \ll b_n$ if $a_n/b_n
\to 0$ as $n \to \infty$.
We write $a_n=O(b_n)$ if $|a_n| \lesssim b_n$ and $a_n=o(b_n)$ if $|a_n| \ll
b_n$.

\section{Exact recovery guarantees for GRAMPA}
\label{sec:universal}	
In this section, we state the the exact recovery guarantees for GRAMPA, making the earlier informal statement precise.
%We first state our universality result, where we take $S=\bJ_n$.

\begin{theorem} \label{thm:diag-dom}
Fix constants $a>0$ and $\kappa>2$, and let $\eta \in [1/(\log n)^a,1]$.
%Suppose Assumption \ref{assump:main} holds 
Consider the correlated Wigner model
with $n \geq d \geq (\log
n)^{c_0}$ where $c_0>\max(32+4a,4+7a)$.
Then there exist $(a,\kappa)$-dependent constants $C,n_0>0$ and a
deterministic quantity $r(n) \equiv r(n,\eta,d,a)$ satisfying $r(n) \to 0$ as
$n \to \infty$, such that for all $n \geq n_0$,
with probability at least $1-n^{- 10}$,
the matrix $X$ in \prettyref{eq:specrep} satisfies
\begin{align}
 & \max_{\ell \neq \pi_*(k)} \left| X_{k\ell} \right|  \le  C (\log n)^{\kappa } 
 \frac{1}{\sqrt{\eta} },  \nonumber \\
 & \max_k \left| X_{k \pi_*(k) }- \frac{1-\sigma^2}{\eta}\right|
 \le C \left(\frac{r(n)}
{\eta}+\frac{\sigma}{\eta^2} + (\log n)^{\kappa} 
 \frac{1}{\sqrt{\eta} }  \right).
 \label{eq:diagonal-wigner-special}
\end{align}
If there is a universal constant $K$ for which $a_{ij}$ and $b_{ij}$
are sub-Gaussian with $\|a_{ij}\|_{\psi_2},\|b_{ij}\|_{\psi_2} \leq K/\sqrt{n}$,
then the above holds also with $\kappa=1$.
\end{theorem}

As an immediate corollary, we obtain the following exact recovery guarantee for GRAMPA.
\begin{corollary}[Universal graph matching]\label{cor:general}
Under the conditions of Theorem \ref{thm:diag-dom},
there exist constants $c,c' >0$ such that for all $n \ge n_0$, if  
\begin{align}
(\log n)^{-a} \le \eta \le c (\log n)^{-2\kappa } \quad \text{ and }  \quad
\sigma \le c'\eta, \label{eq:assump_sigma_main}
\end{align}
then with probability at least $1-n^{- 10}$, 
%for all sufficiently large $n \geq n_0$, 
\begin{align}
\min_k X_{k \pi_*(k) }>\max_{\ell \neq \pi_*(k)} X_{k\ell}, \label{eq:wigner_separation}
\end{align}
and hence $\hat{\Pi}$ which solves the linear assignment problem
(\ref{eq:linearassignment}) equals $\Pi_*$.
\end{corollary}
\begin{proof}
Let $c=1/(64C^2)$ and $c' = 1/(2C)$, where $C$ is the constant given in \prettyref{thm:diag-dom}.  
Then under assumption \prettyref{eq:assump_sigma_main}, 
we have
$$
C (\log n)^{\kappa}  \sqrt{\eta}  
\le C (\log n)^{\kappa}  \frac{ \sqrt{c}  } { (\log n)^{\kappa} } 
=C \sqrt{c} \le 1/8,
$$
so $\max_{\ell \neq \pi_*(k)} |X_{k\ell}| \leq 1/(8\eta)$.
We also have $C\sigma/\eta \le C c'=1/2$
and $1-\sigma^2>7/8$ and $Cr(n)<1/8$ for all large $n$, so that
$\max_k X_{k\pi_*(k)}>(7/8-1/8-1/2-1/8)/\eta>1/(8\eta)$.
This implies $\prettyref{eq:wigner_separation}$.
\end{proof}

% Choose $a=5$ \nb{YW: Why don't we take $a=2\kappa$? It seems more consistent. Also, less importantly, for the sparse case the assumption on the average degree will be $d\geq (\log n)^{80+\epsilon}$.}
% and $\eta=2C \sigma_0 (\log n)^{-2\kappa}$. \nb{YW: The choice of $\eta$ should be part of the statement of the corollary. Also, should it depend on $D$? Logically speaking, when $\sigma$ is smaller, we should get better probability of error I think $D$ should depend on $\sigma_0$ or the other way around?}
% Then
% under assumption \prettyref{eq:assump_sigma_main}, 
% we have that $C \frac{\sigma}{\eta} \le \frac{1}{2}$,  and that
% $$
% C (\log n)^{\kappa}  \sqrt{\eta}  
% \le C (\log n)^{\kappa}  \frac{ \sqrt{2C \sigma_0} } { (\log n)^{\kappa} }  \le 1/8
% $$
% for a sufficiently small constant $\sigma_0>0$. Thus
% $$
% \left(1-\sigma^2 \right) \frac{1+r(\eta)} {\eta} \ge \frac{7 }{8\eta} 
% > C \left( \frac{\sigma}{\eta^2} + 2 (\log n)^{\kappa} 
%  \frac{1}{\sqrt{\eta} }  \right)
% $$
% which implies the desired $\prettyref{eq:wigner_separation}$.

%\subsection{Sparse Erd\H{o}s-R\'enyi graph matching}

%\nb{ZF: Not sure what to do with this section, maybe move some of this
%to immediately follow Assumption \ref{assump:main}.}

An important application of the above universality result is matching two correlated sparse Erd\H{o}s-R\'enyi graphs. 
Let $G$ be an Erd\H{o}s-R\'enyi graph with $n$ vertices and edge probability $q$, denoted by $G \sim \sG(n, q)$. Let $\graphA$ and $\graphB'$ be two copies of Erd\H{o}s-R\'enyi graphs that are i.i.d.~conditional on $G$, each of which is obtained from $G$ by deleting every edge of $G$ with probability $1-s$ independently where $s \in [0,1]$. Then we have that $\graphA, \graphB' \sim \sG(n, p)$ marginally where $p \defn qs$. 
Equivalently, we may first sample an Erd\H{o}s-R\'enyi graph $\graphA \sim \sG(n, p)$, and then define $\graphB'$ by
$$
\graphB'_{ij} \sim 
\begin{cases}
\Bern(s) & \text{ if } \graphA_{ij} = 1 \\
\Bern\Big(\frac{p(1-s)}{1-p} \Big) & \text{ if } \graphA_{ij} = 0 .
\end{cases}
$$
Suppose that we observe a pair of graphs $\graphA$ and $\graphB  = \Pi_*^\top
\graphB' \Pi_*$, where $\Pi_*$ is an unknown permutation matrix. We then wish to
recover the permutation matrix $\Pi_*$. 

We transform the adjacency matrices $\graphA$
and $\graphB$ so that they satisfy the moment conditions \eqref{eq:centerscale}
and \eqref{eq:momentcond}:
Define the centered, rescaled versions of $\graphA$ and $\graphB$ by
\begin{align}
A \defn (np(1-p))^{-1/2} (\graphA - \E[\graphA])
\quad \text{ and } \quad
B \defn (np(1-p))^{-1/2} (\graphB - \E[\graphB]) .
\label{eq:transformgraph}
\end{align} 
%Moreover, define the effective variance of noise $\sigma^2 \defn \frac{1-s}{1-p}$, and the average degree $d \defn np$. 
Then (\ref{eq:centerscale}) clearly holds, and we check the following
additional properties.
%The proof is deferred to \prettyref{app:proof_lmmAB}.
\begin{lemma} \label{lmm:AB}
For all large $n$, the matrices $A = (a_{ij})$ and $B = (b_{ij})$ satisfy
\eqref{eq:momentcond}, \eqref{eq:correlation}, and \prettyref{eq:diffnorm}
 with $d=np(1-p)$ and
\[\sigma^2=\max\left(\frac{1-s}{1-p},\frac{(\log n)^7}{d}\right).\]
%\nb{YW: The assumptions says $\sigma$ cannot be too small. Something felt wrong here.}
%\nbr{JX. The problem is that if $\sigma$ is too small, then $\|\Pi^\top A \Pi - B \| \gg \sigma$, due to the spurious eigenvalues in sparse random graphs. How about we write: 
%there exists universal constants $c,C>0$ such that 
%\begin{align}
%\prob{ \| \Pi^\top A \Pi - B \| \le C \max\left\{ \sigma, \frac{(\log n)^{6}}{\sqrt{d}}  \right\} }
%\ge 1- e^{-c(\log n)^2}. 
%\label{eq:diffnorm_ER}
%\end{align}}
%\nb{ZF: I changed both (\ref{eq:correlation}) and (\ref{eq:diffnorm})
%to be inequalities and $\sigma^2$ to be the above max.}
%\[\E[a_{ij}]=\E[b_{ij}]=0,\qquad \E[a_{ij}^2]=\E[b_{ij}^2]=1/n,
%\qquad \E[a_{ij}b_{ij}]=(1-\sigma^2)/n,\]
%\[ \text{and } \quad \E[|a_{ij}|^k] = \E[|b_{ij}|^k] < 1/(n d^{\frac{k-2}{2}}) , \]
%for all pairs of indices $i,j \in [n]$ and
%all $k \geq 0$.
%%Moreover, the sub-Gaussian norm of any entry of $A$ or $B$ satisfies 
%%%$$ 
%%%\|a_{ij}\|_{\psi_2}=\sup_{k \geq 1} \frac{1}{\sqrt{k}} \Big( \E[|a_{ij}|^k] \Big)^{1/k} \le \sup_{k \ge 1} \frac{p^{1/k}}{q \sqrt{k}} \le \frac 1q .
%%%$$
%%$$ 
%%\|a_{ij}\|_{\psi_2} \lor \|b_{ij}\|_{\psi_2} \lesssim \frac{1}{ q \sqrt{ \log (1/p) } } .
%%$$
\end{lemma}
\begin{proof}
%\subsection{Proof of \prettyref{lmm:AB}}\label{app:proof_lmmAB}
Assume without loss of generality that $\Pi_*$ is the identity matrix. 
For any $k \geq 2$ we have
$$
\expect{\left|a_{ij}\right|^k} = (np(1-p))^{-k/2} 
\left[ p \left( 1- p \right)^k  +(1-p) p^k \right]
=\frac{  \left( 1- p \right)^{k-1} + p^{k-1} }{ n d^{(k-2)/2 } }
\le \frac{1}{n d^{(k-2)/2 } }.
$$
Thus, the moment condition \eqref{eq:momentcond} is satisfied. 
In addition, we have that for all $i<j$,
\begin{align*}
\expect{a_{ij}b_{ij}} &= \frac{1}{d} \expect{ \left( \graphA_{ij} - p \right) 
\left(\graphB_{ij} - p \right)}
%\\
%&=\frac{1}{d} \left( \expect{ \graphA_{ij} \graphB_{ij}} - p^2 \right)
= \frac{1}{d} \left( p s -p^2  \right) = \frac{s-p}{n(1-p)} \leq 
\frac{1-\sigma^2}{n},
\end{align*}
where the last equality holds by the choice of $\sigma^2$. 
Thus, \eqref{eq:correlation} is satisfied. 
Moreover, 
%for all $i<j,$ 
% $$
% a_{ij}-b_{ij}= \frac{1}{\sqrt{d}}
% \begin{cases}
% 1 & \text{ w.p. } p (1-s) \\
% 0 & \text{ w.p. } 1-2p (1-s) \\
% -1 & \text{ w.p. } p (1-s)
% \end{cases}.
% $$
let 
$
\Delta_{ij} = \frac{1}{\sqrt{2\sigma^2}} \left( a_{ij}-b_{ij} \right).
$
It follows that $\expect{\Delta_{ij}}=0$ and 
$$
\expect{\left|\Delta_{ij}\right|^k}   = \frac{2 p (1-s)}{\left(2\sigma^2d \right)^{k/2}}
\leq \frac{1}{n (2\sigma^2d)^{(k-2)/2}}
$$
where the last inequality is due to $\sigma^2 \geq \frac{1-s}{1-p}$.
Thus, by applying \prettyref{lmm:normbound} and
$2(\log n)^7 \le 2\sigma^2 d \le n$ where the upper bound follows from $p(1-s) \leq s(1-s)
\leq 1/4$, there exists a constant $C>0$ such that for any $D>0$,
with probability at least $1-n^{-D}$ for all $n \geq n_0(D)$,
we have $\|\Delta\| \le C $ and hence $\|A-B\| \le \sqrt{2} C \sigma$.
Thus \prettyref{eq:diffnorm} is satisfied.
\end{proof}

Combining \prettyref{lmm:AB} with \prettyref{cor:general} immediately yields 
a sufficient condition for GRAMPA to exactly recover $\Pi_*$ in the correlated Erd\H{o}s-R\'enyi graph model.
\begin{corollary}[Erd\H{o}s-R\'enyi graph matching]
Suppose that either
\begin{enumerate}[(a)]
\item 
(dense case)
\[\delta \le p \le 1-\delta, \qquad \frac{1-s}{1-p} \leq (\log n)^{-c_1}\]
for constants $\delta \in (0,1)$ and $c_1>4$, or
\item
(sparse case)
\[np(1-p) \ge (\log n)^{c_0}, \qquad \frac{1-s}{1-p} \leq (\log n)^{-c_1}\]
for constants $c_0>48$ and $c_1>8$.
\end{enumerate}
There exist
$(\delta,c_0,c_1)$-dependent constants $a,n_0>0$ such that if
$\eta=(\log n)^{-a}$ and
$n \geq n_0$, then with probability at least $1-n^{-10}$,
\begin{align*}
\min_k X_{k\pi_*(k)}>\max_{\ell \neq \pi_*(k)} X_{k\ell},
\end{align*}
and hence the solution $\hat{\Pi}$ to the linear assignment problem
(\ref{eq:linearassignment}) coincides with $\Pi_*$.
\end{corollary}
\begin{proof}
For (a), pick $\kappa=1$ and any $a$ such that $c_1/2>a>2\kappa=2$. For (b),
pick any $a,\kappa$ such that $c_1/2>a>2\kappa>4$ and $c_0>32+4a>4+7a$.
Then all conditions of \prettyref{thm:diag-dom} and \prettyref{cor:general}
are satisfied for large $n$, and the result follows.
\end{proof}

\section{Resolvent representation} \label{sec:res}
For a real symmetric matrix $A$ with spectral decomposition~\eqref{eq:eig-dec}, its resolvent is defined by
$$
R_A(z) \defn (A-z \bI )^{-1}=\sum_i \frac{1}{\lambda_i-z}v_iv_i^\top 
$$
for $z \in \C \setminus \R$. Then we have the matrix symmetry $R_A(z)^\top=R_A(z)$, conjugate symmetry
$\overline{R_A(z)}=R_A(\bar{z})$, and the following Ward identity. 

\begin{lemma}[Ward identity] \label{lem:ward}
For any $z \in \C \setminus \R$ and any real symmetric matrix $A$,
\[R_A(z)\overline{R_A(z)} =\frac{\Im R_A(z) }{\Im z}.\]
\end{lemma}

\begin{proof}
By the definition of $R(z) \equiv R_A(z)$ and conjugate symmetry, it holds
\begin{align*}
\frac{\Im R(z) }{\Im z} = \frac{ R(z) - \overline{ R(z) } }{ z - \bar{z}} 
= \frac{ (A-z \bI )^{-1} - (A- \bar{z} \bI )^{-1} }{ z - \bar{z}} 
= (A-z \bI )^{-1} (A- \bar{z} \bI )^{-1} 
= R(z) \overline{ R(z) } .
\end{align*}
%\begin{align*}
%R(z)\overline{R(z)} = \Big( \sum_i \frac{1}{\lambda_i-z}v_iv_i^\top  \Big) \Big( \sum_i \frac{1}{\lambda_i-\overline{z}}v_iv_i^\top  \Big) 
%=  \sum_i \frac{1}{ (\lambda_i-z) ( \lambda_i-\overline{z}) }v_iv_i^\top  .
%\end{align*}
%On the other hand, we have
%\begin{align*}
%\frac{\Im R(z) }{\Im z} = \sum_i \frac{\Im (\lambda_i-z)^{-1} }{ \Im z} v_iv_i^\top   
%=  \sum_i \frac{1}{ (\lambda_i-z) ( \lambda_i-\overline{z}) }v_iv_i^\top  ,
%\end{align*}
%which yields the equality.
\end{proof}

The following resolvent representation of $X$ is central to our analysis. 

\begin{proposition} \label{prop:res-rep}
Consider symmetric matrices $A$ and $B$ with spectral decompositions~\eqref{eq:eig-dec}, and suppose that $\|A\| \le 2.5$. Then the matrix $X$ defined in \prettyref{eq:specrep} admits the following representation
\begin{align}
X = \frac{1}{2\pi} \Re \oint_{\Gamma}
R_A(z) \bJ R_B(z+ \iu \eta) dz ,
\label{eq:resrep}
\end{align}
where 
\begin{equation}
\Gamma = \{z: |\Re z| = 3 \text{ and }
|\Im z| \leq \eta/2
\;\;\text{ or } \;\;
|\Im z| = \eta/2 \text{ and } |\Re z| \leq 3\}\label{eq:Gamma}
\end{equation}
is the rectangular contour with vertices $\pm 3 \pm \iu \eta/2$.
\end{proposition}

\begin{proof}
We have 
\begin{align} 
X &= \eta \sum_{i, j} v_i v_i^\top \bJ \frac{ w_j w_j^\top }{ (\lambda_i - \mu_j)^2 + \eta^2 }  \nonumber \\
&= \eta \sum_{i} v_i v_i^\top \bJ R_B( \lambda_i + \iu \eta) R_B( \lambda_i - \iu \eta) \nonumber \\
&= \Im \sum_i v_i v_i^\top \bJ R_B(\lambda_i+ \iu \eta)  
\label{eq:Xkl-0}
\end{align}
by Lemma~\ref{lem:ward}. 
Consider the function $f: \C \to \C^{n \times n}$ defined by $f(z)= \bJ R_B(z+ \iu \eta) $. Then each entry $f_{k\ell}$ is analytic in the region
$\{z:\Im z>-\eta\}$. Since $\Gamma$ encloses each eigenvalue $\lambda_i$ of
$A$, the Cauchy integral formula yields entrywise equality
\begin{align}
-\frac{1}{2\pi \iu }\oint_{\Gamma} \frac{f(z)}{\lambda_i-z}dz=f(\lambda_i) . 
\label{eq:cauchy_integral_formula}
\end{align}
Substituting this into~\eqref{eq:Xkl-0},
we obtain 
\begin{align}
X = \Im \sum_i  v_i v_i^\top \left(-\frac{1}{2\pi \iu }
\oint_{\Gamma} \frac{f(z)}{\lambda_i-z}dz\right)
=\frac{1}{2\pi} \Re \oint_{\Gamma} R_A(z) f(z)  dz ,
\label{eq:contour-rep}
\end{align}
which completes the proof in view of the definition of $f$.
\end{proof}

\section{Tools from random matrix theory}
\label{sec:rmt}

Before proving our main results, we introduce the relevant tools from random matrix theory. In particular, the resolvent bounds in Theorem~\ref{thm:locallawwigner}
constitute an important technical ingredient in our analysis. 
%Throughout this section, we consider a matrix pair $(A,B)$ satisfying Assumption
%\ref{assump:main} where $\pi_*$ is the identity. 

\subsection{Concentration inequalities}

We start with some known concentration inequalities in the literature. 

\begin{lemma}[Norm bounds]\label{lmm:normbound}
For any constant $\eps>0$ and a universal constant $c>0$,
if $n \ge d \ge (\log n)^{6+6\eps}$,
then with probability at least $1-e^{-c(\log n)^{1+\eps}}$,
\[\|A\| \le 2+\frac{(\log n)^{1+\eps} }{ d^{1/4} }.\]
\end{lemma}
\begin{proof}
See \cite[Lemma 4.3]{EKYYsparse},  where we fix the
parameter $\xi=1+\eps$ in \cite[Eq.\ (2.4)]{EKYYsparse}. The notational
identification is $q \equiv \sqrt{d}$.
\end{proof}

\begin{lemma}[Concentration inequalities]
Let $\alpha,\beta \in \R^n$ be independent random vectors with independent
entries, satisfying 
\[\E[\alpha_{i}] = \E[\beta_{i}] =0, \qquad \E[\alpha_{i}^2] = \E[\beta_{i}^2] =
\frac{1}{n},\]
\begin{equation}\label{eq:alphabetacond}
\max(\E[|\alpha_{i}|^k],\E[|\beta_{i}|^k])  \le \frac{1}{n d^{(k-2)/2}} \text{
for each } k \in [2,(\log n)^{10 \log \log n}].
\end{equation}
For any constant $\eps>0$ and universal constants $C,c>0$, if
$n \ge d \ge (\log n)^{6+6\eps}$, then:
\begin{enumerate}[(a)]
\item For each $i \in [n]$, with probability at least $1-e^{-c(\log n)^{1+\eps}}$,
\begin{equation}\label{eq:abound}
|\alpha_i| \le \frac{C}{ \sqrt{d} }.
\end{equation}
\item For any deterministic vector $v \in \C^n$, with probability at least $1-e^{-c(\log n)^{1+\eps}}$,
\begin{equation}\label{eq:linearconcentration}
\left|v^\top \alpha\right| \le (\log n)^{1+\eps} \left(\frac{\|v\|_\infty}{ \sqrt{d} }+\frac{\|v\|_2}{\sqrt{n}}\right).
\end{equation}
Furthermore, for any even integer $p \in [2,(\log n)^{10\log\log n}]$,
\begin{equation}\label{eq:concentrationmoment}
\E\left[\left|v^\top \alpha\right|^p\right]
\le (Cp)^p\left(\frac{\|v\|_\infty}{ \sqrt{d} }
+\frac{\|v\|_2}{\sqrt{n}}\right)^p.
\end{equation}
\item For any deterministic matrix $M \in \C^{n \times n}$, with probability at least $1-e^{-c(\log n)^{1+\eps}}$,
\begin{equation}\label{eq:quad1}
\left|\alpha^\top M\alpha-\frac{1}{n}\Tr M \right| 
\le (\log n)^{2+2\eps} \left(\frac{2\|M\|_\infty}{ \sqrt{d} }
+\frac{\|M\|_F}{n}\right)
\end{equation}
and
\begin{equation}\label{eq:quad2}
\left|\alpha^\top M \beta\right| 
\le (\log n)^{2+2\eps} \left(\frac{2\|M\|_\infty}{ \sqrt{d} }
+\frac{\|M\|_F}{n}\right).
\end{equation}
\end{enumerate}
\end{lemma}
\begin{proof}
See \cite[Lemma 3.7, Lemma 3.8, and Lemma A.1(i)]{EKYYsparse}, where again
we fix $\xi=1+\eps$.
\end{proof}

Next, based on the above lemma, 
we state concentration inequalities for a bilinear form 
that apply to our setting directly. 

\begin{lemma}[Concentration of bilinear form] \label{lmm:bilinear}
Let $\alpha, \beta \in \reals^n$ be random vectors such that the pairs 
$(\alpha_i,\beta_i)$ for $i \in [n]$ are independent, with
\[
\E[\alpha_i]=\E[\beta_i]=0, \qquad \E[\alpha_i^2]=\E[\beta_i^2]=\frac{1}{n},
\qquad \E[\alpha_i\beta_i] \geq \frac{1-\sigma^2}{n}.
\]
Let
$M \in \C^{n \times n}$ be any deterministic matrix. 
\begin{enumerate}[(a)]
\item For any constant $\eps>0$, suppose (\ref{eq:alphabetacond}) holds where
$n \geq d \geq (\log n)^{6+6\eps}$.
Then there are universal constants $C,c>0$ such that 
with probability at least $1-e^{-c (\log n)^{1+\eps}}$,
\begin{align}
\left| \alpha^\top M \beta - \frac{1-\sigma^2}{n} \Tr M \right|\le C  \left(\log n\right)^{2+2\eps} \left( \frac{1}{n} \|M\|_F  + \frac{1}{\sqrt{d}} \|M\|_\infty  \right).
\label{eq:diag}
\end{align}

\item Suppose that $\alpha_i, \beta_i$ are sub-Gaussian with 
$
\|\alpha_i\|_{\psi_2}=\|\beta_i\|_{\psi_2} \leq \frac{K}{\sqrt{n}} 
$
for a constant $K>0$. 
%where $\|\xi\|_{\psi_2}=\sup_{p \geq 1} p^{-1/2} \E[|\xi|^p]^{1/p}$
% is the sub-Gaussian norm. 
Then for any $D>0$, there exists 
a constant $C\equiv C_{K,D}$ only depending on $K$ and $D$ such that with probability 
at least $1-n^{-D}$,
\begin{align}
& \left| \alpha^\top M \beta - \frac{1-\sigma^2}{n} \Tr M \right| \le \frac{C \log n}{n} \|M\|_F. \label{eq:bilinear_wigner}
\end{align}
%\nb{ZF: Can we just put $n$ instead of $d$ above?}\\
%\nb{ZF: I'm wondering if it's worth stating this here. Perhaps we can make a
%brief remark that (\ref{eq:quad1}), (\ref{eq:quad2}), and (\ref{eq:diag}) hold
%with a
%better log factor for the sub-Gaussian case, and put the precise statements of
%both Hanson-Wright and this also in the appendix.}
%\nbr{JX. Sounds good to me. But I feel it a bit cumbersome to make a remark and 
%state it again the appendix.
%Please feel free to go ahead if you know how to execute it.}
\end{enumerate}
\end{lemma}

\begin{proof}
In view of the polarization identity 
\begin{align*}
 \alpha^\top M \beta=\frac{1}{4}(\alpha+\beta)^\top M(\alpha+\beta)
 -\frac{1}{4}(\alpha-\beta)^\top M(\alpha-\beta),
% \label{eq:polorizaton}
\end{align*}
it suffices to analyze the two terms separately. 
Note that 
 \[
 \E\left[(\alpha+\beta)^\top M(\alpha+\beta) \right]=\frac{4-2\sigma^2}{n} \Tr M,
 \qquad \E\left[(\alpha-\beta)^\top M(\alpha-\beta) \right]= \frac{2\sigma^2}{n} \Tr M ,
 \]
which yields the desired expectation $\E[ \alpha^\top M \beta ] = \frac{1-\sigma^2}{n} \Tr M.$
Thus it remains to study the deviation. 

To prove the concentration bound \prettyref{eq:diag}, we obtain from \prettyref{eq:quad1} 
%by applying \cite[Lemma 3.8]{EKYYsparse}, where we fix the parameter $\xi = 1 + \delta/2$ therein. In particular, we apply (3.20) and (3.21) of \cite{EKYYsparse} to obtain 
that, there is a universal constant $c>0$ such that 
with probability at least $1-e^{-c (\log n)^{1+\eps}}$,
$$
\left| (\alpha \pm \beta)^\top  M (\alpha \pm \beta) - 
\E[(\alpha \pm \beta)^\top  M (\alpha \pm \beta)] \right|
\le (\log n )^{2 + 2\eps} \left( \frac{1}{n} \|M\|_F  + \frac{2}{ \sqrt{d} } \|M\|_\infty  \right) , 
$$
%This completes the proof of \prettyref{eq:diag} by following the same argument as the sub-Gaussian case. 
from which \prettyref{eq:diag} easily follows. 

The sub-Gaussian concentration bound  \prettyref{eq:bilinear_wigner} follows from 
%when $a_i, b_i$ have sub-Gaussian norms $K/\sqrt{d}$. 
%In this case, 
the Hanson-Wright inequality \cite{HanWri71, RudVer13}. 
More precisely, note that 
$\max\{\|\alpha+\beta\|_{\psi_2}, \|\alpha-\beta\|_{\psi_2}\} \leq \|\alpha\|_{\psi_2} +  \|\beta\|_{\psi_2} \leq
 2K/\sqrt{d}$, so taking $\delta = n^{-D}/2$ in  \cite[Lemma~A.2]{FMWX19a} yields that with probability at least $1 - n^{-D}$, 
\begin{align*}
\left|(\alpha \pm \beta)^\top M(\alpha \pm \beta)-\expect{ (\alpha \pm \beta)^\top M(\alpha \pm \beta) }\right| \le C_{K, D} \frac{\log n}{n} \|M\|_F , 
\end{align*}
which completes the proof. 
%Combining the last three displayed equations yields that 
% \[
% \prob{\left|a^\top Mb - \frac{1-\sigma^2}{n} \Tr M \right| >t/2} \leq 4\exp\left(-c_K
% \min\left(\frac{t^2 d^2 }{\|M\|_F^2},\frac{td}{ \|M\|}\right)\right).\]
% For any fixed $D>0$, applying this with $t=\|M\|_F \cdot 2C_{K,D} (\log n) /d$ for a 
% sufficiently large constant $C_{K,D}>0$, and applying also $\|M\| \leq \|M\|_F$,
% we get that 
% \[
% \prob{ \left|a^\top M b- \frac{1-\sigma^2}{n}\Tr M \right | >C_{K,D} \frac{\log n}{d} \|M\|_F  } \leq n^{-D}, 
% \]
% which completes the proof of \prettyref{eq:bilinear_wigner}.
\end{proof}

\subsection{The Stieltjes transform}

Denote the semicircle density and its Stieltjes transform by
\begin{equation}\label{eq:stieltjes}
\rho(x)=\frac{1}{2\pi}\sqrt{4-x^2} \, \indc{|x| \leq 2}
\quad \text{ and } \quad 
m_0(z)=\int  \frac{1}{x-z} \rho(x) dx = \frac{-z + \sqrt{z^2 - 4}}{2}
\end{equation}
respectively, where $m_0(z)$ is defined
for $z \notin [-2,2]$, and $\sqrt{z^2 - 4}$ is defined with a branch cut on $[-2,2]$ so that $\sqrt{z^2 - 4} \sim z$ as $|z| \to \infty$.
We have the conjugate symmetry
$\overline{m_0(z)}=m_0(\bar{z})$.

We record the following basic facts about the Stieltjes transform.

\begin{proposition}\label{prop:m0wigner}
For each $z \in \C \setminus \R$, the Stieltjes transform
$m_0(z)$ is the unique value satisfying
\begin{equation}\label{eq:wignerequation}
m_0(z)^2+zm_0(z)+1=0 \quad \text{ and } \quad \Im m_0(z) \cdot \Im z>0.
\end{equation}
Setting $\zeta(z) \defn \min(|\Re z-2|,|\Re z+2|)$, uniformly over
$z \in \C \setminus [-2,2]$ with $|z| \leq 10$,
\begin{equation}
    |m_0(z)| \asymp 1, \;\; |\Im m_0(z)| \gtrsim |\Im z|, \;\; \text{ and }
    \;\;
|\Im m_0(z)| \asymp \begin{cases}
\sqrt{\zeta(z)+|\Im z|}  & \text{ if } |\Re z| \leq 2 , \\
|\Im z|/\sqrt{\zeta(z)+|\Im z|}  & \text{ if } |\Re z|>2.
\end{cases}
\label{eq:Imm0}
\end{equation}
For $x \in [-2,2]$, the continuous extensions
$$
m_0^+(x) \defn \lim_{z \to x:\;z \in \C^+} m_0(z), \quad 
m_0^-(x) \defn \lim_{z \to x:\;z \in \C^-} m_0(z)
$$
from $\C^+$ and $\C^-$ both exist. For all $x \in [-2,2]$, these satisfy
\begin{equation}
m_0^\pm(x)^2+xm_0^\pm(x)+1=0, \;\; m_0^+(x)=\overline{m_0^-(x)}, \;\;
\frac{1}{\pi} \Im m_0^+(x)=-\frac{1}{\pi} \Im m_0^-(x)=\rho(x), \;\;
|m_0^\pm(x)|=1.
\label{eq:Imm0-rho}
\end{equation}
\end{proposition}
\begin{proof}
(\ref{eq:wignerequation}) follows from the definition of $m_0$.
(\ref{eq:Imm0}) follows from \cite[Lemma 4.3]{EKYYlocal} and continuity and
    conjugate symmetry of $m_0$.
    For the existence of $m_0^+$ (and hence also
    $m_0^-$), see e.g.\ the more general statement of
\cite[Corollary 1]{biane1997free}.
    The first claim of (\ref{eq:Imm0-rho}) follows from
continuity and (\ref{eq:wignerequation}), the second from conjugate symmetry,
the third from the
Stieltjes inversion formula, and the last from the fact that the two roots of
(\ref{eq:wignerequation}) at $z=x \in [-2,2]$ are $m_0^+(x)$ and
$m_0^-(x)=\overline{m_0^+(x)}$, so that
$1=m_0^\pm(x)\overline{m_0^\pm(x)}=|m_0^\pm(x)|^2$.
\end{proof}

%\begin{lemma}[Properties of Stieltjes transform] \label{lem:stieltjes}
%For each $z \in \C^+$, $m_0(z)$ satisfies the fixed-point equation
%\[m_0(z)^2+zm_0(z)+1=0.\]
%There exist constants $C,c>0$ such that for all $z \in \C^+$ with $\Re z \in
%[-3,3]$ and $\Im z \in (0,1]$,
%\[c \leq |m_0(z)| \leq C, \qquad \left|\frac{d}{dz}m_0(z)\right|
%\leq \frac{1}{\Im z}.\]
%The function $m_0(z)$ for $z \in \C^+$ admits a continuous extension to $x \in
%\R$. For $x \in \R$, the preceding fixed-point identity and bounds hold also
%for $m_0(x)$, and furthermore
%\[\frac{1}{\pi} \Im m_0(x)=\rho(x).\]
%\end{lemma}
%\nb{to do}

% \begin{lemma}[Lipschitz continuity] \label{lem:lipschitz}
% For all $z \in \C^+$ with $\Re z \in
% \R$ and $\Im z \in (0,1]$,
% \[ \left|\frac{d}{dz}m_0(z)\right|
% \leq \frac{1}{\Im z}.\]
% The function $m_0(\cdot +  \iu \eta )$ is $1/\eta$-Lipschitz continuous. 
% Moreover, for any $x, \lambda \in \reals$, 
% we have that 
% \begin{align}
% \left|  m_0(x+ \iu \eta ) - m_0 ( \lambda +  \iu \eta  ) \right|
% \le   C \frac{\left|\lambda -x \right|}{ | \lambda -x| + \eta}. \label{eq:m_0_diff}
% \end{align}
% \end{lemma}

% \begin{proof}
% Lipschitz continuity follows from the bound $\left|\frac{d}{dz}m_0(z)\right| \leq \frac{1}{\Im z}$. 

% Moreover, if $|\lambda-x| \le \eta$, then inequality~\eqref{eq:m_0_diff} follows from 
% the Lipschitz continuity of $m_0(\cdot +  \iu \eta )$.
% If $|\lambda-x| \ge \eta$, then 
% the inequality follows from 
% the fact that $m_0(z)$ is bounded. 
% \end{proof}

\subsection{Resolvent bounds}

%Let
%\[R(z)=(A-z \bI )^{-1}\]
%be the resolvent of $A$. 
For a fixed constant $a>0$ and all large $n$,
we bound the resolvent $R(z) = R_A(z)$ over the spectral domain
\begin{align*}
D&=D_1 \cup D_2, \text{ where} \\
D_1&=\{z \in \C:\Re z \in [-3,3],\;|\Im z| \in [1/(\log n)^a,1]\}, \text{ and} \\
D_2&=\{z \in \C:|\Re z| \in [2.6,3],\;|\Im z| \leq 1/(\log n)^a\}.
\end{align*}
Here, $D_1$ is the union of two strips in the upper and lower half planes,
and $D_2$ is the union of two strips in the left and right half planes.
%\nb{ZF: Enlarged $D_2$, and changed to $|\Im z| \leq 1$ in $D_1$.}

\begin{theorem}[Resolvent bounds]\label{thm:locallawwigner}
Suppose $A \in \R^{n \times n}$ has independent entries $(a_{ij})_{i \leq j}$
satisfying (\ref{eq:centerscale}) and (\ref{eq:momentcond}).
Fix a constant $a>0$ which defines the domain $D$, fix $\eps>0$, and set
\[b=\max(16+3\eps+2a,3+3\eps+5a/2), \qquad b'=\max(16+4\eps+2a,4+5\eps+6a).\]
Suppose $n \ge d \ge (\log n)^{b'}$.
Then for some constants $C,c,n_0>0$ depending on $a$ and $\eps$, and for
all $n \ge n_0$, with probability
$1-e^{-c(\log n)(\log \log n)}$, the following hold simultaneously
for every $z \in D$:
\begin{enumerate}[(a)]
\item (Entrywise bound) For all $j \ne k \in [n]$,
\begin{equation}\label{eq:wigneroffdiaglaw}
|R_{jk}(z)| \le \frac{C(\log n)^{2+2\eps+a}}{ \sqrt{d} }.
\end{equation}
For all $j \in [n]$,
\begin{equation}\label{eq:wignerdiaglaw}
|R_{jj}(z)-m_0(z)| \le \frac{C(\log n)^{2+2\eps+3a/2}}{ \sqrt{d} }.
\end{equation}
\item (Row sum bound) For all $j \in [n]$,
\begin{equation}\label{eq:wignerrowsumlaw}
\left|\coord_j^\top R(z) \bone \right| \le C(\log n)^{1+\eps+a}.
\end{equation}
\item (Total sum bound)
\begin{equation}
\label{eq:wignertotalsumlaw}
|\bone^\top R(z) \bone -n \cdot m_0(z)| \le \frac{Cn(\log n)^b}{ \sqrt{d} }.
\end{equation}
\end{enumerate}
\end{theorem}

The proof follows ideas of \cite{EKYYsparse}, and we defer
this to Section \ref{sec:locallaw}.
As the spectral parameter $z$ is allowed to converge to the interval $[-2,2]$
with increasing $n$,
this type of result is often called a ``local law'' in the random matrix theory
literature. The focus of the above is a bit different from the results
stated in \cite{EKYYsparse}, as we wish
to obtain explicit logarithmic bounds 
%on the above quantities
for $|\Im z| \asymp 1/\polylog(n)$, rather than bounds for more local
spectral parameters down to the scale of $|\Im z| \asymp \polylog(n)/n$.

% \nb{ZF: Sharpened log factors above to use general $\eps$, rather than
% $\eps \equiv 2$.}

% \nb{ZF: Removed the $-m_0(z)$ term from (\ref{eq:wignerrowsumlaw}), as this is
% of smaller order.}

% \nb{ZF: Let's call these resolvent bounds rather than local laws.
% They are not really local semicircle laws for the eigenvalues; also, the
% imaginary part $\eta=1/(\log n)^a$ is not super local.}

\section{Proof of correctness for GRAMPA}
\label{sec:pf-main}

In this section, we prove Theorem \ref{thm:diag-dom}.
Note that the mapping $B \mapsto \Pi_*^\top B\Pi_*$ for any permutation $\Pi_*$
induces $w_j \mapsto \Pi_*^\top w_j$ and $X \mapsto X\Pi_*$, since
$\bJ\Pi_*^\top=\bJ$. By virtue of this equivariance,
throughout the proof, we may assume without loss of generality that $\Pi_* = \bI$, i.e.
the underlying true permutation $\pi_*$ is the identity permutation.
Then we aim to show that $X$ is diagonally dominant, in the sense
that $\min_{k} X_{kk} > \max_{k\neq \ell} X_{k\ell}$. 

In view of \prettyref{lmm:normbound},
we have that $\|A\| \le 2.5$ holds with probability $1-n^{-D}$ for any $D>0$ and
all $n \geq n_0(D)$. In the following, we assume that $\|A\| \le 2.5$ holds.
On this event,
by \prettyref{prop:res-rep}, we get that 
\begin{align}
X_{k\ell} = \frac{1}{2\pi} \Re \oint_{\Gamma}
(\coord_k^\top R_A(z) \bone )
(\coord_\ell^\top R_B(z+\iu \eta) \bone) dz\label{eq:contourrep-wigner}
\end{align}

Note that one may attempt to directly apply  \prettyref{eq:wignerrowsumlaw} to
bound the row sums $\coord_k^\top R_A(z) \bone$ and $\coord_\ell^\top R_B(z+\iu \eta) \bone$. This would yield
\[\left|(\coord_k^\top R_A(z) \bone)
(\coord_\ell^\top R_B(z+\iu \eta) \bone)\right|
\lesssim   (\log n)^{2+2\eps+2a},\]
and hence $|X_{k\ell}| \lesssim (\log n)^{2+2\eps+2a}$.
However, this estimate is too crude to capture the differences between the diagonal and off-diagonal entries.  In fact, the row sum $\coord_k^\top R_A(z) \bone$ does \emph{not} concentrate on its mean, and the deviation $\coord_k^\top R_A(z) \bone-m_0(z)$ and  $\coord_\ell^\top R_B(z+\iu \eta) \bone-m_0(z)$ is uncorrelated for $k \neq \ell$ and positively correlated for $k=\ell$. For this reason, the diagonal entries of \prettyref{eq:contourrep-wigner} dominate the off-diagonals. Thus it is crucial to gain a better understanding of the deviation terms. We do so by applying Schur complement decomposition.

%we obtain more precise estimates on the product of
%$\coord_k^\top R_A(z) \bone$ and $\coord_\ell^\top R_B(z+\iu \eta) \bone$
%by decomposing them using Schur complements. 

\subsection{Decomposition via Schur complement}

We recall the classical Schur complement identity for
the inverse of a block matrix.
\begin{lemma}[Schur complement identity] \label{lem:schur}
For any invertible matrix $M \in \C^{n \times n}$ and block decomposition
\[M=\begin{bmatrix} A & B \\ C & D \end{bmatrix},\]
if $D$ is square and invertible, then
\begin{align}
M^{-1}=\begin{bmatrix} S & -SBD^{-1} \\ -D^{-1}CS & D^{-1}+D^{-1}CSBD^{-1}
\end{bmatrix}
\label{eq:schur}
\end{align}
where $S=(A-BD^{-1}C)^{-1}$.
\end{lemma}

We decompose $\coord_k^\top R_A(z) \bone$
and $\coord_\ell^\top R_B(z+\iu \eta) \bone$
using this identity, focusing without loss of generality
on $(k,\ell)=(1,2)$. Let $R_{A,12} \in \C^{2 \times 2}$
be the upper-left $2 \times 2$ sub-matrix of $R_A$, and let
$R_A^{(12)} \in \C^{(n-2) \times (n-2)}$
be the resolvent of the $(n-2) \times (n-2)$ minor of $A$ with
the first two rows and columns removed. Let $a_1^\top $ and $a_2^\top $ be the
the first two rows of $A$ with first two entries removed, and let $A_o^\top  \in \R^{2
\times (n-2)}$ be the stacking of $a_1^\top $ and $a_2^\top $. 

The following deterministic lemma approximates $\coord_1^\top R_A(z) \bone$
based on the Schur complement. 
\begin{lemma} \label{lmm:RSRapprox}
Suppose $|z| \leq 10$, and
\begin{align}\label{eq:RSRapprox_assump}
 \left\| R_{A,12}(z)-m_0(z) \bI \right\| \le \delta
\end{align}
where $0 \leq \delta \leq \min_{z:|z| \leq 10} |m_0(z)|/2$.
Then for a constant $C>0$ and $k=1,2$
\begin{align}
\left|\coord_k^\top R_A(z) \bone-m_0(z) 
\left( 1  -  a_k^\top R_A^{(12)}(z) \bone_{n-2}  \right)\right|
\leq C\delta \left( 1 + \|R_A(z) \bone \|_\infty  \right).
\label{eq:a-approx-wigner}
\end{align}
\end{lemma}
\begin{proof}
It suffices to consider $k=1$.
Applying the
Schur complement identity~\eqref{eq:schur}, the first two rows of $R_A$ are given by
\begin{equation}\label{eq:A12}
\begin{bmatrix} 
R_{A,12} & -R_{A,12}A_o^\top R_A^{(12)}
\end{bmatrix}.
\end{equation}
Thus
\begin{align*}
\coord_1^\top R_A(z) \bone
& =\begin{bmatrix}
1 & 0
\end{bmatrix} 
\begin{bmatrix} R_{A,12} & -R_{A,12}A_o^\top R_A^{(12)} 
\end{bmatrix}
\begin{bmatrix}
\bone_2 \\ \bone_{n-2}
\end{bmatrix} \\
&=
\begin{bmatrix}
1 & 0
\end{bmatrix}
R_{A,12}
\left(  \bone_2 -A_o^\top R_A^{(12)} \bone_{n-2} \right).
\end{align*}
Denote $\Delta_{A} \triangleq R_{A,12}(z) - m_0 (z) \bI$.
Then 
\begin{align}
 \coord_1^\top R_A(z) \bone 
&=\begin{bmatrix}
1 & 0
\end{bmatrix}
\left(m_0 (z) \bI + \Delta_{A} \right)
\left(  \bone_2 -A_o^\top R_A^{(12)} \bone_{n-2} \right).
 \nonumber \\
&= m_0(z) \left( 1 - a_1^\top R_A^{(12)} \bone_{n-2}\right)
+\begin{bmatrix}
1 & 0
\end{bmatrix} 
 \Delta_A 
\left(  \bone_2 -A_o^\top R_A^{(12)} \bone_{n-2} \right).
\nonumber \\
&= m_0(z) \left( 1 - a_1^\top R_A^{(12)} \bone_{n-2}\right) + O\left( \delta \left( 1 + \left\| A_o^\top R_A^{(12)} \bone_{n-2} \right\|\right)  \right),
\label{eq:RSR_Y_bound}
\end{align}
where the last equality applies \prettyref{eq:RSRapprox_assump}.
We next upper bound $\left\| A_o^\top R_A^{(12)} \bone_{n-2}\right\|$.
In view of the fact that 
 $C \geq |m_0(z)| \geq c$ for absolute constants $c$ and $C$, 
the assumption \prettyref{eq:RSRapprox_assump} implies that 
$R_{A,12}$ is invertible with $\|R_{A,12}^{-1}\| \lesssim 1$. 
Using \prettyref{eq:A12} again, we have 
\begin{align}
 A_o^\top R_A^{(12)} \bone_{n-2} =  \bone_2 - R_{A,12}^{-1}
 \begin{bmatrix}
 \coord_1 & \coord_2
 \end{bmatrix}^\top
 R_A \bone_n.
 \label{eq:RA12_approx}
\end{align}
It follows that  
\begin{align}
\left\| A_o^\top R_A^{(12)} \bone_{n-2}\right\|
&\lesssim 1  + 
\left| \coord_1^\top R_A \bone_n   \right| +\left| \coord_2^\top R_A \bone_n   \right|    \lesssim 1 + \left\| R_A \bone_n \right\|_\infty.  \label{eq:RA12_approx2}
\end{align}
%It follows that  
%\begin{align}
%\left\| A_o^\top R_A^{(12)} \bone_{n-2}\right\|
%&\le \left\| R^{-1}_{A,12} \right\| \left\| R_{A,12} A_o^\top R_A^{(12)} \bone_{n-2}\right\| \nonumber \\
%& \lesssim \left\| R_{A,12} A_o^\top R_A^{(12)} \bone_{n-2}\right\| \nonumber\\
%& \le \left \| R_{A,12} \bone_2 \right\| + 
%\left| \coord_1 R_A \bone   \right| +\left| \coord_2 R_A \bone   \right|   \nonumber \\
%& \lesssim 1+ \left\| R_A \bone \right\|_\infty  \label{eq:RA12_approx2}
%\end{align}
%where we used again (\ref{eq:RSRapprox_assump}) and $C \geq |m_0(z)| \geq c$
%in the second and the last inequalities. 
The desired bound \prettyref{eq:a-approx-wigner} follows by combining \prettyref{eq:RSR_Y_bound} and \prettyref{eq:RA12_approx2}.
%\nb{ZF: Incorporated $|m_0| \asymp 1$ here.}
\end{proof}

\subsection{Off-diagonal entries}
Without loss of generality, we focus on the off-diagonal entry $X_{12}$:
\[
X_{12}
=\frac{1}{2\pi}
\Re \oint_{\Gamma} \left(\coord_1^\top R_A(z) \bone \right) 
\left(\coord_2^\top R_B(z+ \iu \eta) \bone \right)  dz.
\]

For the given value $a>0$ in Theorem \ref{thm:diag-dom}, and for some small
constant $\eps>0$,
let $b,b'$ be as defined in Theorem \ref{thm:locallawwigner}. Under the given
condition for $c_0$ in Theorem \ref{thm:diag-dom}, for
$\eps>0$ sufficiently small, we have $c_0>b'$ and $c_0>2b$---thus $d \gg (\log
n)^{b'}$ so Theorem \ref{thm:locallawwigner} applies, and also $\sqrt{d} \gg
(\log n)^b$. Fix the constant $\kappa$, where $\kappa=1$ in the
sub-Gaussian case where
$\|a_{ij}\|_{\psi_2},\|b_{ij}\|_{\psi_2} \lesssim 1/\sqrt{n}$, and $\kappa>2$
otherwise. For ease of notation, we define
\begin{equation}
\quad \delta_1 = \frac{\left(\log n \right)^{2+2\eps+3a/2} }{\sqrt{d}},
\quad \delta_2 = \frac{\left(\log n \right)^{1+\eps+a} }{\sqrt{n}},
 \quad \delta_3 =\frac{\left(\log n \right)^b }{\sqrt{d}},
 \quad \delta_4 = \frac{\left(\log n \right)^{\kappa/2} }{\sqrt{n}}.
 %\quad \delta_4 = \frac{\left(\log n \right)^{6} }{\sqrt{d}}, 
 %\quad \delta_5 = \frac{\left(\log n \right)^{3} }{\sqrt{n}}.
\label{eq:deltas}
\end{equation}
Note that we have $\delta_i=o(1)$ for each $i=1,2,3,4$, and also
$\delta_1 \delta_2^2 n =o(1)$.

\subsubsection{Resolvent approximation}
Define an event $\calE_1$ wherein the following hold simultaneously for 
all $z \in \Gamma$:
\begin{align}
\left\| R_{A,12}(z)-m_0(z)\bI \right\|  &   \lesssim \delta_1  \label{eq:RA12_local}\\
\left\| R_{B,12}(z+\iu \eta )-m_0(z+\iu \eta )\bI \right\| & \lesssim \delta_1 \label{eq:RB12_local}\\
\left\|  R_A (z) \bone\right\|_\infty &
\lesssim \delta_2 \sqrt{n}  \label{eq:RAS_local}\\
\left\|  R_B (z+\iu\eta)\bone\right\|_\infty &
\lesssim  \delta_2 \sqrt{n} \label{eq:RBS_local}.
\end{align}
Applying the resolvent approximations
given in \prettyref{thm:locallawwigner}, we have that 
$$
\prob{\calE_1} \ge 1-e^{-c(\log n)(\log \log n)}.
$$
In the following, we assume the event $\calE_1$ holds. 

On $\calE_1$, by \prettyref{lmm:RSRapprox}, we get that uniformly over
$z \in \Gamma$,
\begin{align}
\coord_1^\top R_A(z) \bone  &=  m_0(z) 
\left( 1  -  a_1^\top R_A^{(12)} \bone_{n-2}  \right) + 
O\left(  \delta_1  \delta_2\sqrt{n} \right), \label{eq:e_1R_A1}\\
\coord_2^\top R_B(z+\iu\eta) \bone  &=  m_0(z+\iu\eta) 
\left( 1  -  b_2^\top R_B^{(12)} \bone_{n-2}  \right) + 
O\left(  \delta_1  \delta_2 \sqrt{n} \right) . \label{eq:e_2R_B1}
\end{align}
Each of (\ref{eq:e_1R_A1}) and (\ref{eq:e_2R_B1}) is
itself $O(\delta_2\sqrt{n})$,
by \prettyref{eq:RAS_local} and \prettyref{eq:RBS_local}. Then
multiplying the two, we have
\begin{align*}
& \left[\coord_1^\top R_A(z) \bone \right] 
\left[\coord_2^\top R_B(z+ \iu \eta) \bone \right]  \\
& =m_0(z) m_0(z+ \iu \eta )  \left( 
1 -a_1^\top R_A^{(12)} \bone_{n-2} - b_2^\top R_B^{(12)} \bone_{n-2}
+ a_1^\top R_A^{(12)} \bJ_{n-2} R_B^{(12)} b_2  \right)
+ O\left( \delta_1 \delta_2^2 n \right).
\end{align*}
It follows that 
\begin{align}
&\oint_{\Gamma} \left[\coord_1^\top R_A(z) \bone \right] 
\left[\coord_2^\top R_B(z+ \iu \eta) \bone \right] dz \nonumber \\
& =  \oint_{\Gamma} m_0(z) m_0(z+ \iu \eta ) d z
-  a_1^\top  g -  b_2^\top  h + a_1^\top  M b_2+ 
 O\left(  \delta_1  \delta^2_2  n    \right),  \label{eq:contour-goal1-wigner}
\end{align}
where
\begin{align}
g \triangleq & ~ \oint_{\Gamma} m_0(z)m_0(z+ \iu \eta )R_A^{(12)}(z) \bone_{n-2}    dz , \nonumber \\
h \triangleq & ~  \oint_{\Gamma} m_0(z)m_0(z+ \iu \eta )R_B^{(12)}(z+ \iu \eta ) \bone_{n-2}   dz , \nonumber \\
M \triangleq & ~ \oint_{\Gamma} m_0(z)m_0(z+ \iu \eta )R_A^{(12)}(z) \bJ_{n-2}  R_B^{(12)}(z+ \iu \eta )dz. \label{eq:M-expression-wigner}
\end{align}

\subsubsection{Term-by-term analysis}
\label{sec:offdiag-term}
Next, we bound the individual terms of \prettyref{eq:contour-goal1-wigner}.
By the boundedness of $m_0(z)$, we have
\begin{align}
 \oint_{\Gamma} m_0(z) m_0(z+  \iu \eta )  d z
= O(1). \label{eq:m_0_quadratic}
\end{align}

Define the event $\calE_2$ wherein the following hold simultaneously:
\begin{align}
\left| a_1^\top  g \right| + \left| b_2^\top  h \right| & \lesssim 
\delta_1  \left(\| g \|_\infty + \|h\|_\infty  \right)  + \delta_4
 \left(\| g \|_2 + \|h\|_2  \right)    \label{eq:lin-term-wigner} \\
 \left| a_1^\top M b_2 \right| & \lesssim  
\delta_1 \|M\|_\infty+ \delta_4^2 \|M\|_F . \label{eq:qua-term-wigner}
\end{align}
Note that the triple $(g, h, M)$ is independent of the pair $(a_1, b_2)$
and $a_1$ and $b_2$ are independent. 
Hence, by first conditioning on $(g, h, M)$ and then applying
\prettyref{eq:linearconcentration} and \prettyref{eq:quad2}, we get that 
$$
\prob{\calE_2} \ge 1-n^{-D}
$$
for any constant $D>0$,\footnote{The constant $D$ can be made arbitrarily large by setting the hidden constants in \prettyref{eq:lin-term-wigner} and \prettyref{eq:qua-term-wigner} sufficiently large.} and all $n \geq n_0(D)$,
in both the sub-Gaussian ($\kappa=1$) and general ($\kappa>2$) cases.
Henceforth, we assume $\calE_2$ holds. It then remains to bound the $\ell_2$ and
$\ell_\infty$ norms of $g$, $h$, and $M$.

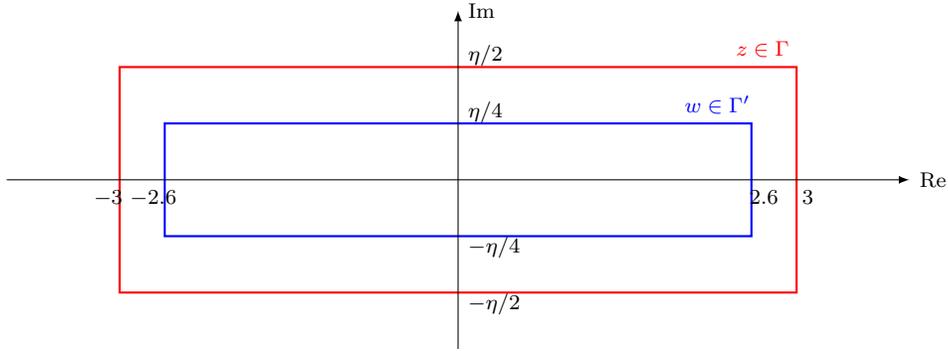
\begin{figure}[ht]%
\centering
\begin{tikzpicture}[scale=1.5,font=\scriptsize,>=latex']
\def\eps{0.1}
\draw[red,thick] (-3,-1) rectangle (3,1);
\draw[blue,thick] (-2.6,-0.5) rectangle (2.6,0.5);
\node[above,red] at (3-3*\eps,1) {$z\in\Gamma$};
\node[above,blue] at (2.6-3*\eps,0.5) {$w\in\Gamma'$};

\node[below] at (-3-\eps,0) {$-3$};
\node[below] at (-2.6-\eps,0) {$-2.6$};
\node[below] at (3+\eps,0) {$3$};
\node[below] at (2.6+1.1*\eps,0) {$2.6$};
\node[right] at (0,1+\eps) {$\eta/2$};
\node[right] at (0,-1-\eps) {$-\eta/2$};
\node[right] at (0,1/2+\eps) {$\eta/4$};
\node[right] at (0,-1/2-\eps) {$-\eta/4$};

%\draw (0.15,0) arc (0:45:0.15);
%\draw[-open triangle 45] (0,0) -- (0.707,0.707);
%\draw[dotted] (0.707,0.707) -- (0.707,0) node[below] {$X$};
%\node at (0.2,0.1) {$\Theta$};
%\draw[->] ([shift=(30:1cm)]2,1) arc (30:60:1cm);
\draw[-latex] (-4,0) -- (4,0) node[right] {$\text{Re}$};
\draw[-latex] (0,-1.5) -- (0,1.5) node[right,black] {$\text{Im}$};
%\draw[blue,fill=blue]  (360/9*\ang:2) circle (0.05);
%\draw[-latex] (-2.2,0) -- (2.2,0) node[right] {$\text{Re}$};
%\draw[-latex] (0,-2.2) -- (0,2.2) node[right] {$\text{Im}$};
\end{tikzpicture} 

\caption{Nested contours $\Gamma$ and $\Gamma'$.}%
\label{fig:contour}%
\end{figure}

Recall that $\Gamma$ is the rectangular contour with vertices $\pm 3\pm \iu \frac{\eta}{2}$.
Let us define another contour (to be used later) $\Gamma'$ inside $\Gamma$, 
with vertices $\pm 2.6 \pm \iu \frac{\eta}{4}$,
cf.~\prettyref{fig:contour}.
Define the event $\calE_3$ wherein the following hold simultaneously for all $z \in \Gamma \cup \Gamma'$:
\begin{align}
\left\| R_A^{(12)}(z) \bone_{n-2}
\right\|_\infty & \lesssim  \delta_2 \sqrt{n}, \label{eq:RAt_2_bound}\\
\left\|R_B^{(12)}(z+\iu \eta ) \bone_{n-2}\right\|_\infty & \lesssim \delta_2 \sqrt{n}, \label{eq:RBs_1_bound}\\
\left| \bone_{n-2}^\top R_A^{(12)}(z) \bone_{n-2} - m_0(z) (n-2) \right|
& \lesssim \delta_3 n, \label{eq:xRA12x_bound} \\
\left| \bone_{n-2}^\top R_B^{(12)}(z+\iu \eta) \bone_{n-2} - m_0(z+\iu \eta) 
(n-2) \right|
& \lesssim \delta_3n.\label{eq:yRB12y_bound}
\end{align}
By \prettyref{thm:locallawwigner}, we have that 
$\prob{\calE_3} \ge 1-e^{-c(\log n)(\log \log n)}$.
In the following, we assume the event $\calE_3$ holds. 

Note that
\begin{equation}
\|g\|_\infty \lesssim \sup_{z \in \Gamma} \|R_A^{(12)}(z) \bone_{n-2} \|_\infty \lesssim \delta_2 \sqrt{n},
\label{eq:ginfty}
\end{equation}
where the second inequality holds in view of \prettyref{eq:RAt_2_bound}.
Similarly, in view of \prettyref{eq:RBs_1_bound}, we have that $
\|h\|_\infty \lesssim  \delta_2 \sqrt{n}$.
Furthermore, 
\begin{equation}
\|M\|_\infty \lesssim 
\sup_{z \in \Gamma}
\left\| R_A^{(12)}(z) \bJ_{n-2}  R_B^{(12)}(z+ \iu \eta )\right\|_\infty
\le \sup_{z \in \Gamma} \left\| R_A^{(12)}(z) \bone_{n-2} \right\|_\infty
\left\| \bone_{n-2}^\top R_B^{(12)}(z+ \iu \eta ) \right\|_\infty
\lesssim  \delta_2^2 n.
\label{eq:Minfty}
\end{equation}
The $\ell_2$ bounds of $g,h$ and $M$ are deferred to \prettyref{lmm:M_norm} below.
Applying  \prettyref{eq:xRA12x_bound},  \prettyref{eq:yRB12y_bound}, and
\prettyref{lmm:M_norm} with $R_A=R_A^{(12)}$ and $R_B=R_B^{(12)}$,
we get 
$\|g\|_2^2 \lesssim  n \log \frac{1}{\eta}$,
$\|h\|_2^2 \lesssim n \log \frac{1}{\eta} $
and
$
\|M\|_F \lesssim  n/\sqrt{\eta} .
$

Combining the above bounds on the norms of $g, h, M$ 
with~\eqref{eq:lin-term-wigner}, \eqref{eq:qua-term-wigner}, 
and \prettyref{eq:m_0_quadratic}, and plugging
into~\eqref{eq:contour-goal1-wigner}, we conclude that on the
event $\{\|A\| \leq 2.5\} \cap \calE_1\cap \calE_2 \cap \calE_3$, 
\begin{align}
|X_{12}| 
= & ~  2\pi \left|\oint_{\Gamma} [\coord_1^\top R_A(z) \bone ][\coord_2^\top R_B(z+ \iu \eta ) \bone ]dz \right| \nonumber \\
\lesssim  & ~ 	1+  \delta_4 \sqrt{n \log \frac{1}{\eta} }
+ \delta_4^2 n 
\frac{1}{\sqrt{\eta} } + \delta_1  \delta^2_2 n \lesssim \delta_4^2 n 
\frac{1}{\sqrt{\eta} } =
\left(\log n\right)^{\kappa} \frac{1}{\sqrt{\eta} },
\label{eq:basicdecomp-offdiag}
\end{align}
where in the third step we used $\delta_1 \delta^2_2 n =o(1)$
and $\eta \le 1$ so that $\delta_4 \sqrt{n} = (\log n)^{\kappa/2} \gtrsim \sqrt{\eta \log\frac{1}{\eta}} + \eta^{1/4}$.

%It remains to establish the claimed bounds on $\|g\|_2$ and $\|M\|_F$ in \prettyref{lmm:M_norm}. 

%In the remainder of the proof, we write for simplicity $R_A^{(12)}$ and
%$R_B^{(12)}$ as $R_A$ and $R_B$,
%and write the eigenvalue-eigenvector pairs of the minors of
%$A$ and $B$ with first two rows and columns removed as $(\lambda_i, v_i)$ and
%$(\mu_j, w_j)$ respectively.

\subsubsection{Bounding the norms of $g, h$ and $M$}

\begin{lemma}\label{lmm:M_norm}
Suppose $\|A\| \le 2.5$ and
$\left| \bone^\top R(z) \bone \right| \lesssim n$ for
all $z \in \Gamma \cup \Gamma'$ and both $R(z)=R_A(z)$ and
$R(z)=R_B(z+\iu \eta)$.
Define
\begin{align*}
g&= \oint_{\Gamma}  m_0(z)m_0(z+ \iu \eta )R_A(z) \bone dz\\
h&= \oint_{\Gamma}  m_0(z)m_0(z+ \iu \eta )R_B(z+\iu \eta) \bone dz\\
M&= \oint_{\Gamma} m_0(z) m_0(z+\iu \eta ) R_A(z) \bJ  R_B (z+ \iu \eta )dz.
\end{align*}
Then $\|g\|^2 \lesssim n \log \frac{1}{\eta}$, $\|h\|^2 \lesssim n \log \frac{1}{\eta}$ and $\|M\|_F^2 \lesssim
\frac{n^2}{\eta}$.
\end{lemma}
\begin{proof}
Since $\|A\| \le 2.5$, the function 
$m_0(z) m_0(z+\iu \eta ) R_A(z) \bone$ is analytic in $z$ in the region between $\Gamma'$ and $\Gamma$. It follows that 
$$
g= \oint_{\Gamma} m_0(z)m_0(z+ \iu \eta )R_A(z) \ones dz = \oint_{\Gamma'}
m_0(w)m_0(w+ \iu \eta )R_A(w) \ones dw.
$$
Thus
\begin{align}
\|g\|^2
\overset{(a)} = & ~ \oint_{\Gamma} dz\oint_{\Gamma'} dw\;m_0(z)m_0(z+ \iu \eta )
m_0(\bar{w})m_0(\bar{w}-\iu \eta ) \ones^\top R_A(\bar{w}) R_A(z)  \ones     \nonumber \\
\overset{(b)} = & ~ -\oint_{\Gamma} dz\oint_{\Gamma'} dw\;m_0(z)m_0(z+ \iu \eta )
m_0(w)m_0(w-\iu \eta ) \ones^\top R_A(w) R_A(z)  \ones     \nonumber \\
\overset{(c)}{=} & ~ -\oint_{\Gamma} dz\oint_{\Gamma'} dw\;m_0(z)m_0(z+ \iu \eta ) m_0(w)m_0(w- \iu \eta ) \ones^\top \frac{R_A(z) - R_A(w)}{z-w}  \ones  \nonumber \\
\overset{(d)}{\lesssim}  & ~ n \oint_{\Gamma} dz \oint_{\Gamma'}
\frac{1}{|z-w|} \label{eq:gnormdeform} 
\end{align}
where (a) applies conjugation symmetry of $m_0$ and $R_A$; (b) changes variables
$w \mapsto \bar{w}$ which reverses the direction of integration along $\Gamma'$;
(c) follows from the identity
\begin{equation}
R_A(z)R_A(w)=(A-z)^{-1}(A-w)^{-1}=\frac{1}{z-w}[(A-z)^{-1}-(A-w)^{-1}]
=\frac{1}{z-w}[R_A(z)-R_A(w)]
\label{eq:RR}
\end{equation}
and (d) holds because $|m_0(z)| \asymp 1$ and 
$\left| \bone^\top R_A(z) \bone \right| \lesssim n$ for all $z \in \Gamma \cup
\Gamma'$ by assumption.
For either $z$ or $w$ in the vertical strips of $\Gamma \cup \Gamma'$ of length
$O(\eta)$, we apply simply $|z-w| \gtrsim \eta$. For both $z$ and $w$ in the
horizontal strips, i.e.\ $|\Im z|=\eta/2$ and $|\Im w|=\eta/4$, we apply
$|z-w| \gtrsim |\Re(z)-\Re(w) |+ \eta$. This gives
\[\|g\|^2 \lesssim
n\left(1+\int_{-3}^{3} dx \int_{-2.6}^{2.6} dy \frac{1}{|x-y|
+\eta} \right) \lesssim n \log \frac{1}{\eta}.\]

For $\|h\|^2$, we have similarly
\begin{align*}
\|h\|^2&=-\oint_\Gamma dz \oint_{\Gamma'} dw\;
m_0(z)m_0(z+\iu \eta) m_0(w)m_0(w-\iu \eta)
\bone^\top \frac{R_B(z+\iu \eta)-R_B(w-\iu \eta)}{(z+\iu \eta)-(w-\iu \eta)}
\bone\\
&\lesssim n \oint_{\Gamma} dz \oint_{\Gamma'}
\frac{1}{|z-w+2\iu \eta|.}
\end{align*}
We may again bound $|z-w+2\iu \eta| \gtrsim \eta$ if either $z$ or $w$ belongs
to a vertical strip, or $|z-w+2\iu \eta| \gtrsim |\Re(z)-\Re(w)|+\eta$
otherwise, to obtain $\|h\|^2 \lesssim n \log (1/\eta)$.
%\nb{ZF: Filled in some details above.}
%\end{proof}
%
%
%\begin{lemma}\label{lmm:M_norm}
%Suppose $\|A\| \le 2.5$ and 
%$\left| \bone^\top R(z) \bone \right| \lesssim n$ for all $z \in \Gamma \cup \Gamma'$
%for both $R(z)=R_A(z)$ and $R(z)=R_B(z+\iu\eta)$.
%Define
%$$
%M= \oint_{\Gamma} m_0(z) m_0(z+\iu \eta ) R_A(z) \bJ  R_B (z+ \iu \eta )dz.
%$$
%Then 
%$
%\|M\|_F^2 \lesssim n^2/\eta.
%$
%\end{lemma}
%\begin{proof}

Finally, we bound $\|M\|_F$. Since $\|A\| \le 2.5$, the function 
$m_0(z) m_0(z+\iu \eta ) R_A(z) \allones R_B (z+ \iu \eta )$ is analytic in $z$
in the region between $\Gamma'$ and $\Gamma$, so
$$
M
= \oint_{\Gamma} m_0(z) m_0(z+\iu \eta ) R_A(z) \allones  R_B (z+ \iu \eta )dz
= \oint_{\Gamma'} m_0(w) m_0(w+\iu \eta ) R_A(w) \allones  R_B (w+ \iu \eta )dw.
$$
Consequently, by the same arguments that leads to \prettyref{eq:gnormdeform},
\begin{align*}
& ~ \fnorm{M}^2 \\
= & ~ \Tr(M^*M) \\
= & ~ \oint_{\Gamma} dz \oint_{\Gamma'} dw\;  m_0(z) m_0(z+\iu \eta ) m_0(\overline{w}) m_0(\overline{w}-\iu \eta ) \Tr \big[ R_A(z) \ones\ones^\top  R_B (z+ \iu \eta ) R_B (\overline w- \iu \eta )\ones\ones^\top R_A(\overline w) \big]	\\
= & ~ -\oint_{\Gamma} dz \oint_{\Gamma'} dw\;  m_0(z) m_0(z+\iu \eta ) m_0(w) m_0(w-\iu \eta ) \ones^\top R_A(w)R_A(z) \ones
\ones^\top  R_B (z+ \iu \eta ) R_B (w- \iu \eta )\ones	\\
= & ~ -\oint_{\Gamma} dz \oint_{\Gamma'} dw\;  m_0(z) m_0(z+\iu \eta ) m_0(w) m_0(w-\iu \eta ) 
\frac{\ones^\top(R_A(z)-R_A(w))\ones}{z-w}
\frac{\ones^\top(R_B (z+ \iu \eta )- R_B (w- \iu \eta ))  \ones}{z+ \iu \eta - (w- \iu \eta)}	\\
\lesssim & ~ n^2 \oint_{\Gamma} dz \oint_{\Gamma'} dw 
\frac{1}{|z-w|} \frac{1 }{|z-w+ 2\iu \eta|}.
\end{align*}
%\nbr{JX. A minor comment. When we switch from $\overline{w}$ to $w$, do we want to point out that the evaluation direction of the contour integral on $\Gamma'$ changes from 
%counterclockwise to clockwise?} \nb{ZF: I added a minus sign here and in the
%preceding lemma.}
If $z$ or $w$ belongs to a vertical strip of $\Gamma \cup \Gamma'$, of
length $O(\eta)$, then $|z-w| \cdot |z-w+2\iu \eta|
\gtrsim \eta^2$; otherwise,
$|z-w| \cdot |z-w+2\iu \eta| \gtrsim (|\Re(z)-\Re(w)|+\eta)^2
\gtrsim (\Re(z)-\Re(w))^2+\eta^2$. Then
\[\|M\|_F^2 \lesssim
n^2\left(\frac{1}{\eta}+
\int_{-3}^{3} dx \int_{-2.6}^{2.6} dy \frac{1}{(x-y)^2 +\eta^2}\right) \lesssim
\frac{n^2}{\eta}.
\]
\end{proof}
%\nb{ZF: Maybe just combine these two lemmas, the argument is very similar.}

\subsection{Diagonal entries}
Without loss of generality, we consider the diagonal entry $X_{11}$:
\[
X_{11}
=\frac{1}{2\pi}\Re \oint_{\Gamma} 
\left[\coord_1^\top R_A(z) \bone\right] \left[ \bone^\top  R_B(z+ \iu \eta ) \coord_1 \right]  dz.\]
By similar arguments as in the off-diagonal entry $X_{12}$ that lead to \prettyref{eq:e_1R_A1}
and \prettyref{eq:e_2R_B1}, we obtain that  for all $z \in \Gamma$, 
\begin{align*}
\coord_1^\top R_A(z) \bone & = m_0(z) \left( 1- a_1^\top R_A^{(1)}(z) \bone_{n-1} \right)+ O\left(\delta_1 \delta_2 \sqrt{n} \right) \\
\coord_1^\top R_B(z+\iu\eta) \bone & = m_0(z+\iu\eta) 
\left( 1- b_1^\top R_B^{(1)}(z) \bone_{n-1} \right)+ O\left(\delta_1 \delta_2 \sqrt{n} \right).
\end{align*}
It follows that 
\begin{align*}
& \left[\coord_1^\top R_A(z) \bone\right] \left[ \bone^\top  R_B(z+ \iu \eta ) \coord_1 \right] \\
& =m_0(z)m_0(z+ \iu \eta ) 
\left( 1 - a_1^\top R_A^{(1)} \bone_{n-1}  - \bone_{n-1}^\top R_B^{(1)} b_1 
+a_1^\top R_A^{(1)} \bJ_{n-1} R_B^{(1)} b_1 \right) 
 +O\left(  \delta_1 \delta_2^2 n \right),
\end{align*}
%}
where respectively, $a_1^\top $ and $b_1^\top $ are the first rows of $A$ and $B$ with first entries
removed; and $R_A^{(1)}$ and $R_B^{(1)}$ are the resolvents of the minors of $A$
and $B$ with first rows and columns removed. Thus, we get that 
\begin{align}
& \oint_{\Gamma} \left[\coord_1^\top R_A(z) \bone\right] \left[ \bone^\top  R_B(z+ \iu \eta ) \coord_1 \right] dz \nonumber \\
& =\oint_{\Gamma} m_0(z) m_0(z+ \iu \eta ) d z
- a_1^\top  g - b_1^\top  h + a_1^\top  M b_1 + O\left(  \delta_1 \delta_2^2 n \right),
\label{eq:contour-goal2}
\end{align}
where
\begin{align*}
g \triangleq & ~ \oint_{\Gamma} m_0(z)m_0(z+ \iu \eta )R_A^{(1)}(z) \bone   dz , \\
h \triangleq & ~ \oint_{\Gamma} m_0(z)m_0(z+ \iu \eta )R_B^{(1)}(z+ \iu \eta )
\bone  dz, \\
M \triangleq & ~ \oint_{\Gamma} m_0(z)m_0(z+ \iu \eta )R_A^{(1)}(z)  \allones R_B^{(1)}(z+ \iu \eta )dz.
\end{align*}

By the same argument as in the off-diagonal entry $X_{12}$, we can control each term above. The only difference is that for the bilinear form, instead of using \prettyref{eq:quad2}, applying Lemma~\ref{lmm:bilinear} 
to control $a_1^\top  M b_1$ gives an extra expectation term $(1-\sigma^2) n^{-1} \Tr M$. Therefore, we obtain that for any fixed constant $D>0$, with probability at least $1-n^{-D}$, for all
sufficiently large $n$, 
\begin{align}
 \left| X_{11} -\frac{1-\sigma^2}{2\pi} \Re  \frac{\Tr M}{n} \right|  \lesssim 
\left(\log n\right)^{\kappa} \frac{1}{\sqrt{\eta}}.
\label{eq:diag-dev}
\end{align}

Denote by $\calE_4$  the event where the following hold simultaneously for
all $z \in \Gamma$:
\begin{align*}
\|A - B\| & \lesssim \sigma \\
\left| \bone_{n-1}^\top R_A^{(1)} (z)  \bone_{n-1}  - m_0(z)n \right| 
& \lesssim \delta_3n \\
\left| \bone_{n-1}^\top R_B^{(1)} (z+\iu\eta) \bone_{n-1}  
- m_0(z+\iu\eta)n \right| 
& \lesssim \delta_3n.
\end{align*}
By the assumption \prettyref{eq:diffnorm} and
\prettyref{thm:locallawwigner}, we have that 
$\prob{\calE_4} \ge 1-n^{-D}$ for any constant $D>0$ and all $n \geq n_0(D)$.

We defer the analysis of $\Tr M$ to \prettyref{lmm:traceM} and
\prettyref{lmm:m_product} below:
Assuming $\calE_4$ holds and applying \prettyref{lmm:traceM} 
and \prettyref{lmm:m_product}
with $R_A, R_B$ replaced by $R_A^{(1)}, R_B^{(1)}$, respectively, we get
\begin{align}
\frac{1}{n} \Re \Tr(M)= \frac{ 2\pi+o_\eta(1)}{\eta}   
+ O\left(   \frac{\sigma}{\eta^2} + \frac{\delta_3}{ \eta} \right). \label{eq:trace_M_desired}
\end{align}
Setting $r(n)=o_\eta(1)+\delta_3$, we get
$$
\left| X_{11} - \frac{1-\sigma^2}{\eta}  \right| 
\lesssim \frac{r(n)}{\eta}+\frac{\sigma}{\eta^2}+
\left(\log n\right)^{\kappa} \frac{1}{\sqrt{\eta}} .
$$

\subsubsection{Analyzing the trace of $M$}

\begin{lemma}\label{lmm:traceM}
Suppose $\|A\| \le 2.5$ and $\|A-B\| \lesssim \sigma$ and 
\begin{align}
\left| \bone^\top R_A(z) \bone - m_0(z) n \right| \lesssim \delta_3 n, \nonumber \\
\left| \bone^\top R_B(z+\iu\eta) \bone - m_0(z+\iu\eta) n \right| \lesssim \delta_3 n, 
\label{eq:traceM_assump}
\end{align}
for all $z \in \Gamma$. Define 
$$
M = \oint_{\Gamma} m_0(z)m_0(z+ \iu \eta )R_A(z)  \bJ R_B (z+ \iu \eta )dz.
$$
Then 
$$
\frac{1}{n} \Tr M = \frac{1}{\iu \eta} \oint_{\Gamma}
m_0(z)m_0(z+\iu \eta)(m_0(z+\iu \eta)-m_0(z))dz + 
O\left(   \frac{\sigma}{\eta^2}+\frac{\delta_3}{\eta}  \right).
$$
%\nb{ZF: Added $\delta_3/\eta$ remainder term here?}
\end{lemma}
\begin{proof}
Applying the identity
\[R_B(z+ \iu \eta )-R_A(z)=(B-(z+ \iu \eta ))^{-1}-(A-z)^{-1}=
R_B(z+ \iu \eta )(A-B+ \iu \eta)R_A(z),\]
we get $R_B(z+ \iu \eta ) R_A(z) = \frac{1}{\iu \eta}(R_B(z+ \iu \eta )-R_A(z) - R_B(z+ \iu \eta )(A-B) R_A(z))$.
Therefore
\begin{align}
\Tr M
&=\oint_{\Gamma} dz\;m_0(z)m_0(z+ \iu \eta ) \Tr \big[ R_A(z)  \bJ R_B(z+ \iu \eta ) \big] \nonumber\\
&=\oint_{\Gamma} dz\;m_0(z)m_0(z+ \iu \eta )   \bone^\top R_B(z+ \iu \eta )  
R_A(z)  \bone \nonumber\\
&=\frac{1}{ \iu \eta }\oint_{\Gamma} dz\;m_0(z)m_0(z+ \iu \eta )
 \bone^\top \left( R_B(z+ \iu \eta )-R_A(z)-R_B(z+ \iu \eta )(A-B)R_A(z) \right) \bone .\label{eq:TrMA-B}
\end{align}

To proceed, we use the following facts. First, it holds that
$$
\left| \bone^\top  R_B(z+ \iu \eta )(A-B)R_A(z) \bone \right|
\le \left\|  \bone^\top R_B(z+ \iu \eta) \right\| \| A- B\| \left\| R_A(z) \bone\right\|.
$$
For $z \in \Gamma$ with $\Im z = \pm \eta/2$, in view of the Ward identity given in Lemma~\ref{lem:ward} and the assumption given in
\prettyref{eq:traceM_assump}, we get that 
$$
\left\| R_A(z) \bone \right\|^2 = \bone^\top R_A(z) \overline{R_A(z)} \bone =
\frac{2}{\eta}  |\Im  \bone^\top R_A(z) \bone| 
\lesssim \frac{n}{\eta}
$$
For $z\in \Gamma$ with $\Re z= \pm 3$, we have that $\left\| R_A(z) \bone \right\|^2
\le n \left\|R_A(z)\right\|^2 \lesssim n$ thanks to the assumption $\|A\| \le 2.5$.  
Similarly, we have $\left\| R_B(z+\iu \eta) \bone \right\|^2 \lesssim n/\eta$.
Combining these bounds with the assumption that $\|A-B\|\lesssim \sigma$ yields that 
$$
\left| \bone^\top  R_B(z+ \iu \eta )(A-B)R_A(z) \bone \right|  \lesssim \frac{n \sigma}{\eta}.
$$
Then applying $|m_0(z)| \asymp 1$ and
\prettyref{eq:traceM_assump}, we obtain
$$
\frac{1}{n} \Tr M
=\frac{1}{\iu \eta} \oint_{\Gamma}
m_0(z)m_0(z+\iu \eta)(m_0(z+\iu \eta)-m_0(z))dz + 
O\left(\frac{\sigma}{\eta^2}+\frac{\delta_3}{\eta}\right).
$$
\end{proof}

\begin{lemma}\label{lmm:m_product}
Let $\Gamma$ be the rectangular contour with vertices $\pm 3 \pm \iu\eta/2$.
Then
$$
\Im \left[ \oint_{\Gamma}
m_0(z)m_0(z+\iu \eta)(m_0(z+\iu\eta)-m_0(z))dz\right]= 2\pi +o_\eta(1) .
$$
\end{lemma}
\begin{proof}
By Proposition \ref{prop:m0wigner}, the integrand is analytic and bounded over
    \[\{z \in \C:\;|z| \leq 9,\,z \notin [-2,2],\,z+\iu \eta \notin [-2,2]\}.\]
Hence we may deform $\Gamma$ to the contour $\Gamma_\epsilon$ with
vertices $\pm (2+\eps) \pm i\eps$, and take $\eps \to 0$ (for fixed $\eta$).
The portion of $\Gamma_\epsilon$ where $|\Re z|>2$ has total length $O(\eps)$, so the
integral over this portion vanishes as $\eps \to 0$.
We may apply the bounded convergence theorem for the remaining two
horizontal strips of $\Gamma_\epsilon$ to get (recall that contour integrals are evaluated counterclockwise):
\begin{align*}
&\oint_\Gamma m_0(z)m_0(z+i\eta)(m_0(z+i\eta)-m_0(z))dz\\
&=\int_2^{-2} m_0^+(x)m_0(x+i\eta)(m_0(x+i\eta)-m_0^+(x))dx
+\int_{-2}^2 m_0^-(x)m_0(x+i\eta)(m_0(x+i\eta)-m_0^-(x))dx,
\end{align*}
where $m_0^+$ and $m_0^-$ are the limits from $\C^+$ and $\C^-$ defined in
Proposition \ref{prop:m0wigner}. Now applying the bounded convergence theorem
again to take $\eta \to 0$, we have $\lim_{\eta \to 0}
m_0(x+i\eta)=m_0^+(x)$ and hence
\begin{align*}
&\lim_{\eta \to 0}
\oint_\Gamma m_0(z)m_0(z+i\eta)(m_0(z+i\eta)-m_0(z))dz\\
&=\int_{-2}^2 m_0^-(x)m_0^+(x)(m_0^+(x)-m_0^-(x))dx
=\int_{-2}^2 |m_0^+(x)|^2 \cdot 2\pi \iu\rho(x)dx=2\pi\iu,
\end{align*}
the last two steps applying (\ref{eq:Imm0-rho}). Thus the imaginary part of the
integral is $2\pi+o_\eta(1)$ for small $\eta$.
\end{proof}

\section{A tighter regularized QP relaxation}
	\label{sec:constrained}
	
%	As mentioned in the companion paper \cite{FMWX19a}, the GRAMPA similarity matrix \prettyref{eq:specrep} originates from a regularized QP relaxation of the quadratic assignment problem.
%	In this section, we analyze a tighter regularized QP relaxation and obtain a similar performance guarantee.

As discussed in the introduction, GRAMPA can be interpreted as solving the regularized QP relaxation~\prettyref{eq:regQP} of the QAP~\prettyref{eq:qap-equiv}. 
We further explore this optimization aspect in this section, and study a tighter regularized QP relaxation. 

Let us begin by recalling the following QP relaxation of the QAP \prettyref{eq:qap-equiv} that replaces the feasible set of permutation matrices by its convex hull, the Birkhoff polytope consisting of all doubly stochastic matrices \cite{zaslavskiy2008path,aflalo2015convex}:
	\begin{align}
\min_{X \in \R^{n \times n}} & ~ 
\| A X - X B \|_F^2 \nonumber \\
\text{s.t.} & ~ X \ones = \ones, \, X^\top \ones =\ones, \, X\geq 0.  \label{eq:fullQP}
\end{align}
%It is shown in \cite[Corollary 2.2]{FMWX19a} that \prettyref{eq:specrep} is equivalent to (i.e., a positive constant multiple of) the solution of
%		\begin{align}
%\min_{X \in \R^{n \times n}} & ~ 
%\| A X - X B \|_F^2 + \eta^2 \|X\|_F^2 \nonumber \\
%\text{s.t.} & ~ \Iprod{X}{\allones} = n  \label{eq:regQP}
%\end{align}
%which is a regularized version of the further relaxation of \prettyref{eq:fullQP}. 
This program differs from the QP relaxation \prettyref{eq:regQP} that underlies GRAMPA in two aspects. 
First, the added ridge penalty $\eta^2 \|X\|_F^2$ in \prettyref{eq:regQP} is crucial for ensuring the desired statistical property of the solution,\footnote{See \cite[Section 1.3]{FMWX19a} for a more detailed discussion in this regard.} while for \prettyref{eq:fullQP} there is no such need for regularization. 
Moreover, the Birkhoff polytope constraint, being the tightest possible convex relaxation, is significantly tighter than the constraint $\bone^\top X \bone = n$. 
Although it is much slower to solve \prettyref{eq:fullQP} than to implement GRAMPA, the doubly stochastic relaxation achieves superior performance over the weaker program \prettyref{eq:regQP} as demonstrated by ample empirical evidence (cf.~\cite{DMWX18,FMWX19a}); nevertheless, a rigorous theoretical understanding is still lacking. 

As a further step toward understanding the relaxations, we analyze the following intermediate program between \prettyref{eq:fullQP} and \prettyref{eq:regQP}:
\begin{align}
%\XQP=
\min_{X \in \R^{n \times n}} & ~ \|AX-XB\|_F^2 + \eta^2 \|X\|_F^2 \nonumber \\
\text{s.t.} & ~  X\ones=\ones , \label{eq:simpleQP}
\end{align}
where we enforce the sum of each row of $X$ to be equal to one.
The above program without the regularization term $\eta^2 \|X\|_F^2$ has been studied in \cite{aflalo2015convex} in a small noise regime. 
%Note that the relaxation \prettyref{eq:simpleQP} is tighter than \prettyref{eq:regQP} which underlies GRAMPA in that all the row sums are enforced to be equal to one. 
As we are analyzing the structure 
%\nb{Before the diag dominance structure was not even mentioned till the end of the section, so the reader will be lost. Now it is fine.} 
of the solution rather than the value of the program, the exact recovery guarantee for GRAMPA (and hence for \prettyref{eq:regQP}) does not automatically carries over to the tighter program \prettyref{eq:simpleQP}. 
Fortunately, we are able to employ similar technical tools to analyze the solution to~\prettyref{eq:simpleQP}, denoted henceforth by $\XQP$. 

The following result is the counterpart of \prettyref{thm:diag-dom} and \prettyref{cor:general}:
\begin{theorem}\label{thm:constrained}
Fix constants $a>0$ and $\kappa>2$, and let $\eta \in [1/(\log n)^a,1]$.
    %Suppose Assumption \ref{assump:main} holds 
    Consider the correlated Wigner model 
    with $n \geq d \geq (\log
    n)^{c_0}$ where $c_0>\max(34+11a,8+12a)$. Then there exist
    $(\alpha,\kappa)$-dependent constants $C,n_0>0$ and a deterministic quantity
    $r(n) \equiv r(n,\eta,d,a)$ satisfying $r(n) \to 0$ as $n \to \infty$, such
    that for all $n \geq n_0$, with probability at least $1-n^{-10}$,
\begin{align}
 & \max_{\pi_*(k) \neq \ell } \left|n \cdot \XQP_{k\ell} \right|  \le  C
    (\log n)^{\kappa}\frac{1}{\sqrt{\eta}},   \label{eq:QPmain1} \\
    & \max_k \left| n \cdot \XQP_{k \pi_*(k) }  - \frac{4(1-\sigma^2)}{\pi\eta}\right| 
    \le C \pth{\frac{r(n)}{\eta}+\frac{\sigma}{\eta^2} + (\log
    n)^{\kappa}\frac{1}{\sqrt{\eta}}}.
 \label{eq:QPmain2}
\end{align}
If $\|a_{ij}\|_{\psi_2},\|b_{ij}\|_{\psi_2} \leq K/\sqrt{n}$, then the above guarantees
hold also for $\kappa=1$, with constants possibly depending on $K$. 

Furthermore, there exist constants $c, c'>0$ such that for all $n\ge n_0$, if 
\begin{align}
\left(\log n \right)^{-a} \le \eta \le c(\log n)^{-2\kappa}
\quad \text{ and } \quad 
\sigma \le c' \eta, \label{eq:assump_sigma_main-QP}
\end{align}
then with probability at least $1-n^{-10}$,
\begin{align}
\min_k X_{k \pi_*(k) }>\max_{\pi_*(k) \neq \ell} X_{k\ell}. \label{eq:diagdom-QP}
\end{align}
\end{theorem}
Compared with \prettyref{cor:general}, the theoretical guarantee for the tighter
program \prettyref{eq:simpleQP} is similar to that for \prettyref{eq:regQP} 
and the GRAMPA method. In practice the performance of the former is slightly better (cf.~\prettyref{fig:exp}).
Furthermore, \prettyref{thm:constrained} applies verbatim to the solution of \prettyref{eq:simpleQP} with column-sum constraints $X^\top \ones=\ones$ instead. This simply follows by replacing $(A,B,X,\Pi_*)$ with $(B,A,X^\top,\Pi_*^\top)$.

\begin{figure}[!ht]
\centering
\includegraphics[clip, trim=4.0cm 8.5cm 4.0cm 9.0cm, width=0.5\textwidth]{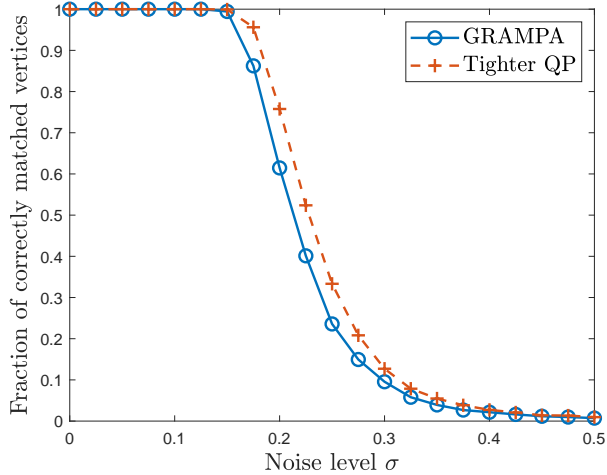}
\caption{Fraction of correctly matched pairs of vertices by GRAMPA and the tighter QP
\prettyref{eq:simpleQP} (both followed by linear assignment rounding) on \ER
graphs with 1000 vertices and edge density 0.5, averaged over 10 repetitions.}
\label{fig:exp}
\end{figure}

\subsection{Structure of solutions to QP relaxations}
	%\label{sec:}

%Next we analyze the behavior of the constrained solution $\XQP$ rigorously. 
Before proving Theorem~\ref{thm:constrained}, we first provide an overview of the structure of solutions to the QP relaxations \eqref{eq:regQP}, \eqref{eq:simpleQP} and \eqref{eq:fullQP}. 
Using the Karush–Kuhn–Tucker (KKT) conditions, the solution of \prettyref{eq:simpleQP} can be expressed as
\begin{align}
\XQP = \sum_{i, j} \frac{\Iprod{v_i}{\mu} \Iprod{w_j}{\ones} }{ (\lambda_i - \mu_j)^2 + \eta^2 } v_i  w_j^\top,
\label{eq:specrep-QP}
\end{align}
where  $\mu\in\reals^n$ is the dual variable corresponding to the row sum constraints, chosen so that $\XQP$ is feasible. 
%Therefore the constrained solution $\XQP$ admits the same spectral representation as in \prettyref{eq:specrep}, where $S=\mu\ones^\top$ as opposed to $S=\ones\ones^\top$.
%Next we solve the dual variable $\mu$.
Since
\[
\XQP\ones = \sum_{i, j} \frac{\Iprod{w_j}{\ones}^2}{ (\lambda_i - \mu_j)^2 + \eta^2 } v_i  v_i^\top  \mu = 
\sth{\sum_{i} \tau_i v_i  v_i^\top}  \mu,
\] 
where
\begin{equation}
\tau_i \triangleq \sum_j \frac{\Iprod{w_j}{\ones}^2}{ (\lambda_i - \mu_j)^2 + \eta^2 }.
\label{eq:rhoi}
\end{equation}
Solving $\XQP\ones=\ones$ yields
\begin{equation}
\mu= \sum_i \frac{\Iprod{v_i}{\ones}}{\tau_i} v_i , 
\label{eq:mu}
\end{equation}
so we obtain 
\begin{equation}
\XQP = \sum_{i, j} \frac{1}{ (\lambda_i - \mu_j)^2 + \eta^2 } \frac{1}{\tau_i} v_iv_i^\top \allones  w_jw_j^\top.
\label{eq:XQP}
\end{equation}

Let us provide some heuristics regarding the solution $\XQP$. As before we can express $\tau_i$ via resolvents as follows:
\begin{align}
\tau_i
= & ~ \frac{1}{\eta} \Im \sum_j \frac{\Iprod{w_j}{\ones}^2}{\mu_j - (\lambda_i+\iu \eta)} =\frac{1}{\eta} \ones^\top \bigg[ \Im \sum_j \frac{1}{\mu_j - (\lambda_i+\iu \eta)} w_jw_j^\top \bigg] \ones \nonumber \\
= & ~  \frac{1}{\eta} \Im[\ones^\top R_B(\lambda_i+\iu \eta)  \ones]. \label{eq:rhoi-res}
\end{align}
Invoking the resolvent bound \prettyref{eq:wignertotalsumlaw}, we expect $\tau_i
\approx \frac{n}{\eta} \Im[m_0(\lambda_i + \iu\eta)]$, where, by properties of the Stieltjes transform (cf.~\prettyref{prop:m0wigner}),  $\Im[m_0(\lambda_i + \iu\eta)] \approx \Im[m_0(\lambda_i)] = \pi \rho(\lambda_i)$ as $\eta \to0$.
Thus we have the approximation
\begin{equation*}
\XQP \approx \frac{1}{\pi n} \sum_{i, j} \frac{\eta}{ (\lambda_i - \mu_j)^2 + \eta^2 } \frac{1}{\rho(\lambda_i)} v_iv_i^\top \allones  v_jw_j^\top,
%\label{eq:}
\end{equation*}
Compared with the unconstrained solution \prettyref{eq:specrep}, apart from normalization, the only difference is the extra spectral weight $\frac{1}{\rho(\lambda_i)}$ according to the inverse semicircle density. The effect is that eigenvalues near the edge are upweighted while eigenvalues in the bulk are downweighted, the rationale being that eigenvectors corresponding to the extreme eigenvalues are more robust to noise perturbation.

\begin{remark}[Structure of the QP solutions]
\label{rmk:structureQP}	
Let us point out that solution of various QP relaxations, including \prettyref{eq:fullQP}, \prettyref{eq:simpleQP}, and \prettyref{eq:regQP}, are of the following common form:
\begin{equation}
X = \sum_{i, j} \frac{ \eta }{ (\lambda_i - \mu_j)^2 + \eta^2 } v_i v_i^\top S  w_j w_j^\top,
\label{eq:commonstructure}
\end{equation}
where $S$ is an $n\times n$ matrix that can depend on $A$ and $B$. Specifically,
%in the order of increasingly tight relaxations, 
 from the loosest to the tightest relaxations, 
we have:
\begin{itemize}
	\item For \prettyref{eq:regQP} with the total sum constraint, $S=\alpha \allones$, where the dual variable $\alpha>0$ is chosen for feasility. Since scaling by $\alpha$ does not effect the subsequent rounding step, this is equivalent to \prettyref{eq:specrep} that we analyze.
	
	\item For \prettyref{eq:simpleQP} with the row sum constraint, $S=\mu\ones^\top$ is rank-one with $\mu$ given in \prettyref{eq:mu}.
	
	\item For \prettyref{eq:fullQP} without the positivity constraint, $S=\mu\ones^\top+\ones\nu^\top$ is rank-two. Unfortunately, the dual variables and the spectral structure of the optimal solution turn out to be difficult to analyze.
	
	\item For \prettyref{eq:fullQP} with the positivity constraint, $S=\mu\ones^\top+\ones\nu^\top+H$, where $H\geq 0$ is the dual variable certifying the positivity of the solution and satisfies complementary slackness.	
\end{itemize}

\end{remark}

\subsection{Proof of Theorem~\ref{thm:constrained}}

We now apply the resolvent technique to analyze the behavior of the constrained solution $\XQP$ and establish its diagonal dominance. 

\subsubsection{Resolvent representation of the solution}

We start by giving a resolvent representation of $\XQP$ via a contour integral. 

\begin{lemma} \label{prop:res-rep-c}
Consider symmetric matrices $A$ and $B$ with the spectral decompositions~\eqref{eq:eig-dec}, and suppose that $\|A\| \le 2.5$. Then the solution $\XQP$ of the program \eqref{eq:simpleQP} admits the following representation
\begin{align}
\XQP = 
\frac{1}{2\pi} \Re  \oint_\Gamma F(z) R_A(z) \allones  R_B(z+\iu \eta), \label{eq:XQP-contour}
\end{align}
where $\Gamma$ is defined by \eqref{eq:Gamma} and
\begin{equation}
F(z) \defn \frac{2\iu }{\ones^\top R_B(z+\iu \eta) \ones - \ones^\top R_B(z- \iu \eta) \ones}.
\label{eq:Fz}
\end{equation} 
\end{lemma}

\begin{proof}
By \prettyref{eq:rhoi-res} we have $\tau_i^{-1} = \eta F(\lambda_i)$. This leads to the following contour representation of $\XQP$ analogous to \prettyref{eq:resrep} for the unconstrained solution:
\begin{align}
\XQP
%= & ~  \sum_{i, j} \frac{F(\lambda_i)}{ (\lambda_i - \mu_j)^2 + \eta^2 }  v_iv_i^\top \allones  v_jw_j^\top,\\
= & ~  \eta \sum_{i} F(\lambda_i)v_iv_i^\top \allones \sth{ \sum_j \frac{1}{ (\lambda_i - \mu_j)^2 + \eta^2 }  w_jw_j^\top} \nonumber \\
\stepa{=} & ~ \Im \qth{ \sum_{i} F(\lambda_i)v_iv_i^\top \allones  R_B(\lambda_i+\iu \eta)} \nonumber \\
\stepb{=} & ~ \Im \qth{ \frac{1}{-2\pi \iu} \oint_\Gamma F(z) R_A(z) \allones  R_B(z+\iu \eta)} \nonumber \\
= & ~ \frac{1}{2\pi} \Re  \oint_\Gamma F(z) R_A(z) \allones  R_B(z+\iu \eta),  \notag
\end{align}
where (a) follows from the Ward identity (\prettyref{lem:ward}); (b) follows from Cauchy integral formula and the analyticity of 
$F$ in the region enclosed by the contour $\Gamma$. 
\end{proof}

\subsubsection{Entrywise approximation} 

For some small constant $\eps>0$, let $b,b'$ be as defined in Theorem
\ref{thm:locallawwigner}.
Under the assumptions of \prettyref{thm:constrained}, we have $c_0>b'$
for $\eps$ sufficiently small, so that Theorem \ref{thm:locallawwigner} applies.
Recall the notation $\delta_1,\ldots,\delta_4$ defined in \prettyref{eq:deltas}.
For sufficiently small $\eps>0$, we may also verify under the
assumptions of \prettyref{thm:constrained} that
$\delta_i=o(1)$ for each $i=1,2,3,4$, and
\begin{equation}
    \frac{\delta_1\delta_2^2n}{\eta} \le 1, \quad
    \frac{\delta_2^2 \delta_3n}{\eta^2} \le \frac{(\log n)^\kappa}{\sqrt{\eta}}, \quad \text{ and } \quad
\delta_3 \le \eta^3. \label{eq:QP-assump3}
\end{equation}
We also assume throughout the proof
that the high-probability event $\|A\| \leq 2.5$ holds.

Thanks to (\ref{eq:wignertotalsumlaw}), we can approximate $F(z)$ by
\begin{equation}
\tF(z) = \frac{1}{n} \frac{2\iu}{m_0(z+\iu \eta) -m_0(z-\iu \eta)}
\label{eq:tFz}
\end{equation}
and approximate $\XQP$ by
\begin{align}
\tXQP
= & ~ \frac{1}{2\pi} \Re  \oint_\Gamma \tF(z) R_A(z) \allones  R_B(z+\iu \eta) \label{eq:tXQP-contour} \\ 
= & ~ \frac{-1}{\pi n} \Im \oint_\Gamma  \frac{1}{m_0(z+\iu \eta)-m_0(z-\iu \eta)} R_A(z)\allones R_B(z+\iu \eta). \label{eq:tXQP}
\end{align}
The following lemma makes the approximation of $\XQP$ precise in the entrywise sense:
\begin{lemma}
\label{lmm:tXX}	
	%With probability at least $1-e^{-c(\log n)(\log \log n)}$, 
Suppose \prettyref{eq:QP-assump3} holds. 
On the high-probability event where \prettyref{thm:locallawwigner} holds and
also $\|A\| \leq 2.5$,
\begin{equation}
\|\tXQP-\XQP\|_{\ell_\infty} \lesssim \frac{\delta_2^2\delta_3}{\eta^2}  \le
    \frac{(\log n)^\kappa}{n\sqrt{\eta}},
%= \frac{(\log n)^{28+12\kappa}}{\eta^2 n\sqrt{d}},
\label{eq:tXX}
\end{equation}
where $\delta_2,\delta_3$ are defined in \prettyref{eq:deltas}.
%, provided that 
%\begin{equation}
%\delta_3 \leq c \eta
%\label{eq:QP-assump1}
%\end{equation}
 %for sufficiently small constant $c$.
\end{lemma}
\begin{proof}
    For notational convenience, put $G(z) = 2\iu/(n F(z)) $ and
    $\tG(z) = 2\iu/(n \tF(z)) $. Note that $ | \Im(z)| \le \eta/2$ 
for $z\in \Gamma$, and thus
$ \Im(z+\iu \eta)$ and $ \Im(z-\iu \eta)$ have different signs. Therefore
\[
|\tG(z)| \geq |\Im\tG(z)| = | \Im m_0(z+\iu \eta)|+|\Im m_0(z-\iu \eta)| \gtrsim \eta,
\]
where the last step follows from \prettyref{eq:Imm0}. 
Furthermore, by \prettyref{eq:wignertotalsumlaw}, we have $\sup_{z\in\Gamma} |G(z)-\tG(z)| \leq 2C \delta_3$.
In view of \prettyref{eq:QP-assump3}, $\delta_3 \ll \eta$. Hence we have $ |G(z)|\gtrsim \eta$ and 
\[
\sup_{z\in\Gamma} |F(z)-\tF(z)| \lesssim \frac{1}{n} \frac{\delta_3}{\eta^2}.
\]
Finally, by \prettyref{eq:XQP-contour} and \prettyref{eq:tXQP-contour}, we have
\[
|(\XQP-\tXQP)_{k\ell}| \leq \oint_\Gamma dz |F(z)-\tF(z)| |e_k^\top R_A(z)\ones| |e_\ell^\top R_B(z+\iu \eta)\ones| .
%\sup_{z\in\Gamma} 
\]
By \prettyref{eq:wignerrowsumlaw}, for all $k,\ell$, $|e_k^\top R_A(z)\ones|
    \lesssim \delta_2\sqrt{n}$  and $|e_\ell^\top R_B(z+\iu \eta)\ones| \lesssim
    \delta_2\sqrt{n}$. Combining the last two displays yields the desired claim.
\end{proof}

In view of the entrywise approximation, we may switch our attention to the approximate solution $\tXQP$ and establish
its diagonal dominance, assuming without loss of generality $\pi_*$ is the
identity permutation.
The proof parallels the analysis in \prettyref{sec:pf-main} so we focus on the differences. 
To make the scaling identical to the unconstrained case, define 
\begin{equation}
Y \triangleq n \tXQP =  \frac{1}{2\pi} \Re  \oint_\Gamma f(z) R_A(z) \allones  R_B(z+\iu \eta), 
 \label{eq:Y} 
\end{equation}
with
\[f(z) \triangleq \frac{2 \iu}{m_0(z+\iu \eta)-m_0(z-\iu \eta)}.\]
Compared with the unconstrained solution \prettyref{eq:resrep}, the only difference is the weighting factor $f(z)$.

We aim to show that with probability at least $1-n^{-D}$, for any constant $D>0$, the following holds:
\begin{enumerate}
	\item For off-diagonals, we have
	\begin{equation}
	\max_{k\neq \ell} |Y_{k\ell}| \lesssim \left(\log n\right)^{\kappa}/\sqrt{\eta}.
	\label{eq:Yoff}
	\end{equation}
	\item For diagonal entries, we have 
	\begin{equation}
            \min_k\left| Y_{kk} -\frac{4(1-\sigma^2)}{\pi\eta}\right|  \lesssim
            \frac{r(n)}{\eta}+\frac{\sigma}{\eta^2}
            +\left(\log n\right)^{\kappa} \frac{1}{\sqrt{\eta}} .
	\label{eq:Ydiag}
	\end{equation}
\end{enumerate}
In view of \prettyref{lmm:tXX}, this implies the desired \prettyref{eq:QPmain1} and \prettyref{eq:QPmain2}. 
Finally, analogous to \prettyref{cor:general},
under the assumption \prettyref{eq:assump_sigma_main-QP} with constants $c=1/(64C^2)$
and $c'=1/(2C)$, for all sufficiently large $n$,
$$
\frac{4(1-\sigma^2)}{\pi\eta} \ge \frac{7}{8\eta} > C \left(\frac{r(n)}{\eta}+
\frac{\sigma}{\eta^2} + 2(\log n)^{\kappa} \frac{1}{\sqrt{\eta}}\right),
$$ 
implying the diagonal dominance in \prettyref{eq:diagdom-QP}.

\subsubsection{Off-diagonal entries}
Let us first consider $Y_{12}$.
Recall that for $z \in \Gamma$, we have $|\Im(z+\iu \eta)| \gtrsim \eta$,
$|\Im (z-\iu \eta)| \gtrsim \eta$, and these imaginary parts
have opposite signs. Then
    \begin{equation}\label{eq:flb3}
        |f(z)| \leq \frac{2}{|\Im[m_0(z+\iu \eta)-m_0(z-\iu \eta)]|}
    =\frac{2}{|\Im m_0(z+\iu \eta)|+|\Im m_0(z-\iu \eta)|}
    \lesssim \frac{1}{\eta},
    \end{equation}
    where the last step applies (\ref{eq:Imm0}).
Analogous to \prettyref{eq:contour-goal1-wigner}, we get
\begin{align}
    2\pi Y_{12} &=\Re\left(\oint_{\Gamma} f(z) \left[\coord_1^\top R_A(z) \bone \right] 
    \left[\coord_2^\top R_B(z+ \iu \eta) \bone \right] dz\right)\nonumber \\
    & = \Re\left(\alpha -  a_1^\top  g -  b_2^\top  h + a_1^\top  M b_2\right)+ 
 O\left(  \frac{\delta_1  \delta^2_2  n}{\eta}    \right),  \label{eq:contour-goal1-wigner-QP}
\end{align}
where
\begin{align}
\alpha \triangleq & ~ \oint_{\Gamma} f(z) m_0(z)m_0(z+ \iu \eta )    dz , \label{eq:alpha-QP}\\
g \triangleq & ~ \oint_{\Gamma} f(z) m_0(z)m_0(z+ \iu \eta )R_A^{(12)}(z) \bone_{n-2}    dz , \label{eq:g-QP}\\
h \triangleq & ~  \oint_{\Gamma}f(z)  m_0(z)m_0(z+ \iu \eta )R_B^{(12)}(z+ \iu \eta ) \bone_{n-2}   dz , \label{eq:h-QP}\\
M \triangleq & ~ \oint_{\Gamma} f(z) m_0(z)m_0(z+ \iu \eta )R_A^{(12)}(z) \bJ_{n-2}  R_B^{(12)}(z+ \iu \eta )dz. \label{eq:M-QP}
\end{align}

Here the constant $\Re \alpha$
is in fact equal to $2\pi$, which is consistent with the row-sum constraints.
Indeed, opening up $m_0(z)$ and applying the Cauchy integral formula, we have
\begin{align}
\Re \alpha
= & ~  \Re \oint dz \frac{2 \iu}{m_0(z+\iu \eta)-m_0(z-\iu \eta)} m_0(z) m_0(z+\iu \eta) \nonumber\\
= & ~  \int \rho(x) dx \Re \oint dz \frac{1}{x-z} \frac{2 \iu~m_0(z+\iu \eta) }{m_0(z+\iu \eta)-m_0(z-\iu \eta)} \nonumber\\
= & ~ \int \rho(x) dx \Re
    \left[(-2\pi\iu) \frac{2 \iu~m_0(x+\iu \eta) }{m_0(x+\iu \eta)-m_0(x-\iu
    \eta)}\right]\nonumber\\
= & ~  2\pi \int \rho(x) dx \Re\qth{\frac{2\,m_0(x+\iu \eta) }{2\iu \Im m_0(x+\iu \eta)}} = 2\pi \int \rho(x) dx = 2\pi. \label{eq:alpha1}
\end{align}

As in \prettyref{sec:offdiag-term}, to bound the linear and bilinear terms, we need to bound the $\ell_\infty$-norms and $\ell_2$-norms of $g,h$ and $M$.
Clearly, by \prettyref{eq:flb3}, the $\ell_\infty$-norms are at most an $O(1/\eta)$ factor of those obtained in \prettyref{eq:ginfty} and \prettyref{eq:Minfty}, i.e., 
$\|g\|_\infty \lesssim \delta_2 \sqrt{n}/\eta$  and $\|M\|_\infty \lesssim
\delta_2^2 n/\eta$.
The $\ell_2$-norms need to be bounded more carefully. The following result is the counterpart of \prettyref{lmm:M_norm}:
\begin{lemma}\label{lmm:M_norm-QP}
Assume the same setting of \prettyref{lmm:M_norm}, and define
    $M$, $g$, and $h$ as in (\ref{eq:g-QP}--\ref{eq:M-QP}) with
    $R_A$, $R_B$ in place of $R_A^{(12)},R_B^{(12)}$. Then 
$\|M\|_F^2 \lesssim n^2/\eta$,
    $\|g\|^2 \lesssim n \log(1/\eta)$, and
    $\|h\|^2 \lesssim n \log(1/\eta)$.
\end{lemma}
\begin{proof}
We start with $\|M\|_F$, as the arguments for $\|g\|$ and $\|h\|$ are
analogous and simpler.
Recall the contour $\Gamma'$ from \prettyref{fig:contour}.
Proceeding as in the proof of \prettyref{lmm:M_norm}, we have
\begin{align*}
& ~ \frac{1}{n^2} \fnorm{M}^2 \\
= & ~ - \oint_{\Gamma} dz \oint_{\Gamma'} dw \;
    m_0(z) m_0(z+\iu \eta ) m_0(w) m_0(w-\iu \eta ) f(z) f(w) \times\\
    &\hspace{1in}\frac{n^{-1}\ones^\top(R_A(z)-R_A(w))\ones}{z-w}
    \frac{n^{-1}\ones^\top(R_B (z+ \iu \eta )- R_B (w- \iu \eta ))  \ones}{z+ \iu \eta - (w- \iu \eta)}	\\
%\stepb{=} 
=
& ~ - \underbrace{\oint_{\Gamma} dz \oint_{\Gamma'} dw\,m_0(z) m_0(z+\iu \eta ) m_0(w) m_0(w-\iu \eta ) f(z)f(w)
\frac{m_0(z)-m_0(w)}{z-w}
\frac{m_0(z+ \iu \eta )- m_0(w- \iu \eta )  }{z+ \iu \eta - (w- \iu
    \eta)}}_{\termI}	\\
    & ~ + \termII ,
\end{align*}
where $\termII$ denotes the remainder term. 
%\nbr{JX. Do we miss a minus sign in front due to switching $\bar{w}$ to $w$?}
Applying \prettyref{eq:wignertotalsumlaw}, \prettyref{eq:flb3}, and the
    boundedness of $m_0$, the residual term is bounded as
\begin{equation}
|\termII| \lesssim \delta_3 \oint_{\Gamma} dz \oint_{\Gamma'} dw  |f(z)| |f(w)|  
\frac{1}{|z-w|}
\frac{1}{|z+ \iu \eta - (w- \iu \eta)|} \lesssim \frac{\delta_3}{\eta^4}
    \lesssim \frac{1}{\eta}.\label{eq:termII}
\end{equation}

To control the leading term $\termI$, let us define the auxiliary contours
    $\gamma$ with vertices $\pm (2+2\eta) \pm (\eta/2)\iu$ and
    $\gamma'$ with vertices $\pm (2+\eta) \pm (\eta/4)\iu$.
By first deforming $\Gamma'$ to $\gamma'$ for each fixed $z \in \Gamma$, 
then deforming $\Gamma$ to $\gamma$, 
%the value of $\termI$ remains unchanged if the double contour integral is over $\gamma$ and $\gamma'$.
and finally taking the complex modulus and
applying $|m_0| \lesssim 1$, we get
\begin{align*}
|\termI|
\lesssim & ~\oint_{\gamma} dz \oint_{\gamma'} dw\, |f(z)||f(w)|
\left|\frac{m_0(z)-m_0(w)}{z-w}\right|
\left|\frac{m_0(z+ \iu \eta )- m_0(w- \iu \eta )  }
    {z+ \iu \eta - (w- \iu \eta)}\right|.
\end{align*}
The reason for performing these deformations is that
for any $z \in \gamma \cup \gamma'$, since
$\Re z \in [-2-2\eta,2+2\eta]$, we have from (\ref{eq:Imm0})
that $\Im m_0(z+\iu \eta) \asymp \sqrt{\eta+\zeta(z)}$
and $-\Im m_0(z-\iu \eta) \asymp \sqrt{\eta+\zeta(z)}$, where $\zeta(z)$ is as defined in Proposition \ref{prop:m0wigner}. 
Then we obtain
from (\ref{eq:flb3}) the improved bound
$|f(z)| \lesssim 1/\sqrt{\eta+\zeta(z)}$, and hence
\[|\termI| \lesssim
\oint_{\gamma} dz \oint_{\gamma'} dw\, \frac{1}{\sqrt{\eta+\zeta(z)}}
\frac{1}{\sqrt{\eta+\zeta(w)}}\left|\frac{m_0(z)-m_0(w)}{z-w}\right|
\left|\frac{m_0(z+ \iu \eta )- m_0(w- \iu \eta )  }
{z+ \iu \eta - (w- \iu \eta)}\right|.\]

To bound the above integral,
for a small constant $c_0>0$, consider the two cases where $|z-w| \geq
c_0$ and $|z-w|<c_0$. For the first case $|z-w| \geq c_0$,
we simply apply $|m_0| \lesssim 1$ and
$\sqrt{\eta+\kappa} \geq \sqrt{\eta}$ to get that
    \begin{equation}\label{eq:caseI}
        \oint \oint_{|z-w| \geq c_0} dz\,dw\, \frac{1}{\sqrt{\eta+\zeta(z)}}
\frac{1}{\sqrt{\eta+\zeta(w)}}\left|\frac{m_0(z)-m_0(w)}{z-w}\right|
\left|\frac{m_0(z+ \iu \eta )- m_0(w- \iu \eta )  }
{z+ \iu \eta - (w- \iu \eta)}\right| \lesssim \frac{1}{\eta}.
    \end{equation}

In the second case $|z-w|<c_0$, we claim that for $c_0$ sufficiently small, we
have
 \begin{align}
     |m_0(z)-m_0(w)| &\lesssim \sqrt{\eta+\zeta(z)}+\sqrt{\eta+\zeta(w)},
\label{eq:m0diffbound}\\
|m_0(z+\iu \eta)-m_0(w-\iu \eta)|
     &\lesssim \sqrt{\eta+\zeta(z)}+\sqrt{\eta+\zeta(w)}.
\label{eq:m0diffbound2}
\end{align}
Indeed, if $\zeta(z)>c_0$, then (\ref{eq:m0diffbound})
and (\ref{eq:m0diffbound2}) hold because
$\sqrt{\eta+\zeta(z)}+\sqrt{\eta+\zeta(w)} \asymp 1$. If
instead $\zeta(z) \leq c_0$, say, $\Re z \geq 2-c_0$, then from the explicit form
(\ref{eq:stieltjes}) for $m_0(z)$ we get $1+m_0(z) = \frac{2-z+\sqrt{z^2-4}}{2}$ and hence
\[
|1+m_0(z)| \lesssim |z-2| + \sqrt{|z-2||z+2|} \asymp \sqrt{|z-2|} \asymp \sqrt{\eta + \zeta(z)}.
\]
Furthermore, since $\Re w \geq \Re z - |z-w| \geq 2-2c_0$, we also have
$|1+m_0(w)| \lesssim \sqrt{\eta + \zeta(w)}$.
Then (\ref{eq:m0diffbound}) follows from the triangle inequality. The case of
$\Re z \leq -2+c_0$, and the argument for (\ref{eq:m0diffbound2}), are
analogous.

Having established (\ref{eq:m0diffbound}) and (\ref{eq:m0diffbound2}),
we apply
\begin{align*}
    \frac{\left(\sqrt{\eta+\zeta(z)}+\sqrt{\eta+\zeta(w)}\right)^2}
    {\sqrt{\eta+\zeta(z)}\sqrt{\eta+\zeta(w)}}
    &\lesssim \frac{\sqrt{\eta+\max(\zeta(z),\zeta(w))}}
    {\sqrt{\eta+\min(\zeta(z),\zeta(w))}}\\
    &\leq \frac{\sqrt{\eta+\min(\zeta(z),\zeta(w))}
    +\sqrt{|\zeta(z)-\zeta(w)|}}
    {\sqrt{\eta+\min(\zeta(z),\zeta(w))}}
    \leq 1+\frac{\sqrt{|z-w|}}{\sqrt{\eta}}
\end{align*}
to get
\begin{align*}
    &\oint \oint_{|z-w|<c_0} dz\,dw\, \frac{1}{\sqrt{\eta+\zeta(z)}}
\frac{1}{\sqrt{\eta+\zeta(w)}}\left|\frac{m_0(z)-m_0(w)}{z-w}\right|
\left|\frac{m_0(z+ \iu \eta )- m_0(w- \iu \eta )  }
{z+ \iu \eta - (w- \iu \eta)}\right| \nonumber \\
    &\lesssim \oint \oint_{|z-w|<c_0} dz\,dw\, 
    \left(1+\frac{\sqrt{|z-w|}}{\sqrt{\eta}}\right)
    \frac{1}{|z-w||z+\iu \eta-(w-\iu \eta)|}.
\end{align*}
Then divide this into the integrals where $|z-w|<\eta$ and $|z-w| \geq \eta$,
applying
\[\oint \oint_{|z-w|<\eta} dz\,dw\, 
    \frac{1}{|z-w||z+\iu \eta-(w-\iu \eta)|}
    \lesssim \oint \oint_{|z-w|<\eta} dz\,dw\, \frac{1}{\eta^2}
    \lesssim \frac{1}{\eta}\]
    and
\begin{align}
    &\oint \oint_{\eta \leq |z-w|<c_0} dz\,dw\, 
    \frac{\sqrt{|z-w|}}{\sqrt{\eta}}
    \cdot \frac{1}{|z-w||z+\iu \eta-(w-\iu \eta)|}\nonumber\\
    &\lesssim \frac{1}{\sqrt{\eta}}
    \oint \oint_{\eta \leq |z-w|<c_0} dz\,dw\, \frac{1}{|z-w|^{3/2}}
    \lesssim \frac{1}{\sqrt{\eta}}
    \frac{1}{\sqrt{\eta}} \lesssim \frac{1}{\eta}.
    \label{eq:caseIIbound}
\end{align}
Combining with the first case (\ref{eq:caseI}), we get
$|\termI| \lesssim 1/\eta$. Finally, combining with (\ref{eq:termII}), we get
$\|M\|_F^2 \lesssim n^2/\eta$ as desired.

Next we bound $\|g\|$. Proceeding as in the proof of \prettyref{lmm:M_norm} and
following the same argument as above, we get
\begin{align*}
    \frac{1}{n}\|g\|^2 &\lesssim  \oint_{\gamma} dz \oint_{\gamma'} dw  \,
|f(z)||f(w)|\frac{|m_0(z)-m_0(w)|}{|z-w|} + O\pth{\frac{\delta_3}{\eta^3}}\\
    &\lesssim  \oint_{\gamma} dz \oint_{\gamma'} dw  \,
    \frac{1}{\sqrt{\eta+\zeta(z)}}\frac{1}{\sqrt{\eta+\zeta(w)}}
    \frac{|m_0(z)-m_0(w)|}{|z-w|} + O\pth{\frac{\delta_3}{\eta^3}}.
\end{align*}
For $|z-w| \geq c_0$, we have
\begin{align*}
    &\oint \oint_{|z-w| \geq c_0}dzdw
\frac{1}{\sqrt{\eta+\zeta(z)}}\frac{1}{\sqrt{\eta+\zeta(w)}}
\frac{|m_0(z)-m_0(w)|}{|z-w|}\\
    &\lesssim \left(\oint \frac{1}{\sqrt{\eta+\zeta(z)}}dz\right)
\left(\oint \frac{1}{\sqrt{\eta+\zeta(w)}}dw\right)
\lesssim 1.
\end{align*}
For $|z-w|<c_0$, we apply $|m_0(z)-m_0(w)| \lesssim
\sqrt{\eta+\zeta(z)}+\sqrt{\eta+\zeta(w)}$ as above, so that
\begin{align*}
    &\oint \oint_{|z-w|<c_0}dzdw
\frac{1}{\sqrt{\eta+\zeta(z)}}\frac{1}{\sqrt{\eta+\zeta(w)}}
\frac{|m_0(z)-m_0(w)|}{|z-w|}\\
    &\lesssim \oint dz
    \frac{1}{\sqrt{\eta+\zeta(z)}} \oint dw \frac{1}{|z-w|}
    +\oint dw
    \frac{1}{\sqrt{\eta+\zeta(w)}} \oint dz \frac{1}{|z-w|}\\
    &\lesssim \log(1/\eta) \cdot \left(\oint dz
    \frac{1}{\sqrt{\eta+\zeta(z)}}+\oint dw
    \frac{1}{\sqrt{\eta+\zeta(w)}}\right) \lesssim \log(1/\eta).
\end{align*}
Combining the above yields $\|g\|^2 \lesssim n\log(1/\eta)$.
The argument for $\|h\|^2$ is the same as that for $\|g\|^2$.
\end{proof}

Finally, proceeding as in \prettyref{eq:lin-term-wigner}--\prettyref{eq:qua-term-wigner} and using the preceeding norm bounds, we obtain from \prettyref{eq:contour-goal1-wigner-QP}:
\[|Y_{12}| 
    \lesssim  ~ 	1+  \delta_4 \sqrt{n \log \frac{1}{\eta}}
    + \frac{\delta_4^2 n}{\sqrt{\eta}} + \frac{\delta_1  \delta^2_2 n}{\eta}
    \lesssim \frac{\delta_4^2 n }{\sqrt{\eta}}=
    \left(\log n\right)^{\kappa}/\sqrt{\eta},\]
with probability at least $1-n^{-D}$, for any constant $D$.
This implies the desired \prettyref{eq:Yoff} by the union bound.

\subsubsection{Diagonal entries}
We now consider $Y_{11}$.
%Analogous to \prettyref{eq:contour-goal2} and using \prettyref{eq:alpha1}, we have
%\begin{align}
%2\pi Y_{11} &= \oint_{\Gamma} \left[\coord_1^\top R_A(z) \bone\right] \left[ \bone^\top  R_B(z+ \iu \eta ) \coord_1 \right] dz \nonumber \\
%& =1
%- a_1^\top  g - b_1^\top  h + a_1^\top  M b_1 + O\left(  \frac{\delta_1 \delta_2^2 n}{\eta} \right),
%\label{eq:contour-goal2-QP}
%\end{align}
%where
%\begin{align*}
%g \triangleq & ~ \oint_{\Gamma} f(z) m_0(z)m_0(z+ \iu \eta )R_A^{(1)}(z) t_1   dz , \\
%h \triangleq & ~ \oint_{\Gamma} f(z) m_0(z)m_0(z+ \iu \eta )R_B^{(1)}(z+ \iu \eta ) s_1  dz, \\
%M \triangleq & ~ \oint_{\Gamma} f(z) m_0(z)m_0(z+ \iu \eta )R_A^{(1)}(z)  \allones R_B^{(1)}(z+ \iu \eta )dz.
%\end{align*}
Following the derivation from \prettyref{eq:contour-goal2} to \prettyref{eq:diag-dev} and using \prettyref{lmm:M_norm-QP} in place of \prettyref{lmm:M_norm}, we obtain, with probability at least $1-n^{-D}$ for any constant $D$,
\begin{align}
 \left| Y_{11} -\frac{1-\sigma^2}{2\pi} \Re  \frac{\Tr(M)}{n} \right|  \lesssim 
\left(\log n\right)^{\kappa} \frac{1}{\sqrt{\eta}},
\label{eq:diag-dev-QP}
\end{align}
where
\begin{align*}
M \triangleq & ~ \oint_{\Gamma} f(z) m_0(z)m_0(z+ \iu \eta )R_A^{(1)}(z)  \allones R_B^{(1)}(z+ \iu \eta )dz.
\end{align*}
The trace is computed by the following result, which parallels \prettyref{lmm:traceM} and \prettyref{lmm:m_product}:
\begin{lemma}\label{lmm:traceM-QP}
Suppose $\delta_3 \le \eta^2$. Assume the setting of \prettyref{lmm:traceM}. Define 
$$
M = \oint_{\Gamma} f(z) m_0(z)m_0(z+ \iu \eta )R_A(z)  \bJ R_B (z+ \iu \eta )dz.
$$
Then 
$$
\frac{1}{n} \Tr(M) = \frac{8+o_{\eta}(1)}{ \eta } + O\pth{ \frac{\sigma +\delta_3}{\eta^2}}.
%\frac{1}{\iu \eta} \oint_{\Gamma} m_0(z)m_0(z+\iu \eta)(m_0(z+\iu \eta)-m_0(z))dz + O\left(   \frac{\sigma}{\eta^2} + \frac{\delta_3}{ \eta} \right).
$$
\end{lemma}
\begin{proof}
Analogous to \prettyref{eq:TrMA-B}, we have
$\frac{1}{n} \Tr(M) =\termI - \termII$, where
\begin{align*}
\termI &=\frac{1}{ \iu \eta }\oint_{\Gamma} f(z) m_0(z)m_0(z+ \iu \eta )
 \frac{1}{n} \bone^\top (R_B(z+ \iu \eta )-R_A(z))\ones dz \\
\termII &=\frac{1}{ \iu \eta }\oint_{\Gamma} f(z)m_0(z)m_0(z+ \iu \eta )
 \frac{1}{n} \bone^\top R_B(z+ \iu \eta )(A-B)R_A(z) \bone  \,dz.
\end{align*}
To bound (II), consider two cases:
\begin{itemize}
	\item For $z \in \Gamma$ with $|\Im z| = \eta/2$, by the Ward identity and \prettyref{eq:wignertotalsumlaw}, we have
$$\left\| R_A(z) \bone \right\|^2 = \frac{2}{\eta}  |\Im  \bone^\top R_A(z) \bone|  
\lesssim \frac{n}{\eta} (|\Im m_0(z)| + O(\delta_3)).$$
%\footnote{Just a reminder to be removed later: before in \prettyref{lmm:traceM} we simply used $\left\| R_A(z) \bone \right\|^2 \lesssim \frac{n}{\eta}$. Here we need to be more precise because $f(z)$ is large and in fact $\oint |f(z)| \asymp \frac{1}{\eta}$ which is too big to be useful.}
and similarly, 
$$\left\| R_B(z+\iu \eta) \bone \right\|^2 \lesssim \frac{n}{\eta} (|\Im m_0(z+\iu \eta)| + O(\delta_3)).$$
Thus it holds that 
$$
\left| \bone^\top  R_B(z+ \iu \eta )(A-B)R_A(z) \bone \right|  \lesssim \frac{n \sigma}{\eta} \left(\sqrt{|\Im m_0(z) \Im m_0(z+\iu \eta)|} + \sqrt{\delta_3} \right).
$$
Using \prettyref{eq:Imm0} and \prettyref{eq:flb3}, we conclude that 
$$
|f(z)| \sqrt{|\Im m_0(z) \Im m_0(z+\iu \eta)|} \leq \frac{2\sqrt{|\Im m_0(z) \Im m_0(z+\iu \eta)|}}{|\Im m_0(z+\iu \eta)| + |\Im m_0(z-\iu \eta)|} \asymp 1
$$ 
for all $z \in \Gamma$ with $|\Im z| = \eta/2$. 
%\nbr{JX. I think it should be $|f(z)| \sqrt{|\Im m_0(z) \Im m_0(z+\iu \eta)|} \le \frac{2\sqrt{|\Im m_0(z) \Im m_0(z+\iu \eta)|}}{|\Im m_0(z+\iu \eta)| + |\Im m_0(z-\iu \eta)|} \le 1$. 
%} \nb{YW: Yes the first is $\leq$. I think the last $\asymp$ is correct.}

\item For $z\in \Gamma$ with $\Re z= \pm 3$, since $\|A\| \le 2.5$, $\left| \bone^\top  R_B(z+ \iu \eta )(A-B)R_A(z) \bone \right| \lesssim n \sigma$. 
\end{itemize}
Furthermore, by \prettyref{eq:flb3}, $|f(z)| \lesssim \frac{1}{\eta}$ for all $z\in\Gamma$. Combining the above two cases yields
\[
|\termII|\lesssim \frac{\sigma}{\eta^2} \pth{1 + \frac{\sqrt{\delta_3}}{\eta}} + \frac{\sigma}{\eta} \asymp \frac{\sigma}{\eta^2},
\]
%\nbr{JX. I think the second term $\sigma/\eta$ should be $\sigma/\eta^2$.}\nb{YW: I think it is correct as is, since the length of the vertical parts is $O(\eta)$.}
since $\delta_3 \leq \eta^2$ by the assumption. 
%\prettyref{eq:QP-assump3}.

For (I), applying (\ref{eq:wignertotalsumlaw}) again
and plugging the definition of $f(z)$ yields
\[
\termI =\frac{2}{ \eta }\oint_{\Gamma} m_0(z)m_0(z+ \iu \eta )
 \frac{m_0(z+ \iu \eta )-m_0(z)}{m_0(z+\iu \eta)-m_0(z-\iu \eta)} dz + O\pth{\frac{\delta_3 }{\eta^2}}.
\]
We now apply an argument similar to that of
\prettyref{lmm:m_product}: Note that
\[|m_0(z+\iu\eta)-m_0(z-\iu \eta)|
\geq \Im (m_0(z+\iu\eta)-m_0(z-\iu \eta))
\gtrsim \eta\]
by (\ref{eq:Imm0}), so the integrand is bounded for fixed $\eta$.
Then deforming $\Gamma$ to $\Gamma_\epsilon$ with vertices
$\pm (2+\eps) \pm \iu \eps$, taking $\eps \to 0$ for fixed $\eta$, and
applying the bounded convergence theorem, we have the equality
\begin{align}
&\oint_{\Gamma} m_0(z)m_0(z+ \iu \eta )
 \frac{m_0(z+ \iu \eta )-m_0(z)}{m_0(z+\iu \eta)-m_0(z-\iu \eta)} dz\nonumber\\
    &=\int_2^{-2} m_0^+(x)m_0(x+ \iu \eta )
 \frac{m_0(x+ \iu \eta )-m_0^+(x)}{m_0(x+\iu \eta)-m_0(x-\iu \eta)} dx
+\int_{-2}^2 m_0^-(x)m_0(x+ \iu \eta )
 \frac{m_0(x+ \iu \eta )-m_0^-(x)}{m_0(x+\iu \eta)-m_0(x-\iu \eta)}
dx.\label{eq:tmpintegralx}
\end{align}

We show that these integrands are uniformly bounded over small $\eta$:
For any constant $\delta>0$ and for $|x| \leq 2-\delta$,
we have the lower bound
\begin{equation}\label{eq:m0pmdiff}
    |m_0(x+\iu \eta)-m_0(x-\iu \eta)|=2\Im m_0(x+\iu \eta) \gtrsim
\sqrt{\zeta(x)+\eta} \geq \sqrt{\delta}.
\end{equation}
Then the above integrands are bounded by $C/\sqrt{\delta}$
for $|x| \leq 2-\delta$.
For $|x| \in [2-\delta,2]$, let us apply
\[|m_0(x+\iu \eta)-m_0^+(x)| \lesssim \sqrt{\zeta(x)+\eta}\]
as follows from (\ref{eq:m0diffbound}) and taking the limit $w \in \C^+ \to x$.
We have also $|m_0^+(x)-m_0^-(x)| \asymp
\sqrt{\zeta(x)} \lesssim \sqrt{\zeta(x)+\eta}$, so that
\[|m_0(x+\iu \eta)-m_0^-(x)| \lesssim \sqrt{\zeta(x)+\eta}.\]
Combining these cases with the first inequality of (\ref{eq:m0pmdiff}),
we see that the integrands of
(\ref{eq:tmpintegralx}) are uniformly bounded for all small $\eta$.

Now apply the bounded convergence theorem and
take the limit $\eta \to 0$, noting that
$\lim_{\eta \to 0} m_0(x+\iu\eta)=m_0^+(x)$ and
$\lim_{\eta \to 0} m_0(x-\iu\eta)=m_0^-(x)$. We get
\begin{align*}
&\lim_{\eta \to 0}
\oint_{\Gamma} m_0(z)m_0(z+ \iu \eta )
 \frac{m_0(z+ \iu \eta )-m_0(z)}{m_0(z+\iu \eta)-m_0(z-\iu \eta)} dz\\
&=\int_{-2}^2 m_0^-(x)m_0^+(x)\frac{m_0^+(x)-m_0^-(x)}
{m_0^+(x)-m_0^-(x)}dx=\int_{-2}^2 |m_0^+(x)|^2 dx=4.
\end{align*}
This gives $\termI=(8+o_\eta(1))/\eta+O(\delta_3/\eta^2)$. Combining with the
bound for $\termII$ yields the lemma.
\end{proof}
Finally, combining \prettyref{eq:diag-dev-QP} with \prettyref{lmm:traceM-QP} and
$\delta_3 \ll \eta$ from
\prettyref{eq:QP-assump3}, and applying a union bound yields the desired \prettyref{eq:Ydiag}.

	%Finally, combining all assumptions \prettyref{eq:QP-assump1} and \prettyref{eq:QP-assump3},

\section{Proof of resolvent bounds}\label{sec:locallaw}

In this section, we prove Theorem \ref{thm:locallawwigner}.
The entrywise bounds of part (a) are essentially the local semicircle
law of \cite[Theorem 2.8]{EKYYsparse}, restricted to the simpler
domain $\{z:\operatorname{dist}(z,[-2,2]) \geq (\log n)^{-a}\}$ and with small
modifications of the logarithmic factors. The bound in (b) follows from (a)
using a straightforward Schur complement identity. The bound in (c) is
more involved, and relies on the fluctuation averaging technique of
\cite[Section 5]{EKYYsparse}. We provide a proof of all
three statements using the tools of \cite{EKYYsparse}.

For each statement, it suffices to establish the claim with the stated
probability for each individual point $z \in D$. The uniform statement
over $z \in D$ then follows from a union bound over a sufficiently fine
discretization of $D$ (of cardinality an arbitrarily large polynomial in $n$)
and standard Lipschitz bounds for $m_0$ and $R_{jk}$ on the event of $\|A\| \leq
2.5$---we omit these details for brevity.

\subsection{Notation and matrix identities}

In this section,
for $S \subset [n]$, denote by $A^{(S)} \in \R^{n \times n}$ the matrix $A$ with
all elements in rows and columns belonging to $S$ replaced by 0. Denote
\[R^{(S)}(z)=(A^{(S)}-z \bI )^{-1} \in \C^{n \times n}.\]
Note that $R^{(S)}(z)$ is block-diagonal with respect to the block decomposition
$\C^n=\C^S \oplus \C^{[n] \setminus S}$, with $S \times S$ block
equal to $(-1/z) \bI _{|S|}$ and $([n] \setminus S) \times ([n] \setminus S)$ block
equal to the resolvent of the corresponding minor of $A$. (We will typically
only access elements of $R^{(S)}$ in this $([n] \setminus S) \times ([n]
\setminus S)$ block, in which case $R^{(S)}$ may be understood as the resolvent
of the minor of $A$.)

For $i \in [n]$, we write as shorthand
\[iS= \{i\} \cup S, \qquad \sum_k^{(S)}=\sum_{k \in [n] \setminus S}.\]
We usually omit the spectral argument $z$ for brevity.

\begin{lemma}[Schur complement identities]
For any $j \in [n]$,
\begin{equation}\label{eq:Rjj}
\frac{1}{R_{jj}}=a_{jj}-z-\sum_{k,\ell}^{(j)} a_{jk}R_{k\ell}^{(j)}a_{\ell j}.
\end{equation}
For any $j \ne k \in [n]$,
\begin{align}
R_{jk}&=-R_{jj}\sum_\ell^{(j)} a_{j\ell}R^{(j)}_{\ell k}
=R_{jj}R_{kk}^{(j)}\left(-a_{jk}+\sum_{\ell,m}^{(jk)} a_{j\ell}R_{\ell
m}^{(jk)}a_{mk}\right),
\label{eq:Rjk}\\
\coord_k^\top R &=\coord_k^\top R^{(j)} +\frac{R_{kj}}{R_{jj}} \cdot \coord_j^\top R , \label{eq:eR1}\\
\frac{1}{R_{kk}}&=\frac{1}{R_{kk}^{(j)}}
-\frac{(R_{kj})^2}{R_{kk}^{(j)}R_{jj}R_{kk}}.\label{eq:Rkkinv}
\end{align}
For any $j,k,\ell \in [n]$ with $j \notin \{k,\ell\}$,
\begin{equation}\label{eq:LOO}
R_{k\ell}=R_{k\ell}^{(j)}+\frac{R_{kj}R_{j\ell}}{R_{jj}}.
\end{equation}
%\nbr{JX. I think \prettyref{eq:Rjk} should be 
%$$
%-R_{jj}\sum_\ell^{(j)} a_{j\ell}R^{(j)}_{\ell k}
%=-R_{jj} a_{jk}R_{kk}^{(j)} + R_{jj}R_{kk}^{(j)}\sum_{\ell,m}^{(jk)} a_{j\ell}R_{\ell m}^{(jk)}a_{mk}
%$$
%} \nb{ZF: Thanks, fixed!}
These identities hold also for any $S \subset [n]$
with $R$ replaced by $R^{(S)}$ and with $j,k,\ell \in [n] \setminus S$.
\end{lemma}
\begin{proof}
For all but \eqref{eq:eR1}, see \cite[Lemma 4.5]{EKYYlocal} and \cite[Lemma
4.2]{EYYbulk}. As for \eqref{eq:eR1}, it is equivalent to verify that \prettyref{eq:LOO} holds also for $\ell=j$, which simply follows from $R^{(j)}_{kj}=0$, due to the block diagonal structure of $R^{(j)}$.
\end{proof}

\subsection{Entrywise bound}

We say an event occurs w.h.p.\ if its 
probability is at least $1-e^{-c(\log n)^{1+\eps}}$ for a universal
constant $c>0$. Let us show that \eqref{eq:wigneroffdiaglaw} and
\eqref{eq:wignerdiaglaw} hold for $z \in D$ w.h.p.

We start with \prettyref{eq:wignerdiaglaw}. Note that the $j$th row $\{a_{jk}: k\in[n]\}$ is independent of $A^{(j)}$ and hence $R^{(j)}$. 
Applying \eqref{eq:Rjj}, \eqref{eq:abound}, and \eqref{eq:quad1} conditional on
$A^{(j)}$, w.h.p.\ for all $j$,
\begin{align*}
\left|\frac{1}{R_{jj}}+z+\frac{1}{n}\sum_k^{(j)} R_{kk}^{(j)}\right|
=\left|a_{jj}-\sum_{k,\ell}^{(j)} a_{jk}R_{k\ell}^{(j)}a_{\ell j}+
\frac{1}{n}\sum_k^{(j)} R_{kk}^{(j)}\right|
&\leq (\log
n)^{2+2\eps}\left(\frac{1}{\sqrt{d}}+\frac{2\|R^{(j)}\|_\infty}{\sqrt{d}}
+\frac{\|R^{(j)}\|_F}{n}\right).
\end{align*}
Note that $\|R^{(j)}\|_\infty \le \|R^{(j)}\|$,
$\|R^{(j)}\|_F \le \sqrt{n}\|R^{(j)}\|$, and $d \leq n$.
For $z \in D_1$ and any $S \subset [n]$,
we have $\|R^{(S)}\| \le 1/|\Im z| \le (\log n)^a$.
For $z \in D_2$, we have $\|R^{(S)}\| \le 10$ on the event $\|A\| \le
2.5$, which occurs w.h.p.\ by Lemma \ref{lmm:normbound}. Then in both cases,
we get
\begin{equation}
\left|\frac{1}{R_{jj}}+z+\frac{1}{n}\sum_k^{(j)} R_{kk}^{(j)}\right|
\lesssim \frac{(\log n)^{2+2\eps+a}}{ \sqrt{d} }.\label{eq:Rjjinvtmp}
\end{equation}

Since $|z| \leq 10$, $|R_{kk}^{(j)}| \le (\log n)^a$, and
$d \gg (\log n)^{4+4\eps}$, this implies $1/|R_{jj}| \lesssim (\log n)^a$.
Let $m_n(z)=n^{-1}\Tr R(z)$ be the empirical Stieltjes transform. Then
\[\left|m_n-\frac{1}{n}\sum_k^{(j)} R_{kk}^{(j)}\right|
=\left|\frac{1}{n}R_{jj}+\frac{1}{n}\sum_k^{(j)} (R_{kk}-R_{kk}^{(j)})\right|
%=\left|\frac{1}{n}R_{jj}+\frac{1}{n}\sum_k^{(j)} \frac{R_{kj}^2}{R_{jj}}\right|
\overset{\eqref{eq:LOO}}{=} \left|\frac{1}{n}\sum_k \frac{R_{kj}^2}{R_{jj}}\right|
= \frac{\|\coord_j^\top R\|^2}{n|R_{jj}|} 
\leq \frac{\|R\|^2}{n|R_{jj}|} 
\lesssim \frac{(\log n)^{3a}}{n}.\]
Using $d \leq n$ and combining with \eqref{eq:Rjjinvtmp}, w.h.p.\ for all $j$,
\begin{equation}\label{eq:Rjjrecurrence}
\left|\frac{1}{R_{jj}}+z+m_n\right| \lesssim \frac{(\log n)^{2+2\eps+a}}{
\sqrt{d} }.
\end{equation}
Then by the triangle inequality, also w.h.p.\ for all $j \neq k$,
\[\left|\frac{1}{R_{jj}}-\frac{1}{R_{kk}}\right|
\lesssim \frac{(\log n)^{2+2\eps+a}}{ \sqrt{d} },\]
so
\[\left|\frac{m_n}{R_{jj}}-1\right|
=\left|n^{-1} \sum_k \frac{R_{kk}-R_{jj}}{R_{jj}}\right|
\leq\max_k \left|\frac{R_{kk}-R_{jj}}{R_{jj}}\right|
=\max_k |R_{kk}| \left|\frac{1}{R_{jj}}-\frac{1}{R_{kk}}\right|
\lesssim \frac{(\log n)^{2+2\eps+2a}}{\sqrt{d}}.\]
For $d \gg (\log n)^{4+4\eps+4a}$, this implies $ \frac{3}{2} |R_{jj}| \ge |m_n| \ge |R_{jj}|/2$ w.h.p.\
for all $j$. Then also
\[\left|\frac{1}{R_{jj}}-\frac{1}{m_n}\right|
=\frac{|R_{jj}-m_n|}{|R_{jj}||m_n|}
\le \max_k \frac{|R_{jj}-R_{kk}|}{|R_{jj}||m_n|}
\le \max_k \frac{2|R_{jj}-R_{kk}|}{|R_{jj}||R_{kk}|}
= 2\max_k \left|\frac{1}{R_{jj}}-\frac{1}{R_{kk}}\right|,\]
so
\begin{equation}\label{eq:Rjjinvest}
\left|\frac{1}{R_{jj}}-\frac{1}{m_n}\right| \lesssim
\frac{(\log n)^{2+2\eps+a}}{ \sqrt{d} }.
\end{equation}
%\nb{ZF: Modified the above arguments to get a better log factor.}
Combining with \eqref{eq:Rjjrecurrence}, w.h.p.\ we have
\[\frac{1}{m_n}+z+m_n=r_n, \qquad |r_n| \lesssim \frac{(\log n)^{2+2\eps+a}}{
\sqrt{d} } \ll (\log n)^{-a}.\]

Solving for $m_n$ yields
\[m_n \in \frac{-z+r_n \pm \sqrt{z^2-4-2zr_n+r_n^2}}{2}\]
where the right side denotes the two complex square-roots.
Note that
$|z^2-4|=|z-2| \cdot |z+2| \gtrsim (\log n)^{-a}|z|$
and $|z| \ge (\log n)^{-a}$ for all $z \in D$.
Then, as $(\log n)^{-a} \gg |r_n|$,
we have $|z^2-4| \gg |zr_n| \gg |r_n|^2$.
Letting $m_0$ be the Stieltjes transform of the semicircle law, and letting
$\tilde{m}_0=1/m_0$ be the other root of the quadratic equation
\eqref{eq:wignerequation}, we obtain by a Taylor expansion of the square-root
that
\begin{equation}\label{eq:dichotomy}
\min(|m_n-m_0|,|m_n-\tilde{m}_0|)
\lesssim |r_n|\left(1+\frac{|z|}{\sqrt{|z^2-4|}}\right)
\lesssim \frac{|r_n|}{\sqrt{\zeta(z)+|\Im z|}},
\end{equation}
where $\zeta(z)$ is as defined in Proposition \ref{prop:m0wigner}.

To argue that this bound holds for $|m_n-m_0|$ rather
than $|m_n-\tilde{m}_0|$, consider
first $z \in D_1$ with $\Im z>0$. In this case $m_n\in\complex_+$ and $\tilde m_0\in\complex_-$. Furthermore, 
note that (\ref{eq:Imm0}) implies
$\Im m_0(z) \geq (\Im z)/\sqrt{\zeta(z)+\Im z}$, and hence
$\Im \tilde{m}_0=-(\Im m_0)/|m_0|^2 \le -c(\log n)^{-a}/\sqrt{\zeta(z)+\Im z}$.
Since $\Im m_n>0$ and $|r_n| \ll (\log n)^{-a}$, 
\eqref{eq:dichotomy} must hold for
$|m_n-m_0|$ rather than $|m_n-\tilde{m}_0|$.
The same argument applies for $z \in D_1$ with $\Im z<0$. For $z \in D_2$,
we have $||m_0(z)|-1| \geq c$ and hence 
$|m_0(z)-\tilde{m}_0(z)|>c$ for a constant $c>0$. Consider the
point $z' \in D_1 \cap D_2$ with $\Re z'=\Re z$ and
$\Im z'=(\log n)^{-a}$. 
Note that for all $z\in D_2$, $|\frac{d}{dz} m_0(z)| \lesssim 1$ and,  on the
event $\|A\| \le 2.5$, $|\frac{d}{dz} m_n(z)| \lesssim 1$ also. Thus 
$|m_0(z)-m_0(z')| \le C(\log n)^{-a}$
and $|m_n(z)-m_n(z')| \le C(\log n)^{-a}$. Since we have already shown that \eqref{eq:dichotomy}
holds for $|m_n(z')-m_0(z')|$ in the previous case, this implies also that
\eqref{eq:dichotomy} must hold for $|m_n-m_0|$ rather than for
$|m_n-\tilde{m}_0|$.

Applying $|\Im z| \geq (\log n)^{-a}$, (\ref{eq:dichotomy}) yields w.h.p.
\begin{equation}
|m_n-m_0| \lesssim (\log n)^{a/2}|r_n| \lesssim
\frac{(\log n)^{2+2\eps+3a/2}}{\sqrt{d}}.
\label{eq:dichotomy1}
\end{equation}
Recalling (\ref{eq:Rjjinvest}), $|R_{jj}| \leq (\log n)^a$ and $|m_n| \le \frac{3}{2} |R_{jj}|$, we get
\begin{equation}\label{eq:Rjjest}
|R_{jj}-m_n| \lesssim |R_{jj}||m_n|
\cdot \frac{(\log n)^{2+2\eps+a}}{\sqrt{d}}
\lesssim \frac{(\log n)^{2+2\eps+3a}}{\sqrt{d}}.
\end{equation}
Combining the last two displayed equations gives the weak estimate
\[|R_{jj}-m_0| \lesssim \frac{(\log n)^{2+2\eps+3a}}{\sqrt{d}}.\]
Since $d \gtrsim (\log n)^{4+4\eps+6a}$ by assumption, this and $|m_0(z)| \asymp 1$ imply $|R_{jj}| \lesssim
1$ w.h.p. Then applying the last display and \prettyref{eq:dichotomy1} to the first inequality of (\ref{eq:Rjjest}) yields the desired
estimate
\[|R_{jj}-m_0|
\leq |R_{jj}-m_n|+|m_n-m_0| \lesssim \frac{(\log n)^{2+2\eps+3a/2}}{\sqrt{d}}.\]

To show \prettyref{eq:wigneroffdiaglaw} for the off-diagonals, we now
apply \eqref{eq:Rjk}, \eqref{eq:abound}, \eqref{eq:quad2} conditional on
$R^{(jk)}$, $|R_{jj}| \lesssim 1$, $|R_{kk}^{(j)}| \lesssim 1$,
$\|R^{(jk)}\|_\infty \le (\log n)^a$,
$\|R^{(jk)}\|_F \le \sqrt{n}(\log n)^a$, and $d \le n$ to get w.h.p.
\begin{align*}
|R_{jk}|&=|R_{jj}||R_{kk}^{(j)}|\left|-a_{jk}+\sum_{\ell,m}^{(jk)}
a_{j\ell}R_{\ell m}^{(jk)}
a_{mk}\right|\nonumber\\
&\lesssim (\log n)^{2+2\eps}
\left(\frac{1}{\sqrt{d}}+\frac{2\|R^{(jk)}\|_\infty}{ \sqrt{d} }+\frac{\|R^{(jk)}\|_F}{n}\right)
\lesssim \frac{(\log n)^{2+2\eps+a}}{ \sqrt{d} }.
\end{align*}

\subsection{Row sum bound}\label{sec:rowsum}

We now show that \eqref{eq:wignerrowsumlaw} holds for $z \in D$ w.h.p. Set
\begin{equation}
\cZ_i\triangleq \sum_{j,k}^{(i)} a_{ik}R_{kj}^{(i)}  =\sum_k^{(i)} a_{ik} \left(\coord_k^\top R^{(i)} \bone \right)
\end{equation}
where the last equality holds because $R^{(i)}_{ki}=0$ for $k\neq i$.
Applying \eqref{eq:Rjk},
\[\coord_i^\top R \bone =\sum_j R_{ij}  =R_{ii}  -R_{ii}\cZ_i.\]
Then applying \eqref{eq:wignerdiaglaw}, w.h.p.\ for every $i \in [n]$,
\begin{equation}\label{eq:eR1tmp}
\left| \coord_i^\top R \bone  \right| \lesssim 1+|\cZ_i|.
\end{equation}
Applying \eqref{eq:linearconcentration} conditional on $A^{(i)}$,
w.h.p.\ for every $i \in [n]$,
\begin{equation}\label{eq:Zbound}
|\cZ_i| \le (\log n)^{1+\eps} \left(\frac{\max_{k \ne i} |\coord_k^\top R^{(i)} \bone |}{ \sqrt{d} }
+\sqrt{\frac{\sum_k^{(i)} |\coord_k^\top R^{(i)} \bone |^2}{n}}\right).
\end{equation}
For the second term above, we apply $\|R^{(i)}\| \le (\log n)^a$ w.h.p.\ to get
\begin{equation}\label{eq:term2bound}
\sum_k^{(i)} \left|\coord_k^\top R^{(i)} \bone \right|^2 \le \bone^\top \overline{R^{(i)}}R^{(i)} \bone
\le (\log n)^{2a} n.
\end{equation}
For the first term, we apply \eqref{eq:eR1}, \eqref{eq:wigneroffdiaglaw},
and \eqref{eq:wignerdiaglaw} to get, w.h.p.\ for all $k \ne i$,
\begin{equation}\label{eq:maxbound}
\left|\coord_k^\top R^{(i)} \bone \right|
=\left|\coord_k^\top R \bone -\frac{R_{ki}}{R_{ii}} \cdot \coord_i^\top R \bone \right|
\le \left|\coord_k^\top R \bone \right|+\frac{C(\log n)^{2+2\eps+a}}{ \sqrt{d} }  \left|\coord_i^\top R \bone \right|.
\end{equation}
Applying $d \gg (\log n)^{4+4\eps+2a}$ and substituting 
\eqref{eq:term2bound} and \eqref{eq:maxbound} into \eqref{eq:Zbound}
and then into \eqref{eq:eR1tmp}, we get that 
\begin{equation}\label{eq:rowsumbd_0}
\left| \coord_i^\top R \bone \right| \lesssim
1+(\log n)^{1+\eps}\left(\frac{\max_k
|\coord_k^\top R \bone  |}{ \sqrt{d} }+ (\log n)^a  \right)
\end{equation}
Taking the maximum over $i$ and 
rearranging yields (\ref{eq:wignerrowsumlaw}).

\subsection{Total sum bound}\label{sec:totalsum}
Finally, we show that \eqref{eq:wignertotalsumlaw} holds with probability
$1-e^{-c(\log n)(\log\log n)}$ for $z \in D$.
As above, we set 
\begin{equation}\label{eq:Zi}
\cZ_i=\sum_{j,k}^{(i)} a_{ik}R_{kj}^{(i)} =\sum_k^{(i)} a_{ik} \left(\coord_k^\top R^{(i)} \bone \right).
\end{equation}
Note that if we apply \eqref{eq:term2bound}, \eqref{eq:maxbound},
and \eqref{eq:wignerrowsumlaw} to \eqref{eq:Zbound}, we obtain w.h.p.\ that
for every $i \in [n]$,
\begin{equation}\label{eq:Zboundfinal}
|\cZ_i| \le (\log n)^{1+\eps+a}.
\end{equation}
The main step of the proof of \eqref{eq:wignertotalsumlaw}
is to use the weak dependence of $\cZ_1,\ldots,\cZ_n$
to obtain a bound on $n^{-1}\sum_i \cZ_i$ that is better than
$(\log n)^{1+\eps+a}$. The idea is encapsulated by the following
abstract lemma from \cite{EKYYsparse}.

\begin{lemma}[Fluctuation averaging]\label{lem:fluctuationavg}
Let $\Xi$ be an event defined by $A$, let $\cZ_1,\ldots,\cZ_n$ be random
variables which are functions of $A$,
let $p$ be an ($n$-dependent) even integer,
and let $x,y>0$ be deterministic positive quantities.
Suppose there exist random variables $\cZ_i^{[U]}$,
indexed by $U \subseteq [n]$ and $i \in [n] \setminus U$,
which satisfy $\cZ_i^{[\emptyset]}=\cZ_i$ as well as the following conditions:
\begin{enumerate}[(i)]
\item Let $a_i$ denote the $i^\text{th}$ row of $A$. Then
$\cZ_i^{[U]}$ is independent of $\{a_j:j \in U\}$, and $\E_i\left[\cZ_i^{[U]}\right]=0$ where $\E_i$ is
the partial expectation over only $a_i$.
\item For any $U \subseteq S \subset [n]$ with $|S| \le p$, and for any
$i \notin S$, denote $u=|U|+1$ and
\begin{equation}\label{eq:ZiSU}
\cZ_i^{S,U}=\sum_{T:\,T \subseteq U} (-1)^{|T|}
\cZ_i^{[(S \setminus U) \cup T]}.
\end{equation}
Then for a constant $C>0$ and any integer $r \in [0,p]$,
\[\E\left[\1\{\Xi\} \left|\cZ_i^{S,U} \right|^r\right] \le \Big(y(Cxu)^u\Big)^r.\]
Furthermore,
\[x \le 1/(p^5 \log n).\]
\item Let $\mathcal{A} \subset \R^{n \times n}$
be the matrices satisfying $\Xi$, i.e., $\Xi=\{A\in\mathcal{A}\}$.
Let $\mathcal{A}_i=\{B \in \R^{n \times n}:B^{(i)}=A^{(i)} \text{ for some }
A \in \mathcal{A}\}$, and define the event $\Xi_i=\{A \in \mathcal{A}_i\}$.
For a constant $C>0$ and any $U,S,i$ as above,
$\E\left[\1\{\Xi_i\}\left|\cZ_i^{S,U} \right|^2\right] \le n^{Cp}$.
\item For a constant $C>0$ and any $U \subseteq [n]$,
$\1\{\Xi\}\left|\cZ_i^{[U]} \right| \le yn^C$.
\item For a constant $\eps>0$, $\P[\Xi] \ge 1-e^{-c(\log n)^{1+\eps}p}$.
\end{enumerate}
Then for constants $C',n_0>0$ depending on $C,\eps$ above, and for
all $n \ge n_0$,
\[
\P\left[\1\{\Xi\}\left|n^{-1}\sum_i \cZ_i\right| \ge p^{12}
y(x^2+n^{-1})\right] \le (C'/p)^p.\]
\end{lemma}
\begin{proof}
See \cite[Theorem 5.6]{EKYYsparse}. (The theorem is stated for $1+\eps=3/2$ in
condition (v), but the proof holds for any $\eps>0$.)
\end{proof}

The important condition encapsulating weak dependence above is (ii). Applying
(ii) with $U=\emptyset$, the condition requires first
that each $|\cZ_i^{[S]}|$, and
in particular each $|\cZ_i|=|\cZ_i^{[\emptyset]}|$, is of typical size
$Cxy$. In the application of this lemma, for $S=U$ and $i \notin U$,
we will define the variables $\cZ_i^{[V]}$ for $\emptyset \subseteq V \subseteq
U$ such that the quantity $\cZ_i^{U,U}$ in (\ref{eq:ZiSU})
is the variable $\cZ_i$ with its dependence on all $\{a_j:j \in U\}$
projected out by an inclusion-exclusion procedure. Then condition (ii) requires
that $\cZ_i$ depends weakly on $\{a_j:j \in U\}$, in the sense that
$|\cZ_i^{U,U}|$ is of typical size $x^{|U|+1}y \cdot (C(|U|+1))^{|U|+1}$,
which is roughly smaller than $|\cZ_i|$ by a factor of $x|U|$ for
each element of $U$. 
Assuming $1/\sqrt{n} \ll x \ll p^{-12}$,
the above then estimates the average $|n^{-1}\sum_i \cZ_i|$ to be of the
smaller order $p^{12}yx^2 \ll xy$. We refer the reader to the discussion in
\cite{EKYYsparse} for additional details.
%\nbr{JX. I could not fully understand the intuitive explanation above. I think maybe 
%some more explanation on how to go from $\cZ_i^{S,U}$ to the desired bound on 
%$\sum_i \cZ_i$ would be helpful here.}\nb{ZF: Tried to clarify this a bit.}

We will check that the conditions of this lemma hold for $\cZ_i$ as defined by
\eqref{eq:Zi}, with the appropriate construction of variables $\cZ_i^{[U]}$.
To this end, we first extend \eqref{eq:wigneroffdiaglaw}, \eqref{eq:wignerdiaglaw}, and
\eqref{eq:wignerrowsumlaw} to $R^{(S)}$ for $|S| \le \log n$ in the following deterministic lemma:

\begin{lemma}\label{lem:locallawRS}
Suppose \eqref{eq:wigneroffdiaglaw}, \eqref{eq:wignerdiaglaw},
and \eqref{eq:wignerrowsumlaw} hold with the constant $C \equiv C_0$
for a deterministic symmetric matrix $A$, some $z \in D$, and all $j,k \in [n]$.
Then for all $S \subset [n]$ with $|S| \le \log n$,
and all $j \ne k \in [n] \setminus S$,
\begin{align}
|R^{(S)}_{jj}(z)-m_0(z)| &\le \frac{2C_0(\log n)^{2+2\eps+3a}}{ \sqrt{d} },
\label{eq:diagRS}\\
|R^{(S)}_{jk}(z)| &\le \frac{2C_0(\log n)^{2+2\eps+a}}{ \sqrt{d} },
\label{eq:offdiagRS}\\
|\coord_j^\top R^{(S)}(z) \bone| &\le 2C_0(\log n)^{1+\eps+a}.
\label{eq:rowsumRS}
\end{align}
\end{lemma}
\begin{proof}
For integers $s \ge 0$, let
\begin{align*}
\Lambda_s^d&=\max\Big\{|R_{jj}^{(S)}-m_0|:\;|S|=s,
\;j \in [n] \setminus S\Big\},\\
\Lambda_s^o&=\max\Big\{|R_{jk}^{(S)}|:\;|S|=s,
\;j \ne k \in [n] \setminus S\Big\}.
\end{align*}
When \eqref{eq:wigneroffdiaglaw} and \eqref{eq:wignerdiaglaw} hold,
we have that $\Lambda_s^d \le C_0(\log n)^{2+2\eps+3a}/ \sqrt{d} $
and $\Lambda_s^o \le C_0(\log n)^{2+2\eps+a}/ \sqrt{d} $ for $s=0$.
By \eqref{eq:LOO}, we have for each $s \ge 1$ and $* \in \{d,o\}$ that
\begin{equation}\label{eq:lambdarecursion}
\Lambda_{s+1}^* \le \Lambda_s^*+\frac{(\Lambda_s^o)^2}{|m_0|-\Lambda_s^d}.
\end{equation}
Assume inductively that for some $s \le \log n$,
\begin{equation}\label{eq:inductivebound}
\Lambda_s^d \le \frac{C_0(\log n)^{2+2\eps+3a}}{ \sqrt{d}
}\left(1+\frac{4C_0(\log n)^{2+2\eps+a}}
{ |m_0|\sqrt{d} }\right)^s,
\; \Lambda_s^o \le \frac{C_0(\log n)^{2+2\eps+a}}{ \sqrt{d} }
\left(1+\frac{4C_0(\log n)^{2+2\eps+a}}{ |m_0|\sqrt{d} }\right)^s.
\end{equation}
Applying $d \gg (\log n)^{6+4\eps+2a}$, $|m_0| \ge c$, and $s \le \log n$,
this implies in particular that
\[\Lambda_s^d \le \frac{2C_0(\log n)^{2+2\eps+3a}}{ \sqrt{d} },
\qquad \Lambda_s^o \le \frac{2C_0(\log n)^{2+2\eps+a}}{ \sqrt{d} }.\]
We then have $|m_0|-\Lambda_s^d \ge |m_0|/2$ for $d \gg (\log n)^{4+4\eps+6a}$,
so \eqref{eq:lambdarecursion} yields
\[\Lambda_{s+1}^* \leq \max(\Lambda_s^*,\Lambda_s^o)
\left(1+\frac{2\Lambda_s^o}{|m_0|}\right)
\leq \max(\Lambda_s^*,\Lambda_s^o)
\left(1+\frac{4C_0(\log n)^{2+2\eps+a}}{|m_0| \sqrt{d}}\right).\]
Thus both bounds of \eqref{eq:inductivebound} hold for $s+1$,
completing the induction. This establishes \eqref{eq:diagRS}
and \eqref{eq:offdiagRS}.

To show \prettyref{eq:rowsumRS}, set
\[\Gamma_s=\max\{|\coord_j^\top R^{(S)} \bone |:\;|S|=s,\;j \notin S\}.\]
When \eqref{eq:wignerrowsumlaw} holds, $\Gamma_0 \le C_0(\log n)^{1+\eps+a}$.
Applying \eqref{eq:eR1} and the bound $|m_0|-\Lambda_s^d \ge |m_0|/2$, we have 
\[
\Gamma_{s+1} \le (1+2\Lambda_s^o/|m_0|)\Gamma_s \overset{\prettyref{eq:offdiagRS}}{\leq} \pth{1 + \frac{4C_0(\log n)^{2+2\eps+3a}}{ |m_0|\sqrt{d}  }} \Gamma_s,
\]
Thus $\Gamma_s \le 2\Gamma_0$ for all $s \le \log n$.
\end{proof}

\begin{lemma}\label{lem:fluctuationavgconditions}
Fix $z \in D$. Let $\cZ_i$ be defined in \eqref{eq:Zi}. For $U \subset [n]$
not containing $i$, define
\[\cZ_i^{[U]}=\sum_{j,k}^{(iU)} a_{ik}R_{kj}^{(iU)}=\sum_k^{(iU)} a_{ik}
(\coord_k^\top R^{(iU)} \bone ).\]
Let $\Xi$ be the event where
\begin{itemize}
\item \eqref{eq:wigneroffdiaglaw}, \eqref{eq:wignerdiaglaw},\
and \eqref{eq:wignerrowsumlaw} all hold at $z$, for all distinct $j,k \in [n]$,
\item $|a_{ij}| \le 1$ for all $i,j \in [n]$, and
\item $\|A\| \le 2.5$.
\end{itemize}
Let $p \in [2,(\log n)-1]$ be an even integer, and set
\[x=\frac{(\log n)^{2+2\eps+a}}{ \sqrt{d} }, \qquad y=C' \sqrt{d} (\log
n)^{-\eps}\]
for a sufficiently large constant $C'>0$. Then
all of the conditions of Lemma \ref{lem:fluctuationavg} are satisfied.
\end{lemma}
\begin{proof}
Condition (i) is clear by definition, as row $a_i$ of $A$ is independent of
$R^{(iU)}$.

To check (ii), note first that the bound $x \le 1/(p^5 \log n)$ follows
from $d \geq (\log n)^{16+4\eps+2a}$.
For $U \subseteq S$ and $i \notin S$ we write
\begin{align*}
\cZ_i^{S,U}&=\sum_{T:\;T \subseteq U} (-1)^{|T|}\cZ_i^{[(S \setminus U) \cup T]}\\
&=\sum_{T:\;T \subseteq U} (-1)^{|T|}\sum_k^{((iS \setminus U) \cup T)}
a_{ik}(\coord_k^\top R^{((iS \setminus U) \cup T)} \bone )\\
&=\sum_{k \in U} a_{ik} \left(
\sum_{T:\;T \subseteq U \setminus \{k\}}
(-1)^{|T|}(\coord_k^\top R^{((iS \setminus U) \cup T)} \bone )\right)
+\sum_k^{(iS)} a_{ik}\left(
\sum_{T:\;T \subseteq U}(-1)^{|T|}(\coord_k^\top R^{((iS \setminus U) \cup T)} \bone )\right)\\
&\defn\sum_{k \in U} a_{ik} \alpha_k+\sum_k^{(iS)} a_{ik}\beta_k.
\end{align*}
We claim that deterministically on the event $\Xi$, there is a constant $C>0$
such that for any $W,V \subset [n]$ disjoint with $|W \cup V| \le \log n$,
and any $i \notin W \cup V$, we have
\begin{equation}\label{eq:weakdependence}
\left|\sum_{T:\;T \subseteq W}
(-1)^{|T|} \left(\coord_i^\top R^{(V \cup T)} \bone \right) \right| \le
\tilde{y}(Cxw)^w , 
\end{equation}
where $w=|W|+1$, $x=(\log n)^{2+2\eps+a}/ \sqrt{d}$, and $\tilde{y}=C\sqrt{d}
(\log n)^{-1-\eps}$. We will verify this claim at the end of the proof.
Assuming this claim, we apply it above with $V=iS \setminus U$ and either
$W=U$ or $W=U \setminus \{k\}$. Then setting $u=|U|+1 \geq w$, we have on $\Xi$
that
\begin{equation}\label{eq:alphabetabound}
|\alpha_k| \leq \tilde{y}(Cxu)^{|U|}, \qquad
|\beta_k| \leq \tilde{y}(Cxu)^{|U|+1}.
\end{equation}
Let $r$ be any even integer with $r \leq p \leq (\log n)-1$.
As $\alpha_k,\beta_k$ are independent of row
$a_i$ of $A$ by definition, we have for the partial expectation $\E_i$ over
$a_i$ that
\begin{align*}
& \E_i \left[ \1\{\Xi\} \left|\cZ_i^{S,U} \right|^r \right] \\
&=\E_i\left[\1\{\Xi\}\left|\sum_{k \in U} a_{ik}\alpha_k
+\sum_k^{(iS)}a_{ik}\beta_k\right|^r\right]\\
&\le \1\{|\alpha_k| \le \tilde{y}(Cxu)^{|U|} \text{ and }
|\beta_k| \le \tilde{y}(Cxu)^{|U|+1} \text{ for all }
k\} \cdot \E_i\left[\left|\sum_{k \in U} a_{ik}\alpha_k
+\sum_k^{(iS)}a_{ik}\beta_k\right|^r\right].
\end{align*}
We apply \eqref{eq:concentrationmoment} for the conditional expectation
$\E_i$, with $v$ having entries
$v_k=\alpha_k$ for $k \in U$, $v_k=\beta_k$ for $k \notin iS$, and $v_k=0$
otherwise. 
Recall that $w \leq |U| \leq |S| \leq \log n$. 
Since $Cxw \ll 1$ and $|U|(Cxw)^{2|U|} \ll (n-|U|)(Cxw)^{2|U|+2}$
by the definition of $x$ and $d \leq n$, the bounds
(\ref{eq:alphabetabound}) imply
\[\|v\|_\infty \leq \tilde{y}(Cxu)^{|U|},
\qquad \|v\|_2 \leq \sqrt{2n} \cdot \tilde{y}(Cxw)^{|U|+1}.\]
Then for a constant $C'>0$, (\ref{eq:concentrationmoment}) gives
\[\E_i \left[ \1\{\Xi\} \left|\cZ_i^{S,U} \right|^r \right]
\le (C'r\tilde{y}(Cxu)^u)^r.\]
Then taking the full expectation and
setting $y=C'(\log n)\tilde{y} \geq C'r\tilde{y}$ (since $r \leq p \leq \log n$) yields condition (ii).
%\nb{ZF: Fixed an error in the above argument; the bound for
%$|\alpha_k|$ only has exponent $|U|$, not $|U|+1$.}

For condition (iii), we have
\begin{align*}
\E \left[\1\{\Xi_i\} \left|\cZ_i^{S,U} \right|^2 \right] &\le
2^{|U|}\sum_{T:\;T \subseteq U} \E[\1\{\Xi_i\}
|\cZ_i^{[(S \setminus U) \cup T]}|^2]\\
&=2^{|U|}\sum_{T:\;T \subseteq U}
\sum_{k,k'}^{((iS \setminus U) \cup T)}
\E[a_{ik}a_{ik'}]\E\left[\1\{\Xi_i\}(\coord_k^\top R^{((iS \setminus U) \cup T)}\bone)
(\coord_{k'}^\top R^{((iS \setminus U) \cup T)} \bone)\right]\\
&=2^{|U|}\sum_{T:\;T \subseteq U} \sum_k^{((iS \setminus U) \cup T)}
\E[a_{ik}^2]\E\left[\1\{\Xi_i\}\left|\coord_k^\top R^{((iS \setminus U) \cup T)} \bone
\right|^2 \right],
\end{align*}
where the second line applies the independence of $a_i$ and $A^{(i)}$.
%\nbr{JX. I could not fully understand why there is the factor $n$ in the second inequality. 
%It seems that the off-diagonal terms vanish because $\expect{a_{ik}a_{ik'}}=0$
%for $k \neq k'$.}\nb{ZF: OK, removed $n$ above and made this an equality.}
Note that on $\Xi_i$, we have $\|A^{(i)}\| \le 2.5$. Then
applying $|U| \le \log n$,
the norm bound $\|R^{((iS \setminus U) \cup T)}\| \le (\log n)^a$ on
$\Xi_i$,
and $\E[a_{ik}^2] \le C^2/n$, we get (iii). 
For (iv), we apply the
condition $|a_{ik}| \le 1$ by definition of $\Xi$, together with the bound
$\|R^{(iU)}\| \leq (\log n)^a$ on $\Xi$. Finally, (v) holds by the probability
bound of $1-e^{-c(\log n)^{1+\eps}}$
established for \eqref{eq:wigneroffdiaglaw}, \eqref{eq:wignerdiaglaw},
\eqref{eq:wignerrowsumlaw}, \eqref{eq:abound}, and in
Lemma \ref{lmm:normbound}.

It remains to establish the claim \eqref{eq:weakdependence}. For $W=\emptyset$,
this follows from (\ref{eq:wignerrowsumlaw}). Assume then that $w \geq 1$,
and write $W=\{j_1,\ldots,j_{w-1}\}$ (in any order).
For a function $f:\R^{n \times n} \to \C$ and any index $j \in [n]$,
define $Q_jf:\R^{n \times n} \to \C$ by
\[(Q_j f)(A)=f(A)-f(A^{(j)}).\]
Note that if $f$ is in fact a function of $A^{(S)}$, i.e.\ $f(A)=f(A^{(S)})$ for
every matrix $A \in \R^{n \times n}$,
then $Q_jf(A)=f(A^{(S)})-f(A^{(jS)})$. Fix $i$ and $V$, and
define $f(A)=\coord_i^\top R^{(V)} \bone $. This satisfies
$f(A)=f(A^{(V)})$ for every $A$. Then by inclusion-exclusion,
the quantity to be bounded is equivalently written as
\[\sum_{T:\;T \subseteq W} (-1)^{|T|}
(\coord_i^\top R^{(V \cup T)} \bone )=(Q_{j_{w-1}}\ldots Q_{j_2}Q_{j_1}f)(A).\]

We apply Schur complement identities
to iteratively to expand $Q_{j_{w-1}}\ldots Q_{j_1}f$:
First applying \eqref{eq:eR1}, we get
\[Q_{j_1}f(A)=\coord_i^\top R^{(V)} \bone -\coord_i^\top R^{(j_1V)} \bone 
=R_{ij_1}^{(V)} \cdot \frac{1}{R_{j_1j_1}^{(V)}} \cdot \coord_{j_1}^\top R^{(V)} \bone .\]
Then applying \eqref{eq:eR1}, \eqref{eq:Rkkinv}, and \eqref{eq:LOO}
to the three factors on the right side above, and using the identity
\[xyz-\tilde{x}\tilde{y}\tilde{z}=xy(z-\tilde{z})
+x(\tilde{y}-y)\tilde{z}+(\tilde{x}-x)\tilde{y}\tilde{z},\]
we get
\begin{align*}
Q_{j_2}Q_{j_1}f(A)&=R_{ij_1}^{(V)} \cdot \frac{1}{R_{j_1j_1}^{(V)}} \cdot 
\left(\frac{R_{j_1j_2}^{(V)}}{R_{j_2j_2}^{(V)}} \cdot
\coord_{j_2}^\top R^{(V)} \bone \right)
+R_{ij_1}^{(V)} \cdot \left(-\frac{ \left(R_{j_1j_2}^{(V)} \right)^2}{R_{j_1j_1}^{(j_2V)}
R_{j_2j_2}^{(V)}R_{j_1j_1}^{(V)}}\right) \cdot \coord_{j_1}^\top R^{(j_2V)} \bone \\
&\hspace{1in}+\frac{R_{ij_2}^{(V)}R_{j_2j_1}^{(V)}}{R_{j_2j_2}^{(V)}}
\cdot \frac{1}{R_{j_1j_1}^{(j_2V)}} \cdot \coord_{j_1}^\top R^{(j_2V)} \bone .
\end{align*}
Applying \eqref{eq:LOO}, \eqref{eq:Rkkinv}, and \eqref{eq:eR1} to
each factor of each summand
above, and repeating iteratively, an induction argument verifies the following
claims for each $t \in \{1,\ldots,w-1\}$:
\begin{itemize}
\item $Q_{j_t}\ldots Q_{j_1}f(A)$ is a sum of at most
$\prod_{s=1}^{t-1} 4s$ summands (with the convention
$\prod_{s=1}^0 4s=1$), where
\item Each summand is a product of at most $4t$ factors, where
\item jach factor is one of the following three forms, for a set
$S \subseteq V \cup W$: $R^{(S)}_{jk}$ for $j,k \notin S$ distinct,
or $1/R^{(S)}_{jj}$ for $j \notin S$, or $\coord_j^\top R^{(S)} \bone $ for $j \notin S$.
Furthermore,
\item Each summand of $Q_{j_t}\ldots Q_{j_1}f(A)$ satisfies:
(a) It has exactly one factor of the form $\coord_j^\top R^{(S)} \bone $. (b) The number of
factors of the form $1/R_{jj}^{(S)}$ is less than or equal to the number of
factors of the form $R_{jk}^{(S)}$ for $j \ne k$. (c) There are at least
$t$ factors of the form $R^{(S)}_{jk}$ for $j \ne k$.
\end{itemize}
Finally, we apply this with $t=w-1$ and use the bound
\[\prod_{s=1}^{t-1} 4s \le (4w)^w.\]
By Lemma \ref{lem:locallawRS}, since $|W \cup V| \le \log n$, we have
$|R_{jk}^{(S)}| \le C(\log n)^{2+2\eps+a}/ \sqrt{d} $,
$|R_{jj}^{(S)}| \ge |m_0|/2$,
and $|\coord_j^\top R^{(S)} \bone | \le C(\log n)^{1+\eps+a}$
on the event $\Xi$. Thus we get
\[|Q_{j_{w-1}}\ldots Q_{j_1}f(A)| \le (4w)^w \cdot
\left(\frac{C(\log n)^{2+2\eps+a}}{ \sqrt{d} }\right)^{w-1}
\cdot C(\log n)^{1+\eps+a} \le \tilde{y}(C'xw)^w\]
for $x=(\log n)^{2+2\eps+a}/ \sqrt{d}$
and $\tilde{y}=C\sqrt{d}(\log n)^{-1-\eps}$, as claimed.
\end{proof}

\medskip
We now show \eqref{eq:wignertotalsumlaw} holds for $z \in D$
with probability $1-e^{-c(\log n)(\log \log n)}$.
The diagonal bound \eqref{eq:wignerdiaglaw} implies
\begin{equation}\label{eq:Rdiagbound}
|\Tr R-n \cdot m_0| \le \frac{Cn(\log n)^{2+2\eps+3a/2}}{ \sqrt{d} }.
\end{equation}
To bound the sum of off-diagonal elements of $R$,
we apply \eqref{eq:Rjk} to write
\begin{equation}
\sum_{i \ne k} R_{ik}=-\sum_i R_{ii}\cZ_i
=-m_0\sum_i \cZ_i-\sum_i (R_{ii}-m_0)\cZ_i.
\label{eq:offRsum}
\end{equation}
Applying \eqref{eq:wignerdiaglaw} and \eqref{eq:Zboundfinal} yields
\begin{equation}\label{eq:Roffdiag1}
\sum_i |(R_{ii}-m_0)\cZ_i| \le \frac{Cn(\log n)^{3+3\eps+5a/2}}{ \sqrt{d} }.
\end{equation}
Then applying Lemma \ref{lem:fluctuationavg} with $x,y,\Xi$ as
defined in Lemma \ref{lem:fluctuationavgconditions} and with $p$ being the
largest even integer less than $(\log n)-1$, we have
\begin{equation}\label{eq:pointwisebound}
\1\{\Xi\}
\left|n^{-1}\sum_i \cZ_i\right| \le C(\log n)^{12} \cdot  \sqrt{d} (\log
n)^{-\eps} \cdot (\log n)^{4+4\eps+2a}/d \le \frac{C(\log n)^{16+3\eps+2a}}
{ \sqrt{d} }
\end{equation}
with probability $1-e^{-c(\log n)(\log \log n)}$. 
Since $\bone^\top R \bone = \Tr R + \sum_{i \ne k} R_{ik}$, 
multiplying \prettyref{eq:pointwisebound} by $n \cdot m_0$ and combining with \prettyref{eq:Rdiagbound}--\prettyref{eq:Roffdiag1} yields \eqref{eq:wignertotalsumlaw}.

\bibliographystyle{alpha}
\bibliography{matching}

\end{document}